\documentclass[reqno,11pt,a4paper,final]{amsart}
\usepackage[a4paper,left=30mm,right=30mm,top=30mm,bottom=30mm,marginpar=20mm]{geometry} 
\usepackage{amsmath}
\usepackage{amssymb}
\usepackage{amsthm}
\usepackage{amscd}
\usepackage{stmaryrd}
\usepackage{esint}
\usepackage[ansinew]{inputenc}
\usepackage{cite}
\usepackage{bbm}
\usepackage{xcolor}
\usepackage[english=american]{csquotes}
\usepackage[final]{graphicx}
\usepackage{hyperref}
\usepackage{calc}
\usepackage{newtxtext} % \usepackage{newtxtext,newtxmath}
\usepackage{bm}
\usepackage{enumerate}
\usepackage[shortlabels]{enumitem}
\usepackage{transparent}
\usepackage{xr}

\numberwithin{equation}{section}

\newtheoremstyle{thmlemcorr}{10pt}{10pt}{\itshape}{}{\bfseries}{.}{10pt}{{\thmname{#1}\thmnumber{ #2}\thmnote{ (#3)}}}
\newtheoremstyle{thmlemcorr*}{10pt}{10pt}{\itshape}{}{\bfseries}{.}\newline{{\thmname{#1}\thmnumber{ #2}\thmnote{ (#3)}}}
\newtheoremstyle{remexample}{10pt}{10pt}{}{}{\bfseries}{.}{10pt}{{\thmname{#1}\thmnumber{ #2}\thmnote{ (#3)}}}
\theoremstyle{thmlemcorr}
\newtheorem{theorem}{Theorem}
\numberwithin{theorem}{section}
\newtheorem{lemma}[theorem]{Lemma}

\newtheorem{proposition}[theorem]{Proposition}

\newtheorem{definition}[theorem]{Definition}

\theoremstyle{thmlemcorr*}
\newtheorem{theorem*}{Theorem}
\newtheorem{lemma*}[theorem]{Lemma}
\newtheorem{corollary*}[theorem]{Corollary}
\newtheorem{proposition*}[theorem]{Proposition}
\newtheorem{problem*}[theorem]{Problem}
\newtheorem{conjecture*}[theorem]{Conjecture}
\newtheorem{definition*}[theorem]{Definition}

\theoremstyle{remexample}
\newtheorem{remark}[theorem]{Remark}
\newtheorem{example}[theorem]{Example}

\newcommand{\Crm}{\mathrm{C}}

\newcommand{\Irm}{\mathrm{I}}

\newcommand{\Lrm}{\mathrm{L}}

\newcommand{\Trm}{\mathrm{T}}

\newcommand{\Wrm}{\mathrm{W}}

\newcommand{\Bcal}{\mathcal{B}}

\newcommand{\Dcal}{\mathcal{D}}
\newcommand{\Ecal}{\mathcal{E}}

\newcommand{\Hcal}{\mathcal{H}}

\newcommand{\Lcal}{\mathcal{L}}
\newcommand{\Mcal}{\mathcal{M}}

\newcommand{\Pcal}{\mathcal{P}}

\newcommand{\Wcal}{\mathcal{W}}

\newcommand{\Ycal}{\mathcal{Y}}
\newcommand{\Zcal}{\mathcal{Z}}

\newcommand{\Fbf}{\mathbf{F}}

\newcommand{\Mbf}{\mathbf{M}}

\DeclareMathOperator*{\esssup}{ess\,sup}

\DeclareMathOperator{\Id}{Id}

\DeclareMathOperator{\im}{im}

\DeclareMathOperator{\graph}{graph}

\DeclareMathOperator*{\wslim}{w*-lim}

\DeclareMathOperator{\supmod}{sup}

\DeclareMathOperator{\dist}{dist}

\DeclareMathOperator{\tr}{tr}
\DeclareMathOperator{\spn}{span}

\DeclareMathOperator{\supp}{supp}
\DeclareMathOperator{\cof}{cof}
\DeclareMathOperator{\diag}{diag}

\DeclareMathOperator{\proj}{proj}

\DeclareMathOperator{\Wedge}{{\textstyle\bigwedge}}

\newcommand{\ee}{\mathrm{e}}

\newcommand{\setn}[2]{\{\, #1 \ \ \textup{\textbf{:}}\ \ #2 \,\}}
\newcommand{\setb}[2]{\bigl\{\, #1 \ \ \textup{\textbf{:}}\ \ #2 \,\bigr\}}
\newcommand{\setB}[2]{\Bigl\{\, #1 \ \ \textup{\textbf{:}}\ \ #2 \,\Bigr\}}
\newcommand{\setBB}[2]{\biggl\{\, #1 \ \ \textup{\textbf{:}}\ \ #2 \,\biggr\}}

\newcommand{\norm}[1]{\|#1\|}

\newcommand{\normb}[1]{\bigl\|#1\bigr\|}

\newcommand{\abs}[1]{|#1|}

\newcommand{\absb}[1]{\bigl|#1\bigr|}

\newcommand{\absBB}[1]{\biggl|#1\biggr|}

\newcommand{\altnorm}[1]{{\left\vert\kern-0.25ex\left\vert\kern-0.25ex\left\vert #1 \right\vert\kern-0.25ex\right\vert\kern-0.25ex\right\vert}}

\newcommand{\spr}[1]{( #1 )}

\newcommand{\dpr}[1]{\langle #1 \rangle}

\newcommand{\dprb}[1]{\bigl\langle #1 \bigr\rangle}

\newcommand{\dbr}[1]{\llbracket #1 \rrbracket}

\newcommand{\cl}[1]{\overline{#1}}
\newcommand{\di}{\mathrm{d}}
\newcommand{\dd}{\;\mathrm{d}}
\newcommand{\DD}{\mathrm{D}}
\newcommand{\N}{\mathbb{N}}
\newcommand{\R}{\mathbb{R}}

\newcommand{\loc}{\mathrm{loc}}

\newcommand{\toto}{\rightrightarrows}
\newcommand{\toweak}{\rightharpoonup}
\newcommand{\toweakstar}{\overset{*}\rightharpoonup}

\newcommand{\toup}{\uparrow}
\newcommand{\todown}{\downarrow}

\newcommand{\embed}{\hookrightarrow}
\newcommand{\cembed}{\overset{c}{\embed}}
\newcommand{\conv}{*}
\newcommand{\hodge}{\star}
\newcommand{\BigO}{\mathrm{\textup{O}}}
\newcommand{\SmallO}{\mathrm{\textup{o}}}
\newcommand{\sbullet}{\begin{picture}(1,1)(-0.5,-2.5)\circle*{2}\end{picture}}
\newcommand{\frarg}{\,\sbullet\,}
\newcommand{\BV}{\mathrm{BV}}

\newcommand{\eps}{\epsilon}

\newcommand{\tv}[1]{\norm{#1}}

\newcommand{\term}[1]{\textbf{#1}}

\newcommand{\proofstep}[1]{\medskip\textit{#1}}

\newcommand{\SO}{\mathrm{SO}}

\newcommand{\Diss}{\mathrm{Diss}}

\newcommand{\Disl}{\mathrm{Disl}}
\newcommand{\Slip}{\mathrm{Slip}}
\newcommand{\Td}{\bm{T}}
\newcommand{\Sd}{\bm{S}}
\newcommand{\ff}{\gg}
\newcommand{\tbf}{\mathbf{t}}
\newcommand{\pbf}{\mathbf{p}}

\newcommand{\Lip}{\mathrm{Lip}}

\DeclareMathOperator{\Tan}{T}

\DeclareMathOperator{\Var}{Var}

\DeclareMathOperator{\Argmin}{Argmin}

\newcounter{assumption}
\makeatletter
\newcommand{\nextas}[1]{%
   \refstepcounter{assumption}%
   \protected@write \@auxout{}{\string\newlabel{#1}{{(A\theassumption)}{\thepage}{(A\theassumption)}{#1}{}}}%
   \hypertarget{#1}{(A\theassumption)}%
}
\newcommand{\nextasnamed}[2]{%
   \refstepcounter{assumption}%
   \protected@write \@auxout{}{\string\newlabel{#1}{{(#2)}{\thepage}{(#2)}{#1}{}}}%
   \hypertarget{#1}{(#2)}%
}
\makeatother

\def\XXint#1#2#3{{\setbox0=\hbox{$#1{#2#3}{\int}$} 
\vcenter{\hbox{$#2#3$}}\kern-.5\wd0}}

\usepackage{pict2e}
\makeatletter
\DeclareRobustCommand{\intprod}{%
  \mathbin{\mathpalette\int@prod{(0.1,0)(0.9,0)(0.9,0.8)}}}
\DeclareRobustCommand{\restrict}{%
  \mathbin{\mathpalette\int@prod{(0.1,0.8)(0.1,0)(0.9,0)}}}	
\newcommand{\int@prod}[2]{%
  \begingroup
  \sbox\z@{$\m@th#1+$}%
  \setlength\unitlength{\wd\z@}%
  \begin{picture}(1,1)
  \roundcap
  \polyline#2
  \end{picture}%
  \endgroup
}
\makeatother

\renewcommand{\eps}{\varepsilon}
\renewcommand{\epsilon}{\varepsilon}
\renewcommand{\phi}{\varphi}
\renewcommand{\tilde}{\widetilde}
\renewcommand{\hat}{\widehat}
\renewcommand{\bar}{\overline}

\begin{document}

%% TITLE MATTERS

\title[Elasto-plastic evolutions driven by discrete dislocation flow]{Energetic solutions to rate-independent large-strain\\elasto-plastic evolutions driven by discrete dislocation flow}

%\date{\today{}.}

\author{Filip Rindler}
\address{Mathematics Institute, University of Warwick, Coventry CV4 7AL, United Kingdom.}
\email{F.Rindler@warwick.ac.uk}

%% PDF MATTERS

% \hypersetup{
%   pdfauthor = {},
%   pdftitle = {},
%   pdfsubject = {},
%   pdfkeywords = {}
% }

%% START OF CONTENT

\maketitle

%\hrule\vspace{1pt}
%\begin{center}
%\textbf{\large
%DRAFT -- Version of \today}
%\end{center}
%\hrule
%\vspace{10mm}

\begin{abstract}
This work rigorously implements a recent model of large-strain elasto-plastic evolution in single crystals where the plastic flow is driven by the movement of discrete dislocation lines. The model is geometrically and elastically nonlinear, that is, the total deformation gradient splits multiplicatively into elastic and plastic parts, and the elastic energy density is polyconvex. There are two internal variables: The system of all dislocations is modeled via $1$-dimensional boundaryless integral currents, whereas the history of plastic flow is encoded in a plastic distortion matrix-field. As our main result we construct an energetic solution in the case of a rate-independent flow rule. Besides the classical stability and energy balance conditions, our notion of solution also accounts for the movement of dislocations and the resulting plastic flow. Because of the path-dependence of plastic flow, a central role is played by so-called ``slip trajectories'', that is, the surfaces traced out by moving dislocations, which we represent as integral $2$-currents in space-time. The proof of our main existence result further crucially rests on careful a-priori estimates via a nonlinear Gronwall-type lemma and a rescaling of time. In particular, we have to account for the fact that the plastic flow may cause the coercivity of the elastic energy functional to decay along the evolution, and hence the solution may blow up in finite time.
 
\vspace{4pt}

%\noindent\textsc{MSC (2010): ...} 

%\noindent\textsc{Keywords:} Elasto-plasticity, dislocations, rate-independent evolution, energetic solution, nonlinear elasticity, large-strain regime.

%\vspace{4pt}

\noindent\textsc{Date:} \today{}.
\end{abstract}

%\setcounter{tocdepth}{1} 
%\tableofcontents

%\newpage

\section{Introduction}

Dislocation flow is the principal mechanism behind macroscopic plastic deformation in crystalline materials such as metals~\cite{AbbaschianReedHill09,HullBacon11book,AndersonHirthLothe17book}. The mathematical theories of large-strain elasto-plasticity and of crystal dislocations have seen much progress recently. Notably, a number of works have investigated phenomenological models of large-strain elasto-plasticity~\cite{Mielke03a,MainikMielke05,FrancfortMielke06,MielkeMuller06,Stefanelli08,MainikMielke09,MielkeRossiSavare18} by utilizing so-called \enquote{internal variables}. This area has a long tradition and we refer to~\cite{LemaitreChaboche90,Lubliner08,AbbaschianReedHill09,GurtinFriedAnand10book,HanReddy13,Temam18book} for recent expositions and many historical references. However, the internal variables are usually conceived in a somewhat ad hoc manner (e.g., total plastic strain) and do not reflect the microscopic physics, at least not directly.

In parallel, the theory of dislocations has developed rapidly over the last years, but usually macroscopic plastic effects are neglected in this area. On the static (non-evolutionary) side we mention~\cite{ContiTheil05,GarroniLeoniPonsiglione10,ContiGarroniMassaccesi15,ContiGarroniOrtiz15,KupfermanMaor15,KupfermanOlami20, ArizaContiGarroniOrtiz18,Ginster19,EpsteinKupfermanMaor20} for some recent contributions.  On the evolutionary side, the field of discrete dislocation dynamics (DDD) considers discrete systems of dislocations moving in a crystal; {see~\cite{BulatovCai06book,BlassFonsecaLeoniMorandotti15,BlassMorandotti17,ScalaVanGoethem19,DondlKurzkeWojtowytsch19,FonsecaGinsterWojtowytsch21} for recent works in this direction. In the case of fields of dislocation we also mention the field dislocation mechanics of Acharya and collaborators~\cite{Acharya01,Acharya03,AroraAcharya20,Acharya21}.}

The recent article~\cite{HudsonRindler22} introduced a model of large-strain elasto-plastic evolution in single crystals with the pivotal feature that the plastic flow is driven directly by the movement of dislocations. In the case of a rate-independent flow rule, the present work places this model on a rigorous mathematical foundation, defines a precise notion of (energetic) solution, and establishes an existence theorem (Theorem~\ref{thm:main}) for such evolutions under physically meaningful assumptions. Such a theorem may in particular be considered a validation of the model's mathematical structure.

In the following we briefly outline the model from~\cite{HudsonRindler22}, our approach to making the notions in it precise, and some aspects of the strategy to prove the existence of solutions.

\subsection*{Kinematics}

The reference (initial) configuration of a material specimen is denoted by $\Omega \subset \R^3$, which is assumed to be a bounded Lipschitz domain (open, connected, and with Lipschitz boundary). It is modelled as a macroscopic continuum with total deformation $y \colon [0,T] \times \Omega \to \R^3$, for which we require the orientation-preserving condition $\det \nabla y(t) > 0$ pointwise in $\Omega$ (almost everywhere) for any time $t \in [0,T]$. We work in the large-strain, geometrically nonlinear regime, where the deformation gradient splits according to the \emph{Kr\"{o}ner decomposition}~\cite{Kroner60,LeeLiu67FSEP,Lee69EPDF,GreenNaghdi71,CaseyNaghdi80,GurtinFriedAnand10book,ReinaConti14,KupfermanMaor15,ReinaDjodomOrtizConti18,EpsteinKupfermanMaor20}
\[
  \nabla y = EP
\]
into an \emph{elastic distortion} $E \colon [0,T] \times \Omega \to \R^{3 \times 3}$ and a \emph{plastic distortion} $P \colon [0,T] \times \Omega \to \R^{3 \times 3}$ (with $\det E, \det P > 0$ pointwise a.e.\ in $\Omega$). We refer in particular to the justification of this relation in~\cite{HudsonRindler22}, which is based on a description of the crystal lattice via the \enquote{scaffold} $Q = P^{-1}$. However, neither $E$ nor $P$ can be assumed to be a gradient itself and $P$ is treated as an internal variable, that is, $P$ is carried along the plastic flow.

In line with much of the literature, we impose the condition of \emph{plastic incompressibility}
\[
  \det P(t) = 1  \quad\text{a.e.\ in $\Omega$,}
\]
that is, the plastic distortion $P(t) = P(t,\frarg)$ is volume-preserving, which is realistic for many practically relevant materials~\cite{AbbaschianReedHill09,GurtinFriedAnand10book}.

\subsection*{Dislocations and slips}

As mentioned before, in crystalline materials the dominant source of plasticity is the movement of dislocations, that is, $1$-dimensional topological defects in the crystal lattice~\cite{AbbaschianReedHill09,HullBacon11book,AndersonHirthLothe17book}. Every dislocation has associated with it a (constant) \emph{Burgers vector} from a finite set $\Bcal = \{ \pm b_1, \ldots, \pm b_m \} \subset \R^3 \setminus \{0\}$, which is determined by the crystal structure of the material. We collect all dislocation lines with Burgers vector $b \in \Bcal$ that are contained in our specimen at time $t$, in a $1$-dimensional integral current $T^b(t)$ on $\cl{\Omega}$ (see~\cite{ContiGarroniMassaccesi15,ContiGarroniOrtiz15,ScalaVanGoethem19} for similar ideas and~\cite{KrantzParks08book,Federer69book} as well as Section~\ref{sc:curr} for the theory of integral currents). This current is boundaryless, i.e.,
\[
  \partial T^b(t) = 0
\]
since dislocation lines are always closed loops inside the specimen $\Omega$; for technical reasons we assume that all dislocation lines are in fact closed globally (one may need to add \enquote{virtual} lines on the surface $\partial \Omega$ to close the dislocations; also see Remark~\ref{rem:Wc_restrict}). 

When considering the evolution of $t \mapsto T^b(t)$, several issues need to be addressed: First, in order to rigorously define the dissipation, that is, the energetic cost to move the dislocations from $T^b(s)$ to $T^b(t)$ ($s < t$), we need a notion of \enquote{traversed area} between $T^b(s)$ and $T^b(t)$. Indeed, in a rate-independent model, where only the trajectory, but not the speed of movement, matters, this area, weighted in a manner depending on the state of the crystal lattice, corresponds to the dissipated energy.

Second, only evolutions $t \mapsto T^b(t)$ that can be understood as \enquote{deformations} of the involved dislocations should be admissible. In particular, jumps are not permitted (at least not without giving an explicit jump path).

Third, on the technical side, we need a theory for evolutions of integral currents $t \mapsto T^b(t)$ based on their \emph{variation} in time. For instance, we require a form of the Helly selection principle to pick subsequences of sequences $(t \mapsto T^b_n(t))_n$ for which $T^b_n(t)$ converges for every $t \in [0,T]$. 

It is a pivotal idea of the present work that all of the above requirements can be fulfilled by considering as fundamental objects not the dislocations $T^b(t)$ themselves, but the associated \emph{slip trajectories}, which contain the whole evolution of the dislocations in time. We represent a slip trajectory as a $2$-dimensional integral current $S^b$ (for the Burgers vector $b \in \Bcal$) in the space-time cylinder $[0,T] \times \R^3$ with the property that
\[
  \partial S^b \restrict ((0,T) \times \R^3) = 0.
\]
Moreover, since one may flip the sign of a Burgers vector when at the same time also reversing all dislocation line directions, the symmetry relation $S^{-b} = -S^b$ needs to hold for the family $(S^b)_{b \in \Bcal}$.

In this description, the dislocation system at time $t$ is given by
\[
  T^b(t) := \pbf_*(S^b|_t),
\]
i.e., the pushforward under the spatial projection $\pbf(t,x) := x$ of the slice $S^b|_t$ of $S^b$ at time $t$ (that is, with respect to the temporal projection $\tbf(t,x) := t$). The theory of integral currents entails that $T^b(t)$ is a $1$-dimensional integral current and $\partial T^b(t) = 0$ for almost every $t \in (0,T)$.

The total traversed \emph{slip surface} from $T^b(s)$ to $T^b(t)$ can be seen to be the integral $2$-current in $\R^3$ given by
\[
  S^b|_s^t := \pbf_* \bigl[ S^b \restrict ([s,t] \times \R^3) \bigr],
\]
that is, the pushforward under the spatial projection of the restriction of $S^b$ to the time interval $[s,t]$. Note, however, that $S^b|_s^t$ does not contain a \enquote{time index}, which is needed to describe the plastic flow (see below), and also that multiply traversed areas may lead to cancellations in $S^b|_s^t$. This will require us to define the dissipation as a function of {the slip trajectories and not of the slip surfaces.}

\subsection*{Plastic flow}

With a family $(S^b)_b$ of slip trajectories at hand, we can proceed to specify the resulting plastic effect. To give the discrete dislocations a non-infinitesimal size we convolve $S^b$ with the \emph{dislocation line profile} $\eta \in \Crm^\infty_c(\R^3;[0,\infty))$, to obtain the \emph{thickened slip trajectory} $S^b_\eta := \eta \conv S^b$ (with \enquote{$\conv$} the convolution in space). This expresses the \enquote{macroscopic} shape of the dislocation orthogonal to the line direction, which in single crystals is not infinitesimal.

For kinematic reasons detailed in~\cite{HudsonRindler22}, the plastic distortion $P$ follows the \emph{plastic flow equation}, which describes the effect of the moving dislocations on the plastic distortion:
\begin{equation} \label{eq:plastflow}
   \frac{\di}{\di t} P(t,x) = D(t,x,P(t,x){;(S^b)_b}) := \frac12 \sum_{b \in \Bcal} b \otimes \proj_{\langle P(t,x)^{-1}b \rangle^\perp} \bigl[ \hodge \gamma^b(t,x) \bigr].
\end{equation}
Here, the spatial $2$-vector $\gamma^b(t,x) \in \Wedge_2 \R^3$ is the density of the measure
\[
  \pbf(S^b_\eta) := \pbf(\vec{S}^b_\eta) \, \tv{S^b_\eta}
\]
at $(t,x)$, which takes the role of the \emph{geometric slip rate}, and \enquote{$\hodge$} denotes the Hodge star operation, so that $\hodge \gamma^b(t,x)$ is the normal to the (thickened) slip surface at $(t,x)$. The factor $\frac12$ is explained by the fact that every dislocation with Burgers vector $b \in \Bcal$ is also a dislocation with Burgers vector $-b$ (with the opposite orientation).

Note that the projection in the definition of $D$ has the effect of disregarding dislocation \emph{climb}, so that $P$ represent the history of dislocation \emph{glide} only (see Section~6.2 in~\cite{HudsonRindler22} for more on this). It turns out that for technical reasons we cannot enforce that $\hodge \gamma^b$ is orthogonal to $P^{-1} b$ for admissible slip trajectories (which would obviate the need for the projection in~\eqref{eq:plastflow}); see Remark~\ref{rem:D_proj} for an explanation.

\subsection*{Energy functionals}

For the elastic energy we use
\[
  \Wcal_e(y,P) := \int_\Omega W_e(\nabla y P^{-1}) \dd x
\]
and make the hyperelasticity assumption that
\[
  \text{$y(t)$ is a minimizer of $\Wcal_e(\frarg,P(t))$ for all $t \in [0,T]$.}
\]
This is justified on physical grounds by the fact that elastic movements are usually much faster than plastic movements~\cite{DeHossonyRoosMetselaar02,ArmstrongArnoldZerilli09,BenDavidEtAl14}. For the elastic energy density $W_e$ we require polyconvexity~\cite{Ball76,Ball02} as well as (mild) growth and continuity conditions. In particular, our assumptions will be satisfied for the prototypical elastic energy densities of the form $W_e(E) := \tilde{W}(E) + \Gamma(\det E)$, {which only depend on the elastic part $E = \nabla y P^{-1}$ in the Kr\"{o}ner decomposition.} Here, $\tilde{W} \colon \R^{3 \times 3} \to [0,\infty)$ is convex {or polyconvex}, has $r$-growth, and is $r$-coercive with a sufficiently large $r > 3$ (depending on the other exponents in the full setup). Moreover, $\Gamma \colon \R \to [0,+\infty]$ is assumed to be continuous, convex, and $\Gamma(s) = +\infty$ if and only if $s \leq 0$; see Example~\ref{ex:We} for details. {In applications, one usually also requires \emph{frame-indifference} of $W_e$, that is $W_e(QE) = W_e(E)$ for all $Q \in \SO(3), E \in \R^{3 \times 3}$. This is satisfied for instance if $\tilde{W}_e(E) = \abs{E}^r$ with $\abs{\frarg}$ the Frobenius norm, yielding a superlinear-growth compressible neo-Hookean material~\cite{Ciarlet88book}.}

Further, we introduce the \emph{core energy} as
\[
  \Wcal_c((T^b)_b) := \frac{\zeta}{2} \sum_{b \in \Bcal} \Mbf(T^b),
\]
where $\zeta > 0$. Here, $\Mbf(T^b) = \tv{T^b}(\R^3)$ is the \emph{mass} of the current $T^b$, i.e., the total length of all lines contained in $T^b$. This core energy represents an \emph{atomistic} potential energy \enquote{trapped} in the dislocations~\cite{AbbaschianReedHill09,HullBacon11book,AndersonHirthLothe17book} (also see Section~6.4 in~\cite{HudsonRindler22}). The present work could be extended to also incorporate more complicated (e.g., anisotropic) core energies, but we refrain from doing so for expository reasons.

Given further an \emph{external loading} $f \colon [0,T] \times \Omega \to \R^3$, the \emph{total energy} is then
\[
  \Ecal(t,y,P,(T^b)_b) := \Wcal_e(y,P) - \int_\Omega f(t,x) \cdot y(x) \dd x + \Wcal_c((T^b)_b).
\]
{We only consider bulk loadings in this work, but this is not an essential restriction; see Remark~\ref{rem:loadings} for possible extensions.}

{It is interesting to note that we do not need to employ a hardening term \emph{in the energy functional $\Ecal$} that gives coercivity in $P$ or $\nabla P$, like in previous works on (phenomenological) elasto-plastic evolution in the large-strain regime, see, e.g.,~\cite{Mielke03a,MainikMielke05,FrancfortMielke06,MielkeMuller06,Stefanelli08,MainikMielke09,MielkeRossiSavare18}. Instead, we will impose a coercivity assumption on the \emph{dissipation} with respect to the \emph{variation} of the dislocation motion (see below). Thus, we do not penalize large amounts of movement via the modulus of $P$ (which may go up or down, e.g., in a periodic motion), but via the total amount of dislocation movement (which can only increase along the evolution). Since in our model the evolution of $P$ occurs \emph{only} via dislocation slip, the $\Wrm^{1,q}$-variation in time of $P$ remains bounded as long as the dissipation remains bounded; see Lemma~\ref{lem:PSigma_W1q} for the precise statement. Hence, no hardening term in $\Ecal$ is necessary. We finally remark that also the curl of $P$, and in fact any derivative of any order of $P$, remain likewise controlled since the effect of dislocation movement is assumed to be macroscopic (via the smooth dislocation line profile $\eta$).}

\subsection*{Dissipation}

A key role in the formulation of the dynamics is played by the \emph{dissipation}, i.e., the energetic cost associated with a slip trajectory $S^b$ moving the dislocations from $T^b(s)$ to $T^b(t)$, where $s < t$ and the Burgers vector $b \in \Bcal$ is fixed for the moment. In first approximation, this dissipation is given by the \emph{(space-time) variation} of $S^b$, which is defined as
\[
  \Var(S^b;[s,t]) := \int_{[s,t] \times \R^3} \abs{\pbf(\vec{S}^b)} \dd \tv{S^b}.
\]
Here, $S^b = \vec{S}^b \tv{S^b}$ is the Radon--Nikodym decomposition of the integral current $S^b$ into its \emph{orienting $2$-vector} $\vec{S}^b \in \Lrm^\infty(\tv{S^b};\Wedge_2 \R^{1+3})$ (which is simple and has unit length) and the \emph{total variation measure} $\tv{S^b} \in \Mcal^+([0,T] \times \R^3)$. We refer to Section~\ref{sc:notation} for details on these notions. The quantity $\Var(S^b;[s,t])$ expresses precisely the area traversed by the moving dislocation with absolute multiplicity, that is, areas traversed several times are also counted several times. From a physical perspective, the (space-time) variation counts roughly the number of bonds that are cut when the dislocation moves, in line with micromechanical principles~\cite{HullBacon11book,AndersonHirthLothe17book}.

However, the space-time variation does not account for the progressive lattice distortion of the deforming crystal and the resulting change to the number of bonds per (referential) traversed area. In the model introduced in~\cite{HudsonRindler22} (see, in particular, Sections~4.3, 4.5, but using the multi-vector formulation detailed in the appendix to~\cite{HudsonRindler22}), the dissipation along a slip trajectory $S^b$ from $s$ to $t$ is therefore given instead as  
\begin{equation} \label{eq:Diss_intro}
  \int_{[s,t] \times \R^3} R^b\bigl( P \pbf(\vec{S}^b) \bigr) \dd \tv{S^b},
\end{equation}
where the function $R^b \colon \Wedge_2\R^3 \to [0,\infty)$ is the convex and $1$-homogeneous \emph{dissipation potential}, expressing the dissipational cost of a unit slip surface, which may be anisotropic and $b$-dependent. We require $R^b$ to satisfy the bounds
\[
  C^{-1}\abs{\xi} \leq R^b(\xi) \leq C\abs{\xi},
\]
with a constant $C > 0$ that is independent of $b$. We remark that the \enquote{pre-multiplication} of $\pbf(\vec{S}^b)$ with $P$ actually means the pushforward under $P$, i.e., $P(v \wedge w) = (Pv) \wedge (Pw)$ for simple $2$-vectors $v \wedge w$, and for non-simple $2$-vectors extended by linearity. It is precisely this pre-multiplication with $P$ that accommodates the additional anisotropy introduced by the plastic distortion.

{The precise form of the \emph{total dissipation} we employ, denoted by $\Diss((S^b)_b;[s,t])$, can be found in Section~\ref{sc:assumpt}. It is a bit more involved than~\eqref{eq:Diss_intro} due to the further mathematical necessity to require a form of coercivity of the dissipation with respect to the variation, which is independent of the magnitude of $P$. Such a coercivity could be interpreted as a form of hardening (on the level of the dissipation) since it is precisely this coercivity that obviates the need for the usual hardening terms in the total energy; see Example~\ref{ex:diss_hardening} and Remark~\ref{rem:additive_hardening} for further explanation. Without this coercivity the specimen could rip immediately, preventing the existence of solutions for any non-trivial time interval.

Since our dissipation then controls all the (space-time) variations $\Var(S^b;\frarg)$ up to constants, one is naturally lead to a theory of integral currents with bounded (space-time) variation, which was developed in~\cite{Rindler23}. The required aspects of this theory are recalled in Section~\ref{sc:BVcurr} as the basis upon which our rigorous modeling of dislocations and slip trajectories in Section~\ref{sc:disl_slips} is built.}

\subsection*{Energetic solutions}

In~\cite{HudsonRindler22}, the relation linking plastic distortion rates (velocities) and the corresponding stresses is given by the \emph{flow rule} (in its multi-vector version)
\begin{equation} \label{eq:flowrule}
  P^{-T} X^b \in \partial R^b(P\gamma^b),
\end{equation}
where $\gamma^b$ is the geometric slip rate (see~\eqref{eq:plastflow}), $P\gamma^b$ is the pushforward of the $2$-vector $\gamma^b$ under $P$, and $R^b$ is the dissipation potential (see~\eqref{eq:Diss_intro}). Moreover, $X^b$ denotes the \emph{configurational stress}, that is, the stress associated with changes of dislocation configuration, which is thermodynamically conjugate to $\gamma^b$. In a smooth setting and neglecting the core energy, it can be expressed as
\[
  X^b = \hodge b^T M P^{-T},
\]
where $M$ is the \emph{Mandel stress} (structural plastic stress),
\[
  M := P^{-T} \nabla y^T \DD W_e(\nabla y P^{-1}).
\]
{While it is often possible to make rigorous sense of the Mandel stress by imposing a \enquote{multiplicative stress control} as in~\cite{FrancfortMielke06,MainikMielke09,MielkeRossiSavare18} (see, e.g., (3.W$_3$) and Lemma~4.6 in~\cite{MielkeRossiSavare18}), the differentiability of integral currents \enquote{along the flow}, and hence the definition of $X^b$, turns out to be a delicate matter, which is explored in detail in~\cite{BonicattoDelNinRindler22?}.

To avoid these issues, we formulate our whole system in a completely \emph{derivative-free} setup, where $X^b$ and $M$ do not appear. For this we employ an energetic framework based on the Mielke--Theil theory of rate-independent systems introduced in~\cite{MielkeTheil99,MielkeTheilLevitas02,MielkeTheil04}; see~\cite{MielkeRoubicek15book} for a comprehensive monograph, which also contains many more references. The basic idea is to replace the flow rule by a \emph{(global) stability relation} and an \emph{energy balance}, which employ only the total energy and dissipation functionals.}

However, our framework differs from the classical energetic theory, as presented in the monograph~\cite{MielkeRoubicek15book}, in a number of significant ways. Most notably, the central idea of the energetic theory to use a dissipation distance between any two states of the system~\cite{FrancfortMielke06,MainikMielke09,MielkeRossiSavare18,ScalaVanGoethem19} is modified here. This is a consequence of the fact that in order to define the change in plastic distortion associated with the movement of a dislocation we do not merely need the endpoints, but the whole trajectory. We will associate two \enquote{forward operators} to a slip trajectory, which determine the endpoint of the evolution for the dislocations and for the plastic distortion, respectively. The definition of the dislocation forward operator is straightforward (see Section~\ref{sc:slips}), but for the plastic forward operator some effort needs to be invested (see Section~\ref{sc:plast_evol}). Further, we need to avoid the formation of jumps in the evolution since, for the reasons discussed above, one cannot define the plastic distortion associated with these jumps. As rate-independent evolutions can develop jumps naturally, we need to introduce a rescaling of time to keep the jump paths resolved.

{The precise definition of our notion of solution is given in Definition~\ref{def:sol}, after all the aforementioned objects have been rigorously defined and the precise mathematical assumptions have been stated. Our main existence result is Theorem~\ref{thm:main}. Roughly, it states that under suitable assumptions and given initial values $y_0$, $P_0$, $\Td_0 := (T^b_0)_b$ for the total deformation, plastic distortion, and dislocation system, respectively (satisfying suitable compatibility conditions), there are total deformation, plastic distortion and slip trajectory processes
\[
  y(t),  \qquad P(t),  \qquad \Sd = (S^b)_b,
\]
respectively, from which we also define the dislocation system at time $t$ via
\[
  T^b(t) := \pbf_*(S^b|_t),
\]
such that $y, P, \Sd$ start at the prescribed initial values and satisfy the following conditions in a non-trivial time interval $[0,T_*)$:
\[
\left\{
\begin{aligned}
  &\text{\term{(S) Stability:} If $t \in [0,T_*)$ is not a jump point then for all $\hat{y}$ and $\hat{\Sd}$:} \\
  &\quad\qquad  \Ecal \bigl( t,y(t),P(t),(T^b(t))_b \bigr) \leq \Ecal \bigl( t,\hat{y},\hat{\Sd}_\ff P(t), (\hat{\Sd}_\ff T^b(t))_b \bigr) + \Diss(\hat{\Sd}), \\
  &\text{\phantom{(S)}where for a test slip trajectory $\hat{\Sd}$ the effect of the evolution by $\hat{\Sd}$ is expressed via}\\[-4pt]
  &\text{\phantom{(S)}the forward operator $\hat{\Sd}_\ff$ (see Section~\ref{sc:disl_slips}) and the resulting dissipation is denoted}\\[-4pt]
  &\text{\phantom{(S)}by $\Diss(\hat{\Sd})$.}\\[8pt]
  &\text{\term{(E) Energy balance:}} \\
  &\quad\qquad \hspace{-2mm} \Ecal \big( t,y(t),P(t),(T^b(t))_b \big) = \Ecal(0,y_0,z_0) - \Diss(\Sd;[0,t]) - \int_0^t \dprb{\dot{f}(\tau),y(\tau)} \dd \tau, \\
  &\text{\phantom{(S)}where $\Diss(\Sd;[0,t])$ is the dissipation of the dislocation movement up to time $t$.}\\[8pt]
  &\text{\term{(P) Plastic flow:}} \\
  &\quad\qquad  \frac{\di}{\di t} P(t,x) = D(t,x,P(t,x){;\Sd})  \qquad\text{and}\qquad  \text{$\det P(t) = 1$ in $\Omega$,} \\[4pt]
  &\text{\phantom{(P)}where $D$ is the is the plastic drift defined in~\eqref{eq:plastflow}.}
\end{aligned}
\right.
\]
Moreover, bounds of bounded variation (BV) type in time hold on $P(t)$ and $\Sd$, but not in general on $y(t)$.}

In line with the general energetic theory of rate-independent systems, see~\cite{MielkeRoubicek15book}, no uniqueness of solutions can be expected. Moreover, since our system includes nonlinear elasticity, also the non-uniqueness inherent in that theory is contained in our model, see, e.g.,~\cite{Ciarlet88book} for examples.

We will construct solutions as limits of a time-stepping scheme, where we minimize over \enquote{elementary} slip trajectories at every step. While we employ a number of ideas of the classical energetic theory, we will give a complete and essentially self-contained proof.

\subsection*{Decay of coercivity}

An important argument in the limit passage, as the step size tends to zero in the time-stepping scheme, is to establish sufficient a-priori estimates on the total energy. This is, however, complicated by the fact that the integrand of $\Wcal_e$ depends on $\nabla y P^{-1}$ and hence the coercivity of $\Wcal_e$ in $\nabla y$ may decay as $P$ evolves. As a consequence, we can only obtain a differential estimate of the form
\[
  \frac{\di}{\di t} \alpha^N(t) \leq C \ee^{\alpha^N(t)},
\]
where $\alpha^N$ is the energy plus dissipation of the $N$'th approximate solution. The above differential inequality (or, more precisely, the associated difference inequality) does not fall into the situation covered by the classical Gronwall lemma and finite-time blowup to $+\infty$ is possible as $N \to \infty$. Indeed, the ODE $\dot{u} = C\ee^u$, $u(0) = u_0$ has the solution $u(t) = -\log(\ee^{-u_0}-Ct)$, which blows up for $t \to \ee^{-u_0}/C$. However, using a nonlinear Gronwall-type lemma (see Lemma~\ref{lem:Gronwall_ext}), we can indeed show an $N$-independent interval of boundedness for all the $\alpha^N$. Physically, if the time interval of existence is bounded, then the material fails (e.g., rips) in finite time. 

\subsection*{Other notions of solution}

Let us finally remark that our variational framework describes the transport of dislocations in an implicit fashion, that is, we treat the slip trajectory as fundamental and recover the dislocations at a given time via slicing. It is also possible to directly consider the transport of integral currents by a vector field, see~\cite{BonicattoDelNinRindler22?}, but coupling this with elasto-plasticity is out of reach at the moment.

Let us also mention the general theory of balanced viscosity solutions developed by Mielke--Rossi--Savar\'{e}~\cite{MielkeRossiSavare09,MielkeRossiSavare12,MielkeRossiSavare16,MielkeRossiSavare18} (see also~\cite{KneesRossiZanini19} for a recent application to damage and~\cite{DalMasoDeSimoneSolombrino10,DalMasoDeSimoneSolombrino11} for other related results about viscoplastic relaxation), which could enable a finer study of the jump behavior (see also~\cite{RindlerSchwarzacherVelazquez21} for a related approach).

\subsection*{Outline of the paper}

We begin by recalling notation, basic facts, and the theory of space-time integral currents of bounded variation in Section~\ref{sc:notation}. In Section~\ref{sc:disl_slips} we define rigorously the basic kinematic and dynamic objects of our theory, namely dislocation systems, slip trajectories, and the forward operators. The following Section~\ref{sc:evolution} details our assumptions on the energy and dissipation functionals, defines our notion of solutions, and states the main existence result, Theorem~\ref{thm:main}. The time-incremental approximation scheme to construct a solution is introduced in Section~\ref{sc:approx}. Finally, Section~\ref{sc:existence_proof} is devoted to the limit passage and the proof of the existence theorem.

\subsection*{Acknowledgments}

The author would like to thank Amit Acharya, Paolo Bonicatto, Kaushik Bhattacharya, Giacomo Del~Nin, Gilles Francfort, Thomas Hudson, Andrea Marchese, and Alexander Mielke for discussions related to this work and the referees for their very helpful suggestions, which led to many improvements. This project has received funding from the European Research Council (ERC) under the European Union's Horizon 2020 research and innovation programme, grant agreement No 757254 (SINGULARITY).

\section{Notation and preliminaries} \label{sc:notation}

This section recalls some notation and results, in particular from geometric measure theory.

\subsection{Linear and multilinear algebra}

The space of $(m \times n)$-matrices $\R^{m \times n}$ is equipped with the Frobenius inner product $A : B := \sum_{ij} A^i_j B^i_j = \tr(A^TB) = \tr(B^TA)$ (upper indices indicate rows and lower indices indicate columns) as well as the Frobenius norm $\abs{A} := (A : A)^{1/2} = (\tr(A^T A))^{1/2}$.

The $k$-vectors in an $n$-dimensional real Hilbert space $V$ are contained in $\Wedge_k V$ and the $k$-covectors in $\Wedge^k V$, $k = 0,1,2,\ldots$. For a simple $k$-vector $\xi = v_1 \wedge \cdots \wedge v_k$ and a simple $k$-covector $\alpha = w^1 \wedge \cdots \wedge w^k$ the duality pairing is given as $\dpr{\xi,\alpha} = \det \, (v_i \cdot w^j)^i_j$; this is then extended to non-simple $k$-vectors and $k$-covectors by linearity. The inner product and restriction of $\eta \in \Wedge_k V$ and $\alpha \in \Wedge^l V$ are $\eta \intprod \alpha \in \Wedge^{l-k} V$ and $\eta \restrict \alpha \in \Wedge_{k-l} V$, respectively, which are defined as
\begin{align*}
  \dprb{\xi, \eta \intprod \alpha} := \dprb{\xi \wedge \eta, \alpha},  \qquad \xi \in \Wedge_{l-k} V,\\
  \dprb{\eta \restrict \alpha, \beta} := \dprb{\eta, \alpha \wedge \beta},  \qquad \beta \in \Wedge^{k-l} V.
\end{align*}
We will exclusively use the mass and comass norms of $\eta \in \Wedge_k V$ and $\alpha \in \Wedge^k V$, given via
\begin{align*}
  \abs{\eta} &:= \sup \setb{ \absb{\dprb{\eta,\alpha}} }{ \alpha \in \Wedge^k V,\;\abs{\alpha} = 1 },  \\
  \abs{\alpha} &:= \sup \setb{\absb{\dprb{\eta,\alpha}} } { \eta \in \Wedge_k V \text{ simple, unit} },
\end{align*}
where we call a simple $k$-vector $\eta = v_1 \wedge \cdots \wedge v_k$ a unit if the $v_i$ can be chosen to form an orthonormal system in $V$.

For a $k$-vector $\eta \in \Wedge_k V$ in an $n$-dimensional Hilbert space $V$ with inner product $\spr{\frarg,\frarg}$ and fixed ambient orthonormal basis $\{\ee_1,\ldots,\ee_n\}$, we define the Hodge dual $\hodge \eta \in \Wedge_{n-k} V$ as the unique vector satisfying
\[
  \xi \wedge \hodge \eta = \spr{\xi,\eta} \, \ee_1 \wedge \cdots \wedge \ee_n,  \qquad \xi \in \Wedge_k V.
\]
In the special case $n=3$ we have the following geometric interpretation of the Hodge star: $\hodge \eta$ is the normal vector to any two-dimensional hyperplane with orientation $\eta$. In fact, for $a,b \in \Wedge_1 \R^3$ the identities
\[
  \hodge(a \times b) = a \wedge b,  \qquad
  \hodge (a \wedge b) = a \times b
\]
hold, where \enquote{$\times$} denotes the classical vector product. Indeed, for any $v \in \R^3$, the triple product $v \cdot (a \times b)$ is equal to the determinant $\det(v,a,b)$ of the matrix with columns $v,a,b$, and so
\[
  v \wedge \hodge(a \times b)
  = v \cdot (a \times b) \, \ee_1 \wedge \ee_2 \wedge \ee_3
  = \det(v,a,b) \, \ee_1 \wedge \ee_2 \wedge \ee_3
  = v \wedge (a \wedge b).
\]
Hence, the first identity follows. The second identity follows by applying $\hodge$ on both sides and using $\hodge^{-1} = \hodge$ (since $n = 3$).

A linear map $S \colon V \to W$, where $V,W$ are real vector spaces, extends (uniquely) to a linear map $S \colon \Wedge^k V \to \Wedge^k W$ via
\[
  S(v_1 \wedge \cdots \wedge v_k) := (Sv_1) \wedge \cdots \wedge (Sv_k),  \qquad v_1, \ldots, v_k \in V,
\]
and extending by (multi-)linearity to $\Wedge^k V$.

\subsection{Spaces of Banach-space valued functions} \label{sc:BVX}

Let $w \colon [0,T] \to X$ ($T > 0$) be a process (i.e., a function of \enquote{time}) that is measurable in the sense of Bochner, where $X$ is a reflexive and separable Banach space; see, e.g.,~\cite[Appendix~B.5]{MielkeRoubicek15book} for this and the following notions. We define the corresponding \term{$X$-variation} for $[\sigma,\tau] \subset [0,T]$ as
\[
  \Var_X(w;[\sigma,\tau]) := \sup \setBB{ \sum_{\ell = 1}^N \norm{w(t_\ell)-w(t_{\ell-1})}_X }{ \sigma = t_0 < t_1 < \cdots t_N = \tau },
\]
where $\sigma = t_0 < t_1 < \cdots t_N = \tau$ is any partition of $[\sigma,\tau]$ ($N \in \N$). Let
\[
  \BV([0,T];X) := \setb{ w \colon [0,T] \to X }{ \Var_X(w;[0,T]) < \infty }.
\]
Its elements are called \term{($X$-valued) functions of bounded variation}. We further denote the space of Lipschitz continuous functions with values in a Banach space $X$ by $\Lip([0,T];X)$. Note that we do not identify $X$-valued processes that are equal almost everywhere (with respect to \enquote{time}).

By repeated application of the triangle inequality we obtain the Poincar\'{e}-type inequality
\[
  \norm{w(\tau)}_X \leq \norm{w(\sigma)}_X + \Var_X(w;[\sigma,\tau]).
\]

The following result is contained in the discussion in Section~3.1 of~\cite{Rindler23}:

\begin{lemma} \label{lem:BVX_prop}
Let $w \in \BV([0,T];X)$. Then, for every $t \in [0,T]$, the left and right limits
\[
  w(t\pm) := \lim_{s \to t\pm} w(s)
\]
exist (only the left limit at $0$ and only the right limit at $T$). For all but at most countably many \term{jump points} $t \in (0,T)$, it also holds that $w(t+) = w(t-) =: w(t)$.
\end{lemma}

The main compactness result in $\BV([0,T];X)$ is Helly's selection principle, for which a proof can be found, e.g., in~\cite[Theorem~B.5.13]{MielkeRoubicek15book}:

\begin{proposition} \label{prop:BVX_Helly}
Assume that the sequence $(w_n) \subset \BV([0,T];X)$ satisfies
\[
  \sup_n \bigl( \norm{w_n(0)}_X + \Var_X(w_n;[0,T]) \bigr) < \infty.
\]
Then, there exists $w \in \BV([0,T];X)$ and a (not relabelled) subsequence of the $n$'s such that $w_n \toweakstar w$ in $\BV([0,T];X)$, that is,
\[
  w_n(t) \toweak w(t)  \qquad\text{for all $t \in [0,T]$.}
\]
Moreover,
\[
  \Var_X(w;[0,T]) \leq \liminf_{n \to \infty} \, \Var_X(w_n;[0,T]).
\]
If additionally $(w_n) \subset \Lip([0,T];X)$ with uniformly bounded Lipschitz constants, then also $w \in \Lip([0,T];X)$.
\end{proposition}

\subsection{Integral currents} \label{sc:curr}

We refer to~\cite{KrantzParks08book} and~\cite{Federer69book} for the theory of currents and in the following only recall some basic facts that are needed in the sequel.

We denote by $\Hcal^k \restrict R$ the $k$-dimensional Hausdorff measure restricted to a (countably) $\Hcal^k$-rectifiable set $R$; $\Lcal^d$ is the $d$-dimensional Lebesgue measure. The Lebesgue spaces $\Lrm^p(\Omega;\R^N)$ and the Sobolev spaces $\Wrm^{k,p}(\Omega;\R^N)$ for $p \in [1,\infty]$ and $k = 1,2,\ldots$ are used with their usual meanings.

Let $\Dcal^k(U) := \Crm^\infty_c(U;\Wedge^k \R^d)$ ($k \in \N \cup \{0\}$) be the space of \term{(smooth) differential $k$-forms} with compact support in an open set $U \subset \R^d$. The dual objects to differential $k$-forms, i.e., elements of the dual space $\Dcal_k(U) := \Dcal^k(U)^*$ ($k \in \N \cup \{0\}$) are the \term{$k$-currents}. There is a natural notion of \term{boundary} for a $k$-current $T \in \Dcal_k(\R^d)$ ($k \geq 1$), namely the $(k-1)$-current $\partial T \in \Dcal_{k-1}(\R^d)$ given as
\[
  \dprb{\partial T, \omega} := \dprb{T, d\omega}, \qquad \omega \in \Dcal^{k-1}(\R^d),
\]
where \enquote{$d$} denotes the exterior differential. For a $0$-current $T$, we formally set $\partial T := 0$. 

A $\Wedge_k \R^d$-valued (local) Radon measure $T \in \Mcal_\loc(\R^d;\Wedge_k \R^d)$ is called an \term{integer-multiplicity rectifiable $k$-current} if
\[
  T = m \, \vec{T} \, \Hcal^k \restrict R,
\]
meaning that
\[
  \dprb{T,\omega} = \int_R \dprb{\vec{T}(x), \omega(x)} \, m(x) \dd \Hcal^k(x),  \qquad \omega \in \Dcal^k(\R^d),
\]
where:
\begin{enumerate}[(i)]
	\item $R \subset \R^d$ is countably $\Hcal^k$-rectifiable with $\Hcal^k(R \cap K) < \infty$ for all compact sets $K \subset \R^d$;
	\item $\vec{T} \colon R \to \Wedge_k \R^d$ is $\Hcal^k$-measurable and for $\Hcal^k$-a.e.\ $x \in R$ the $k$-vector $\vec{T}(x)$ is simple, has unit length ($\abs{\vec{T}(x)} = 1$), and lies in (the $k$-times wedge product of) the approximate tangent space $\Tan_x R$ to $R$ at $x$;
	\item $m \in \Lrm^1_\loc(\Hcal^k \restrict R;\N)$;
\end{enumerate}
The map $\vec{T}$ is called the \term{orientation map} of $T$ and $m$ is the \term{multiplicity}.

Let $T = \vec{T} \tv{T}$ be the Radon--Nikodym decomposition of $T$ with the \term{total variation measure} $\tv{T} = m \, \Hcal^k \restrict R \in \Mcal^+_\loc(\R^d)$. The \term{(global) mass} of $T$ is
\[
  \Mbf(T) := \tv{T}(\R^d) = \int_R m(x) \dd \Hcal^k(x).
\]

Let $\Omega \subset \R^d$ be a bounded Lipschitz domain, i.e., open, connected and with a (strong) Lipschitz boundary. We define the following sets of \term{integral $k$-currents} ($k \in \N \cup \{0\}$):
\begin{align*}
  \Irm_k(\R^d) &:= \setb{ \text{$T$ integer-multiplicity rectifiable $k$-current} }{ \Mbf(T) + \Mbf(\partial T) < \infty }, \\
  \Irm_k(\cl{\Omega}) &:= \setb{ T \in \Irm_k(\R^d) }{ \supp T \subset \cl{\Omega} }.
\end{align*}
The boundary rectifiability theorem, see~\cite[4.2.16]{Federer69book} or~\cite[Theorem~7.9.3]{KrantzParks08book}, entails that for $T \in \Irm_k(\R^d)$ also $\partial T \in \Irm_{k-1}(\R^d)$. 

For $T_1 = m_1 \vec{T}_1 \, \Hcal^{k_1} \restrict R_1 \in \Irm_{k_1}(\R^{d_1})$ and $T_2 = m_2 \vec{T}_2 \, \Hcal^{k_2} \restrict R_2 \in \Irm_{k_2}(\R^{d_2})$ with $R_1$ $k_1$-rectifiable (not just $\Hcal^{k_1}$-rectifiable) or $R_2$ $k_2$-rectifiable, we define the \term{product current} of $T_1,T_2$ as
\[
  T_1 \times T_2 := m_1 m_2 \, (\vec{T}_1 \wedge \vec{T}_2) \, \Hcal^{k_1+k_2} \restrict (R_1 \times R_2) \in \Irm_{k_1+k_2}(\R^{d_1+d_2}).
\]
For its boundary we have the formula
\[
  \partial(T_1 \times T_2) = \partial T_1 \times T_2 + (-1)^{k_1} T_1 \times \partial T_2.
\]

Let $\theta \colon \cl{\Omega} \to {\R^\ell}$ be smooth and let $T = m \, \vec{T} \, \Hcal^k \restrict R \in \Irm_k(\cl{\Omega})$. The \term{(geometric) pushforward} $\theta_* T$ (often also denoted by \enquote{$\theta_\# T$} in the literature) is
\[
  \dprb{\theta_* T,\omega} := \dprb{T, \theta^*\omega}, \qquad \omega \in \Dcal^k(\R^\ell),
\]
where $\theta^*\omega$ is the pullback of the $k$-form $\omega$.

We say that a sequence $(T_j) \subset \Irm_k(\R^d)$ \term{converges weakly*} to $T \in \Dcal_k(\R^d)$, in symbols \enquote{$T_j \toweakstar T$}, if
\[
  \dprb{T_j,\omega} \to \dprb{T,\omega} \qquad\text{for all $\omega \in \Dcal^k(\R^d)$.}
\]
For $T \in \Irm_k(\R^d)$, the \term{(global) Whitney flat norm} is given by
\[
  \Fbf(T) := \inf \, \setB{ \Mbf(Q) + \Mbf(R) }{ \text{$Q \in \Irm_{k+1}(\R^d)$, $R \in \Irm_k(\R^d)$ with $T = \partial Q + R$} }
\]
and one can also consider the \term{flat convergence} $\Fbf(T-T_j) \to 0$ as $j \to \infty$. Under the mass bound $\sup_{j\in\N} \, \bigl( \Mbf(T_j) + \Mbf(\partial T_j) \bigr) < \infty$, this flat convergence is equivalent to weak* convergence (see, for instance,~\cite[Theorem~8.2.1]{KrantzParks08book} for a proof). Moreover, the Federer--Fleming compactness theorem, see~\cite[4.2.17]{Federer69book} or~\cite[Theorems~7.5.2,~8.2.1]{KrantzParks08book}, shows that, under the uniform mass bound, we may select subsequences that converge weakly* or, equivalently, in the flat convergence.

The slicing theory of integral currents (see~\cite[Section~7.6]{KrantzParks08book} or~\cite[Section~4.3]{Federer69book}) entails that a given integral current $S = m \, \vec{S} \, \Hcal^{k+1} \restrict R \in \Irm_{k+1}(\R^n)$ can be sliced with respect to a Lipschitz map $f \colon \R^n \to \R$ as follows: Set $R|_t := f^{-1}(\{t\}) \cap R$. Then, $R|_t$ is (countably) $\Hcal^k$-rectifiable for almost every $t \in \R$. Moreover, for $\Hcal^{k+1}$-almost every $z \in R$, the approximate tangent spaces $\Trm_z R$ and $\Trm_z R|_t$, as well as the approximate gradient $\nabla^R f(z)$, i.e., the projection of $\nabla f(z)$ onto $\Trm_z R$, exist and
\[
  \Trm_z R = \spn \bigl\{ \Trm_z R|_t, \xi(z) \bigr\},  \qquad
  \xi(z) := \frac{\nabla^R f(z)}{\abs{\nabla^R f(z)}} \perp \Trm_z R|_t.
\]
Also, $\xi(z)$ is simple and has unit length. Set
\[
  m_+(z) := \begin{cases}
    m(z) &\text{if $\nabla^R f(z) \neq 0$,} \\
    0    &\text{otherwise,}  
  \end{cases}
  \qquad\qquad
  \xi^*(z) := \frac{D^R f(z)}{\abs{D^R f(z)}} \in \Wedge^1 \R^d,
\]
  where $D^R f(z)$ is the restriction of the differential $Df(z)$ to $\Trm_z R$, and
\[
  \vec{S}|_t(z) := \vec{S}(z) \restrict \xi^*(z) 
  \in \Wedge_k \Trm_z R|_t
  \subset \Wedge_k \Trm_z R.
\]
Then, the \term{slice}
\[
  S|_t := m_+ \, \vec{S}|_t \, \Hcal^k \restrict R|_t
\]
is an integral $k$-current, $S|_t \in \Irm_k(\R^n)$. We recall several important properties of slices: First, the coarea formula for slices,
\begin{equation} \label{eq:coarea_slice}
  \int_R g \, \abs{\nabla^R f} \dd \Hcal^{k+1} = \int \int_{R|_t} g \dd \Hcal^k \dd t,
\end{equation}
holds for all $g \colon R \to \R^N$ that are $\Hcal^{k+1}$-measurable and integrable on $R$. Second, we have the mass decomposition
\[
  \int \Mbf(S|_t) \dd t = \int_R \abs{\nabla^R f} \dd \tv{S}.
\]
Third, the cylinder formula
\[
  S|_t = \partial(S \restrict \{f<t\}) - (\partial S) \restrict \{f<t\}
\]
and, fourth, the boundary formula
\[
  \partial (S|_t) = - (\partial S)|_t
\]
hold.

\subsection{BV-theory of integral currents and deformations} \label{sc:BVcurr}

In this section we briefly review some aspects of the theory of space-time currents of bounded variation, which was developed in~\cite{Rindler23}. In the space-time vector space $\R^{1+d} \cong \R \times \R^d$ we denote the canonical unit vectors as $\ee_0, \ee_1,\ldots,\ee_d$ with $\ee_0$ the \enquote{time} unit vector. The orthogonal projections onto the \enquote{time} component and \enquote{space} component are respectively given by by $\tbf \colon \R^{1+d} \to \R$, $\tbf(t,x) := t$, and $\pbf \colon \R^{1+d} \to \R^d$, $\pbf(t,x) := x$.

The \term{variation} and \term{boundary variation} of a $(1+k)$-integral current $S \in \Irm_{1+k}([\sigma,\tau] \times \cl{\Omega})$ in the interval $I \subset [\sigma,\tau]$ are defined as
\begin{align*} 
  \Var(S;I) &:= \int_{I \times \R^d} \abs{\pbf(\vec{S})} \dd \tv{S},  \\
  \Var(\partial S;I) &:= \int_{I \times \R^d} \abs{\pbf(\overrightarrow{\partial S})} \dd \tv{\partial S}.
\end{align*}
If $[0,T] = [0,1]$, we abbreviate $\Var(S)$ and $\Var(\partial S)$ for $\Var(S;[0,1])$ and $\Var(\partial S;[0,1])$, respectively. It is immediate to see that
\[
  \Var(S;I) \leq \Mbf(S \restrict (I \times \R^d)) \leq \Mbf(S).
\]
  
For $\Lcal^1$-almost every $t \in [\sigma,\tau]$,
\[
  S(t) := \pbf_*(S|_t) \in \Irm_k(\cl{\Omega})
\]
is defined, where $S|_t \in \Irm_k([\sigma,\tau] \times \cl{\Omega})$ is the slice of $S$ with respect to time (i.e., with respect to $\tbf$). If $\tv{S}(\{t\} \times \R^d) > 0$ then $S(t)$ is not defined and we say that $S$ has a \term{jump} at $t$. In this case, the vertical piece $S \restrict (\{t\} \times \R^d)$ takes the role of a \enquote{jump transient}. This is further elucidated by the following \enquote{Pythagoras} lemma, which contains an estimate for the mass of an integral $(1+k)$-current in terms of the masses of the slices and the variation, see Lemma~3.5 in~\cite{Rindler23} for a proof.

\begin{lemma} \label{lem:mass}
Let $S = m \, \vec{S} \, \Hcal^{1+k} \restrict R \in \Irm_{1+k}([\sigma,\tau] \times \cl{\Omega})$. Then,
\begin{equation} \label{eq:decomp}
  \abs{\nabla^R \tbf}^2 + \abs{\pbf(\vec{S})}^2 = 1  \qquad \text{$\tv{S}$-a.e.}
\end{equation}
and
\begin{align*}
  \Mbf(S) &\leq \int_\sigma^\tau \Mbf(S(t)) \dd t + \Var(S;[\sigma,\tau]) \\
  &\leq \abs{\sigma-\tau} \cdot \esssup_{t \in [\sigma,\tau]} \, \Mbf(S(t)) + \Var(S;[\sigma,\tau]).
\end{align*}
\end{lemma}

The \term{integral $(1+k)$-currents with Lipschitz continuity}, or \term{Lip-integral $(1+k)$-currents} are the elements of the set
\begin{align*}
  \Irm^\Lip_{1+k}([\sigma,\tau] \times \cl{\Omega}) := \setBB{ S \in \Irm_{1+k}([\sigma,\tau] \times \cl{\Omega}) }{ &\esssup_{t \in [\sigma,\tau]} \, \bigl( \Mbf(S(t)) + \Mbf(\partial S(t)) \bigr) < \infty, \\
  & \tv{S}(\{\sigma,\tau\} \times \R^d) = 0, \\
  &t \mapsto \Var(S;[\sigma,t]) \in \Lip([\sigma,\tau]), \\
  &t \mapsto \Var(\partial S;(\sigma,t)) \in \Lip([\sigma,\tau]) }.
\end{align*}
Here, $\Lip([\sigma,\tau])$ contains all scalar Lipschitz functions on the interval $[\sigma,\tau]$.

It can be shown that for $S \in \Irm^\Lip_{1+k}([\sigma,\tau] \times \cl{\Omega})$ there exists a \term{good representative}, also denoted by $t \mapsto S(t)$, for which the $\Fbf$-Lipschitz constant
\[
  L := \sup_{s,t \in [\sigma,\tau]} \frac{\Fbf(S(s)-S(t))}{\abs{s-t}}
\]
is finite and $t \mapsto S(t)$ is continuous with respect to the weak* convergence in $\Irm_k(\cl{\Omega})$. Moreover,
\[
  \partial S \restrict (\{\sigma,\tau\} \times \R^d) = \delta_\tau \times S(\tau-) - \delta_\sigma \times S(\sigma+),
\]
and thus $S(\sigma+) := \wslim_{t \todown \sigma} S(t)$, $S(\tau-) := \wslim_{t \toup \tau} S(t)$ can be considered the left and right \term{trace values} of $S$.

It is straightforward to see that our notion of variation is indeed a generalization of the usual variation by identifying a scalar function of bounded variation $u \in \BV([0,1])$ (see~\cite{AmbrosioFuscoPallara00book}) with $S_u := \tau \, \Hcal^1 \restrict \graph(u)$, where $\graph(u) := \setn{ (t,u^\theta(t)) }{ t \in [0,1], \; \theta \in [0,1] }$ is the graph of $u$ and $\tau$ is the unit tangent to $\graph(u)$ (with $\tau \cdot \ee_0 \geq 0$). Here, $u^\theta(t) := (1-\theta)u^-(t) + \theta u^+(t)$ the affine jump between the left and right limits $u^\pm(t) = u(t\pm)$ at $t \in [0,1]$. Then, $\Var(S_u;I) = \Var(u;I) = \abs{Du}(I)$. See Example~3.1 in~\cite{Rindler23} for the details.

More relevant to the present work is the following:

\begin{example} \label{ex:LipHom}
Consider a Lipschitz homotopy $H \in \Lip([0,1] \times \cl{\Omega};\cl{\Omega})$ with $H(0,x) = x$, and $T \in \Irm_k(\cl{\Omega})$. Define $\bar{H}(t,x) := (t,H(t,x))$ and
\[
  S_H := \bar{H}_* (\dbr{(0,1)} \times T) \in \Irm^\Lip_{1+k}([0,1] \times \cl{\Omega}),
\]
where $\dbr{(0,1)}$ is the canonical current associated with the interval $(0,1)$ (with orientation $+1$ and multiplicity $1$). Then, according to the above definition,
\[
  S_H(t) = H(t,\frarg)_* T,  \qquad t \in [0,T].
\]
Thus, the $S_H$ so defined can be understood as deforming $T$ via $H$ into $H(1,\frarg)_* T$. We refer to Lemma~4.3 in~\cite{Rindler23} for estimates relating to the variation of such homotopical deformations.
\end{example}

The following lemma concerns the rescaling of space-time currents, see Lemma~3.4 in~\cite{Rindler23} for a proof. {In particular, as is a common technique for rate-independent systems, we will later use it to rescale bounded-variation processes to \enquote{steady} Lipschitz ones, see the proof of Proposition~\ref{prop:IP_solution}.}

\begin{lemma} \label{lem:rescale}
Let $S \in \Irm_{1+k}([\sigma,\tau] \times \cl{\Omega})$ and let $a \in \Lip([\sigma,\tau])$ be injective. Then,
\[
  a_* S := [(t,x) \mapsto (a(t),x)]_* S \in \Irm_{1+k}(a([\sigma,\tau]) \times \cl{\Omega})
\]
with
\[
  (a_* S)(a(t)) = S(t),  \qquad t \in [\sigma,\tau],
\]
and
\begin{align*}
  \Var(a_* S;a([\sigma,\tau])) &= \Var(S;[\sigma,\tau]), \\
  \Var(\partial(a_* S);a([\sigma,\tau])) &= \Var(\partial S;[\sigma,\tau]), \\
  \esssup_{t \in a([\sigma,\tau])} \, \Mbf((a_* S)(t)) &= \esssup_{t \in [\sigma,\tau]} \, \Mbf(S(t)), \\
  \esssup_{t \in a([\sigma,\tau])} \, \Mbf(\partial(a_* S)(t)) &= \esssup_{t \in [\sigma,\tau]} \, \Mbf(\partial S(t)).
\end{align*}
If $S \in \Irm^\Lip_{1+k}([\sigma,\tau] \times \cl{\Omega})$, then also $a_* S \in \Irm_{1+k}(a([\sigma,\tau]) \times \cl{\Omega})$.
\end{lemma}

Next, we turn to topological aspects. For this, we say that $(S_j) \subset \Irm_{1+k}([\sigma,\tau] \times \cl{\Omega})$ \term{converges BV-weakly*} to $S \in \Irm_{1+k}([\sigma,\tau] \times \cl{\Omega})$ as $j \to \infty$, in symbols \enquote{$S_j \toweakstar S$ in BV}, if
\begin{equation} \label{eq:BVcurr_w*}
  \left\{ \begin{aligned}
    S_j &\toweakstar S        &&\text{in $\Irm_{1+k}([\sigma,\tau] \times \cl{\Omega})$,} \\
    S_j(t) &\toweakstar S(t)  &&\text{in $\Irm_k(\cl{\Omega})$ for $\Lcal^1$-almost every $t \in [\sigma,\tau]$.}
  \end{aligned} \right.
\end{equation}
The following compactness theorem for this convergence in the spirit of Helly's selection principle is established as Theorem~3.7 in~\cite{Rindler23}.

\begin{proposition} \label{prop:current_Helly}
Assume that the sequence $(S_j) \subset \Irm_{1+k}([\sigma,\tau] \times \cl{\Omega})$ satisfies
\[
  \esssup_{t \in [\sigma,\tau]} \, \bigl( \Mbf(S_j(t)) + \Mbf(\partial S_j(t)) \bigr) + \Var(S_j;[\sigma,\tau]) + \Var(\partial S_j;[\sigma,\tau]) \leq C < \infty
\]
for all $j \in \N$. Then, there exists $S \in \Irm_{1+k}([\sigma,\tau] \times \cl{\Omega})$ and a (not relabelled) subsequence such that
\[
  S_j \toweakstar S  \quad\text{in BV.}
\]
Moreover,
\begin{align*}
  \esssup_{t \in [\sigma,\tau]} \, \Mbf(S(t)) &\leq \liminf_{j \to \infty} \, \; \esssup_{t \in [\sigma,\tau]} \, \Mbf(S_j(t)), \\
  \esssup_{t \in [\sigma,\tau]} \, \Mbf(\partial S(t)) &\leq \liminf_{j \to \infty} \, \; \esssup_{t \in [\sigma,\tau]} \, \Mbf(\partial S_j(t)), \\
  \Var(S;[\sigma,\tau]) &\leq \liminf_{j \to \infty} \, \Var(S_j;[\sigma,\tau]), \\
  \Var(\partial S;(\sigma,\tau)) &\leq \liminf_{j \to \infty} \, \Var(\partial S_j;(\sigma,\tau)).
\end{align*}
If additionally $(S_j) \subset \Irm^\Lip_{1+k}([\sigma,\tau] \times \cl{\Omega})$ such that the Lipschitz constants $L_j$ of the scalar maps $t \mapsto \Var(S_j;[\sigma,t]) + \Var(\partial S_j;(\sigma,t))$ are uniformly bounded, then also
\[
  S \in \Irm^\Lip_{1+k}([\sigma,\tau] \times \cl{\Omega})
\]
with Lipschitz constant bounded by $\liminf_{j\to\infty} L_j$. Moreover, in this case, $S_j(t) \toweakstar S(t)$ in $\Irm_k(\cl{\Omega})$ for \emph{every} $t \in [\sigma,\tau)$.
\end{proposition}

We can use the variation to define the \term{(Lipschitz) deformation distance} between $T_0, T_1 \in \Irm_k(\cl{\Omega})$ with $\partial T_0 = \partial T_1 = 0$:
\begin{align*}
  \dist_{\Lip,\cl{\Omega}}(T_0,T_1) := \inf \setB{ \Var(S) }{ &\text{$S \in \Irm^\Lip_{1+k}([0,1] \times \cl{\Omega})$ with $\partial S = \delta_1 \times T_1 - \delta_0 \times T_0$} }.
\end{align*}
The key result for us in this context is the following \enquote{equivalence theorem}; see Theorem~5.1 in~\cite{Rindler23} for the proof.

\begin{proposition} \label{prop:equiv}
For every $M > 0$ and $T_j,T$ ($j \in \N$) in the set
\[
  \setb{ T \in \Irm_k(\cl{\Omega}) }{ \partial T = 0, \; \Mbf(T) \leq M }
\]
the following equivalence holds (as $j \to \infty$):
\[
  \dist_{\Lip,\cl{\Omega}}(T_j,T) \to 0  \qquad\text{if and only if}\qquad   T_j \toweakstar T \quad \text{in $\Irm_k(\cl{\Omega})$}.
\]
Moreover, in this case, for all $j$ from a subsequence of the $j$'s, there are $S_j \in \Irm^\Lip_{1+k}([0,1] \times \cl{\Omega})$ with
\[
  \partial S_j = \delta_1 \times T - \delta_0 \times T_j, \qquad
  \dist_{\Lip,\cl{\Omega}}(T_j,T) \leq \Var(S_j) \to 0,
\]
and
\[
  \limsup_{j\to\infty} \, \; \esssup_{t \in [0,1]} \, \Mbf(S_j(t)) \leq C \cdot \limsup_{\ell \to \infty} \, \Mbf(T_\ell).
\]
Here, the constant $C > 0$ depends only on the dimensions and on $\Omega$.
\end{proposition}

\section{Dislocations and slips}  \label{sc:disl_slips}

This section introduces the key notions that we need in order to formulate the model from~\cite{HudsonRindler22} rigorously, most notably dislocation systems and slip trajectories. Dislocation systems are collections of dislocation lines, indexed by their (structural) Burgers vector, which is constant along a dislocation line. Slip trajectories describe the evolution of a dislocation system. Crucially, they also provide a way to obtain the evolution of the plastic distortion. To this aim we will introduce suitable \enquote{forward operators}, one for dislocation systems and one for plastic distortions.

\subsection{Dislocation systems} \label{sc:DS}

Assume that we are given a set of \term{Burgers vectors}
\[
  \Bcal = \bigl\{ \pm b_1, \ldots, \pm b_m \} \subset \R^3 \setminus \{0\}.
\]
The set of \term{(discrete) dislocation systems} is defined to be
\begin{align*}
  \Disl(\cl{\Omega}) := \setB{ {\Td} =(T^b)_{b\in\Bcal} \subset \Irm_1(\cl{\Omega}) }{ &\text{$T^{-b} = -T^b$, $\partial T^b = 0$ for all $b \in \Bcal$}},
\end{align*}
where $\Irm_1(\cl{\Omega})$ is the set of all integral $1$-currents supported in $\cl{\Omega}$ (see Section~\ref{sc:curr} for notation). We interpret this definition as follows: $T^b$ contains all dislocation lines with Burgers vector $b \in \Bcal$. The symmetry condition $T^{-b} = -T^b$ for all $b \in \Bcal$ means that the sign of a Burgers vector can be flipped when accompanied by a change of line orientation. The dislocation lines are assumed to be closed (globally). While usually one only assumes closedness inside the specimen $\Omega$, in all of the following we require global closedness, essentially for technical reasons. This can always be achieved by adding \enquote{virtual} dislocation lines on $\partial \Omega$ (also see Remark~\ref{rem:Wc_restrict}).

The \term{(joint) mass} of $\Td \in \Disl(\cl{\Omega})$ is
\[
  \Mbf(\Td) := \frac12 \sum_{b \in \Bcal} \Mbf(T^b) < \infty.
\]
The factor $\frac12$ is explained by the fact that every dislocation with Burgers vector $b \in \Bcal$ is also a dislocation with Burgers vector $-b$ (with the opposite orientation).

\subsection{Slips and dislocation forward operator} \label{sc:slips}

To describe evolutions (in time) of dislocation systems, we define the set of \term{Lipschitz slip trajectories} as
\begin{align*}
  \Lip([0,T];\Disl(\cl{\Omega})) := \setB{ {\Sd} = (S^b)_{b\in\Bcal} \subset \Irm_2([0,T] \times \cl{\Omega}) }{ &\text{$S^b \in \Irm^\Lip_{1+1}([0,T] \times \cl{\Omega})$,}\\
  &\text{$S^{-b} = -S^b$, and} \\
  &\text{$\partial S^b \restrict ((0,T) \times \R^d) = 0$} }. 
\end{align*}
Also set
\[
  T^b(t) := \pbf_*(S^b|_t),  \qquad t \in [0,T],
\]
where $S^b|_t$ is the slice of $S^b$ at time $t$ (i.e., with respect to $\tbf = t$). We then have
\[
  (T^b(t))_b \in \Disl(\cl{\Omega})  \qquad t \in (0,T).
\]
We let the \term{$\Lrm^\infty$-(mass-)norm} and the \term{(joint) variation} of $\Sd \in \Lip([0,T];\Disl(\cl{\Omega}))$ be defined for any interval $I \subset [0,T]$ as, respectively,
\begin{align*}
  \norm{\Sd}_{\Lrm^\infty(I;\Disl(\cl{\Omega}))} &:= \max_{b\in\Bcal} \; \esssup_{t \in I} \, \Mbf(S^b|_t) < \infty, \\
  \Var(\Sd;I) &:= \frac12 \sum_{b \in \Bcal} \Var(S^b;I) < \infty.
\end{align*}

In the following, we will also make frequent use of the space of \term{elementary slip trajectories} starting from $\Td = (T^b)_b \in \Disl(\cl{\Omega})$, namely
\[
  {\Slip(\Td;[0,T])} := \setB{ \Sd = (S^b)_b \in \Lip([0,T];\Disl(\cl{\Omega})) }{ \text{$\partial S^b \restrict (\{0\} \times \R^d) = - \delta_0 \times T^b$} }.
\]
The idea here is that an elementary slip trajectory $\Sd \in \Slip(\Td)$ gives us a way to transform a dislocation system $\Td$ into a new dislocation system in a \emph{progressive-in-time} manner. The additional condition in the definition of $\Slip(\Td)$ entails that $S^b$ starts at $T^b$, for which we could equivalently require $S^b(0) = T^b$ for all $b \in \Bcal$.

{If $[0,T] = [0,1]$, we abbreviate
\begin{align*}
  \Slip(\Td) &:= \Slip(\Td;[0,1]), \\
  \norm{\Sd}_{\Lrm^\infty} &:= \norm{\Sd}_{\Lrm^\infty([0,1];\Disl(\cl{\Omega}))}, \\
  \Var(\Sd) &:= \Var(\Sd;[0,1]).
\end{align*}}
We may then define the \term{dislocation forward operator} for $\Td =(T^b)_b \in \Disl(\cl{\Omega})$ and $\Sd = (S^b)_b \in \Slip(\Td)$ as
\[
  \Sd_\ff \Td := (T^b_\ff)_b \in \Disl(\cl{\Omega})  \qquad\text{with}\qquad
  T^b_\ff := \pbf_* \bigl[ \partial S^b + \delta_0 \times T^b \bigr] \in \Irm_1(\cl{\Omega}),
\]
{where $\Irm_1(\cl{\Omega})$ is the set of all integral $1$-currents with support in $\cl{\Omega}$.}

\begin{example}
Let $\Td = (T^b)_b \in \Disl(\cl{\Omega})$ and let $H^b \in \Lip([0,1] \times \cl{\Omega};\cl{\Omega})$, $b \in \Bcal$, be a family of Lipschitz-homotopies satisfying
\[
  H^b(0,x) = x  \qquad\text{and}\qquad
  H^{-b} = H^b.
\]
Define $\bar{H}^b(t,x) := (t,H^b(t,x))$ and set $\Sd_H := (S^b_H)_b$ with
\[
  S^b_H := \bar{H}^b_* (\dbr{(0,1)} \times T^b) \in \Irm^\Lip_{1+1}([0,1] \times \cl{\Omega}),
\]
like in Example~\ref{ex:LipHom}. Note that $S^{-b}_H(t) = -S^b_H(t)$, $\partial S^b_H(t) = 0$, and $S^b_H(0) = T^b$ for every $b \in \Bcal$ and $\Lcal^1$-a.e.\ $t$. Thus, $\Sd_H \in \Slip(\Td)$. The $\Sd_H$ so defined deforms $\Td$ into $\Sd_\ff \Td = (H^b(1,\frarg)_* T^b)_b$.
\end{example}

\subsection{Plastic evolution} \label{sc:plast_evol}

We now consider how slip trajectories give rise to an evolution of the plastic distortion. For this, consider a dislocation system $\Td = (T^b)_b \in \Disl(\cl{\Omega})$ {(see Section~\ref{sc:DS} for the definition of this set)} and a slip trajectory $\Sd = (S^b)_b \in \Slip(\Td;[0,T])$ {(see Section~\ref{sc:slips})}. Let $\eta \in \Crm^\infty_c(\R^3;[0,\infty))$ be a \term{dislocation line profile}, which is globally fixed and determines the shape of the dislocation orthogonal to the line direction. We here allow the mass of the profile to be any number (not just $1$ like for a mollifier) to allow dislocations with a \enquote{weight}. We then define the \term{thickened slip trajectory} {$\Sd_\eta := (S^b_\eta)_b$ with
\[
  S^b_\eta := (\eta \conv S^b) \restrict \Omega \in {\Mcal([0,T] \times \cl{\Omega}; \Wedge_2 \R^3)},
\]
which is to be understood as follows:
\[
  \dprb{S^b_\eta, \omega} := \int \dprb{\vec{S}^b(t,x), [\eta \conv \omega(t,\frarg)](x)} \dd \tv{S^b}(t,x),  \qquad \omega \in \Dcal^2([0,T] \times \cl{\Omega}).
\]
Here, the convolution \enquote{$\conv$} acts in space only and $\omega$ is considered to be extended by zero outside $\cl{\Omega}$. Note that $S^b_\eta$ is no longer an \emph{integral} $2$-current. However, we only need $\Sd_\eta$ to define the plastic flow, whereas all convergence and dissipational considerations involve $\Sd$ directly, so no results from the general theory of currents~\cite{Federer69book} are required.}

\begin{lemma} \label{lem:Gb}
For all $b \in \Bcal$, the measure 
\[
  \pbf(S^b_\eta) := \pbf(\vec{S}^b_\eta) \, \tv{S^b_\eta} \in \Mcal([0,T] \times \cl{\Omega};\Wedge_2 \R^3)
\]
is absolutely continuous with respect to Lebesgue measure. For its density, called the \term{geometric slip rate}, it holds that
\begin{align*}
  \gamma^b &\in \Lrm^\infty([0,T];\Crm^\infty(\cl{\Omega};\Wedge_2 \R^3)), \\
  \gamma^b(\frarg,x) &\in \Lrm^\infty([0,T];\Wedge_2 \R^3)  \qquad\text{for almost every fixed $x \in \cl{\Omega}$,}
\end{align*}
and for all $k = 0,1,2,\ldots$ there is a constant $C_k > 0$, which only depends on $\eta$, such that
\begin{equation} \label{eq:Gb_est}
  \int_\sigma^\tau \norm{\gamma^b(t,\frarg)}_{\Crm^k} \dd t \leq C_k \cdot \Var(S^b;[\sigma,\tau])
\end{equation}
for any interval $[\sigma,\tau] \subset [0,T]$.
\end{lemma}

{For ease of notation, here and in the following we suppress the dependence of $\gamma^b$ on $\eta$ (which is considered to be globally fixed).}

\begin{proof}
Fix $b \in \Bcal$. We first observe by linearity that $\pbf(S^b_\eta) = \eta \conv \pbf(S^b)$. Then, for $\omega \in \Dcal^2([0,T] \times \cl{\Omega})$ with $\abs{\omega} \leq 1$,
\[
  \absb{\dprb{\pbf(S^b_\eta),\omega}}
  \leq \int \absb{\dprb{\pbf(\vec{S}^b(t,x)), [\eta \conv \omega(t,\frarg)](x)}} \dd \tv{S^b}(t,x)
  \leq \norm{\eta}_{\Lrm^1} \int \abs{\pbf(\vec{S}^b)} \dd \tv{S^b}
\]
since $\abs{\eta \conv \omega} \leq \norm{\eta}_{\Lrm^1}$ by the properties of the mollification. Thus, for $[\sigma,\tau] \subset [0,T]$,
\[
  \tv{\pbf(S^b_\eta)}([\sigma,\tau] \times \cl{\Omega}) \leq C \cdot \Var(S^b;[\sigma,\tau]) \leq C L \abs{\sigma-\tau},
\]
where $L > 0$ is a universal Lipschitz constant of the scalar functions $t \mapsto \Var(S^b;[0,t])$ ($b \in \Bcal$) and $C > 0$. Moreover, if
\[
  \tv{S^b}(\di t, \di x) = \int_0^T \mu_t(\di x) \dd \lambda(t)
\]
for $\lambda \in \Mcal^+([0,T])$ and $\mu_t \in \Mcal^+(\R^3)$, which is weak*-measurable as a function of $t \in [0,T]$, denotes a disintegration of $\tv{S^b}$ with respect to time (see~\cite[Section~2.5]{AmbrosioFuscoPallara00book} and also~\cite[Section~4.2]{BonicattoDelNinRindler22?}), then
\begin{align*}
  \dprb{\pbf(S^b_\eta),\omega} &= \int \dprb{\pbf(\vec{S}^b(t,x)), [\eta \conv \omega(t,\frarg)](x)} \dd \tv{S^b}(t,x) \\
  &= \int_0^T \int \dprb{\pbf(\vec{S}^b(t,x)), [\eta \conv \omega(t,\frarg)](x)} \dd \mu_t(x) \dd \lambda(t) \\
  &= \int_0^T \int \dprb{\eta \conv \bigl[ \pbf(\vec{S}^b(t,\frarg)) \, \mu_t \bigr], \omega} \dd x \dd \lambda(t).
\end{align*}
Thus, the density of $\pbf(S^b_\eta)$ has been identified as
\[
  \gamma^b(t,\frarg)
  := \eta \conv \bigl[ \pbf(\vec{S}^b(t,\frarg)) \, \mu_t \bigr].
\]
Via Young's convolution inequality it satisfies for $[\sigma,\tau] \subset [0,T]$,
\begin{align*}
  \int_\sigma^\tau \norm{\gamma^b(t,\frarg)}_\infty \dd t
  &\leq C \int_\sigma^\tau \int \abs{\pbf(\vec{S}^b(t,x))} \dd \mu_t(x) \dd \lambda(t)  \\
  &= C \int_{[\sigma,\tau] \times \cl{\Omega}} \abs{\pbf(\vec{S}^b)} \dd \tv{S^b} \\
  &= C \cdot \Var(S^b;[\sigma,\tau]) \\
  &\leq CL \abs{\sigma-\tau}.
\end{align*}
Hence, we see that $\gamma^b \in \Lrm^\infty([0,T];\Crm(\cl{\Omega};\Wedge_2 \R^3))$, $\gamma^b(\frarg,x) \in \Lrm^\infty([0,T];\Wedge_2 \R^3)$ for almost every $x$, and~\eqref{eq:Gb_est} holds for $k = 0$. The higher differentiability follows by pushing the derivatives onto the mollifier and estimating analogously.
\end{proof}

\begin{remark}
Let us remark that if we additionally knew that $\abs{\pbf(\vec{S}^b(t,x))} < 1$ for $\tv{S^b}$-almost every $(t,x)$, then an application of the coarea formula for slices, see~\eqref{eq:coarea_slice}, in conjunction with the relation $\nabla^{S^b} \tbf(t,x) \neq 0$, see~\eqref{eq:decomp} in Lemma~\ref{lem:mass}, would yield the physically easier to understand formula
\[
  \gamma^b(t,\frarg)
  = \eta \conv \biggl[ \frac{\pbf(\vec{S}^b(t,\frarg)) \tv{S^b(t)}}{\abs{\nabla^{S^b} \tbf(t,\frarg)}} \biggr],
\]
where $\nabla^{S^b} \tbf(t,\frarg)$ denotes the projection of $\nabla \tbf(t,x)$ onto the tangent space to the (rectifiable) carrier set for $S^b$. This corresponds to~(7.14) in the modelling paper~\cite{HudsonRindler22}. However, the condition $\abs{\pbf(\vec{S}^b)} < 1$ is not necessarily satisfied $\tv{S^b}$-almost everywhere for general Lipschitz-in-time currents. This is related to singular phenomena that are \enquote{smeared out in time and space}, which are discussed at great length in~\cite{BonicattoDelNinRindler22?}. There, also an explicit counterexample, the \enquote{Flat Mountain}, is presented and investigated in detail. The variational approach in the present work, however, does not depend on an explicit formula for $\gamma^b$.
\end{remark}

We also define the \term{normal slip rate}
\[
  g^b := \hodge \gamma^b \in \Lrm^\infty([0,T];\Crm^\infty(\cl{\Omega};\R^3)),
\]
where $\hodge \colon \Wedge_2 \R^3 \to \R^3$ is the Hodge star operation. By the preceding Lemma~\ref{lem:Gb} we have that $g^b(\frarg,x) \in \Lrm^\infty([0,T];\R^3)$ for almost every fixed $x \in \cl{\Omega}$, and
\[
  \int_\sigma^\tau \norm{g^b(t,\frarg)}_{\Crm^k} \dd t \leq C_k \cdot \Var(S^b;[\sigma,\tau])
\]
for all intervals $[\sigma,\tau] \subset [0,T]$ and $k = 0,1,2,\ldots$.

Let $P \in \Lrm^s(\Omega;\R^{3 \times 3})$ for an $s \in [1,\infty]$ with $\det P = 1$ a.e.\ in $\Omega$, $\Td =(T^b)_b \in \Disl(\cl{\Omega})$, and $\Sd = (S^b)_b \in \Slip(\Td;[0,T])$. Denote for almost every $x \in \Omega$ by $R_x \colon [0,T] \to \R^{3 \times 3}$ a solution of the ODE
\begin{equation} \label{eq:PSigma_ODE}
  \left\{ \begin{aligned}
    \frac{\di}{\di t} R_x(t) &= D(t,x,R_x(t){;\Sd}) \qquad \text{for a.e.\ $t \in (0,T)$,}\\
    R_x(0) &= P(x),
  \end{aligned} \right.
\end{equation}
where the \term{(total) plastic drift} $D(t,x,R;\Sd)$ for $t \in [0,T]$ and $R \in \R^{3 \times 3}$ with $\det R > 0$ is given as
\begin{equation} \label{eq:D}
  D(t,x,R{;\Sd}) := \frac12 \sum_{b \in \Bcal} b \otimes \proj_{\langle R^{-1}b \rangle^\perp}[g^b(t,x)],
\end{equation}
with $g^b$ corresponding to $\gamma^b$ for $S^b$ as above. By $\proj_{\langle R^{-1}b \rangle^\perp}$ we here denote the orthogonal projection onto the orthogonal complement to the line $\langle R^{-1}b \rangle$. We will show in Lemma~\ref{lem:PSigma_def} below that this ODE indeed has a solution for almost every $x \in \Omega$. {Note that $D(t,x,R{;\Sd})$ also implicitly depends on the dislocation line profile $\eta$; however, this is considered to be globally fixed and hence we suppress this dependence in our notation.}

We then define the \term{plastic distortion path $P_{\Sd}$} starting at $P$ induced by the slip trajectory $\Sd = (S^b)_b$ as
\begin{equation} \label{eq:PSigma_def}
  P_{\Sd}(t,x) := R_x(t),  \qquad (t,x) \in [0,T] \times \Omega,
\end{equation}
with $R_x$ the solution of~\eqref{eq:PSigma_ODE}. Moreover, if $\Sd \in \Slip(\Td)$ (i.e., $[0,T] = [0,1]$), the \term{plastic forward operator} is given via
\[
  (\Sd_\ff P)(x) := R_x(1),  \qquad x \in \Omega.  
\]
We first consider the question of well-definedness:

\begin{lemma} \label{lem:PSigma_def}
The ODE~\eqref{eq:PSigma_ODE} has a unique solution for almost every $x \in \Omega$ and $P_{\Sd}$ defined in~\eqref{eq:PSigma_def} is a Lipschitz function on $[0,T]$ with values in $\Lrm^s(\Omega;\R^{3 \times 3})$. For all $t \in [0,T]$ it holds that $\det P_{\Sd}(t) = 1$ a.e.\ in $\Omega$ as well as
\begin{equation} \label{eq:PSigma_Ls}
  \norm{P_{\Sd}(t) - P}_{\Lrm^s} \leq C \cdot \Var(\Sd;[0,t]),
\end{equation}
where $C > 0$. In particular, if $\Sd \in \Slip(\Td)$ (i.e., $[0,T] = [0,1]$), then
\[
  \Sd_\ff P \in \Lrm^s(\Omega;\R^{3 \times 3}),  \qquad
  \det(\Sd_\ff P) = 1 \text{ a.e.,}  \qquad
  \norm{\Sd_\ff P - P}_{\Lrm^s} \leq C \cdot \Var(\Sd).
\]
\end{lemma}

\begin{proof}
Fix $b \in \Bcal$ and $x \in E \subset \Omega$ with the set $E$ of those $x \in \Omega$ where $\abs{P(x)} < \infty$, $\det P(x) = 1$, and $g^b(\frarg,x)$ is uniformly bounded. By Lemma~\ref{lem:Gb}, $E$ has full measure in $\Omega$. We observe that $(t,R) \mapsto D(t,x,R;\Sd)$ is a Carath\'{e}odory map (measurable in $t$ for fixed $R$ and continuous in $R$ for fixed $t$) on the open set
\[
  U_M := \setB{ (t,R) }{ t \in (0,T), \; R \in \R^{3 \times 3}, \; \abs{R} < \abs{P(x)} + M, \; \det R > \tfrac12 }
\]
for any $M > 0$. Indeed, the plastic drift $D(\frarg,x,\frarg;\Sd)$ (with $x$ fixed) is bounded and Lipschitz on the compact set $\cl{U_M}$ since $g^b(\frarg,x)$ is uniformly bounded and Lipschitz in $U_M$, and the projection $\proj_{\langle R^{-1}b \rangle^\perp}$ depends smoothly on $R^{-1} = (\cof R)^T/\det R$, which in turn depends smoothly on $R$ in $U_M$. In fact, $R \mapsto D(t,x,R;\Sd)$ is Lipschitz in $U_M$ with a $t$-uniform Lipschitz constant.

Then, from Carath\'{e}odory's existence and uniqueness theorem for ODEs~\cite[Theorems~2.1.1, 2.2.1]{CoddingtonLevinson55book}, we obtain a unique maximal solution $R_x \colon [0,T_x) \to \R^{3 \times 3}$ of~\eqref{eq:PSigma_ODE} in $U_M$, where $T_x > 0$ is the maximum time of existence. We compute that almost everywhere in the time interval $[0,T_x)$ it holds that
\begin{align*}
  \frac{\di}{\di t} \det R_x &= \cof R_x : \dot{R}_x \\
  &= (\det R_x) R_x^{-T} : \dot{R}_x \\
  &= (\det R_x) \tr( R_x^{-1} \dot{R}_x ) \\
  &= (\det R_x) \frac12 \sum_{b \in \Bcal} \tr \Bigl( (R_x^{-1}b) \otimes \proj_{\langle R_x^{-1}b \rangle^\perp} [g^b(\frarg,x)] \Bigr) \\
  &= (\det R_x) \frac12 \sum_{b \in \Bcal} (R_x^{-1}b) \cdot \proj_{\langle R_x^{-1}b \rangle^\perp} [g^b(\frarg,x)] \\
  &= 0.
\end{align*}
Hence, as $\det P(x) = 1$,
\begin{equation} \label{eq:detRx1}
  \det R_x(t) = 1,  \qquad\text{for $t \in [0,T_x)$.}
\end{equation}

By the uniform boundedness of $D(\frarg,x,\frarg;\Sd)$, which with regard to $x$ only depends on the quantity $\norm{g^b(\frarg,x)}_{\Lrm^\infty}$, on the interval $[0,T]$ the solution $R_x$ remains bounded. Choosing $M > 0$ sufficiently large and also employing~\eqref{eq:detRx1}, we see that our solution $R_x$ is in fact defined for all $t \in [0,T]$.

To show~\eqref{eq:PSigma_Ls}, we estimate for all $t \in [0,T]$,
\begin{align*}
  \int_0^t \abs{D(\tau,x,R_x(\tau);\Sd)} \dd \tau
  &\leq C \int_0^t \norm{g^b(\tau,\frarg)}_\infty \dd \tau \\
  &\leq C \cdot \Var(\Sd;[0,t])
\end{align*}
by~\eqref{eq:Gb_est} in Lemma~\ref{lem:Gb}. Hence, 
\[
  \abs{R_x(t) - P(x)}
  \leq \int_0^t \absBB{\frac{\di}{\di \tau} R_x(\tau)} \dd \tau
  = \int_0^t \abs{D(\tau,x,R_x(\tau);\Sd)} \dd \tau
  \leq C \cdot \Var(\Sd;[0,t]).
\]
Taking the $\Lrm^s$-norm in $x$, this gives
\[
  \norm{P_{\Sd}(t) - P}_{\Lrm^s} \leq C \cdot \Var(\Sd;[0,t]).
\]
In particular, $P_{\Sd}(t) \in \Lrm^s(\Omega;\R^{3 \times 3})$ for all $t \in [0,T]$.
 
The same arguments hold also when starting the evolution at $t_0 \in [0,t)$. Thus, the Lipschitz continuity of $t \mapsto \Var(\Sd;[0,t])$ in conjunction with the additivity of the variation yield the Lipschitz continuity of $t \mapsto P_{\Sd}(t)$, considered with values in $\Lrm^s(\Omega;\R^{3 \times 3})$.

The claimed incompressibility property $\det P_{\Sd}(t,x) = 1$ for all $t \in [0,T]$ and almost every $x$ follows directly from~\eqref{eq:detRx1}.
\end{proof}

The next lemma shows the transportation of regularity along the plastic evolution.

\begin{lemma} \label{lem:PSigma_W1q} 
Assume additionally that $P \in \Wrm^{1,q}(\Omega;\R^{3 \times 3})$ for a $q \in (3,\infty]$. Then,
\begin{equation} \label{eq:PSigma_W1q}
  \norm{P_{\Sd}(t)-P}_{\Wrm^{1,q}} \leq C \cdot \Var(\Sd;[0,t]),  \qquad t \in [0,T],
\end{equation}
and $P_{\Sd}$ is a Lipschitz function with values in $\Wrm^{1,q}(\Omega;\R^{3 \times 3})$, where $C > 0$ and the Lipschitz constant depend (monotonically) on $\norm{P}_{\Wrm^{1,q}}$ and $\Var(\Sd;[0,T])$. In particular, if $\Sd \in \Slip(\Td)$ (i.e., $[0,T] = [0,1]$), then
\[
  \Sd_\ff P \in \Wrm^{1,q}(\Omega;\R^{3 \times 3}),  \qquad
  \norm{\Sd_\ff P - P}_{\Wrm^{1,q}} \leq C \cdot \Var(\Sd).
\]
\end{lemma}

Note that while the growth of $C$ in $\norm{P}_{\Wrm^{1,q}}$ and $\Var(\Sd;[0,T])$ may be very fast, we will always apply this result in the presence of a uniform bound for those quantities; then the estimates are of the same type as the ones in Lemma~\ref{lem:PSigma_def} and in particular additive in the variation.

\begin{proof}
By Lemma~\ref{lem:PSigma_def} (for $s = \infty$) and the embedding of $\Wrm^{1,q}(\Omega;\R^{3 \times 3})$ into $\Crm(\Omega;\R^{3 \times 3})$ (since $q > 3$) we have $\norm{P_{\Sd}(t)}_{\Lrm^\infty} \leq \norm{P}_{\Lrm^\infty} + C \cdot \Var(\Sd;[0,T]) =: M$. The set
\[
  X_M := \setb{ A \in \R^{3 \times 3} }{ \abs{A} \leq M, \; \det A = 1 }
\]
is compact. Hence, the map $H^b \colon X_M \to \R^{3 \times 3}$, given by
\[
  H^b(A) := \proj_{\langle A^{-1}b \rangle^\perp} = \Id - \frac{(A^{-1}b) \otimes (A^{-1}b)}{\abs{A^{-1}b}^2},  \qquad A \in X_M,
\]
is smooth and
\begin{equation} \label{eq:H_est}
  \abs{H^b(A)}, \; \abs{\DD H^b(A)} \leq C,  \qquad A \in X_M,
\end{equation}
where $C = C(M) > 0$ is a constant.

Denote by $P_{\Sd}(t,x) := R_x(t)$ the solution of~\eqref{eq:PSigma_ODE} for $x \in \Omega$. We have by the chain rule that
\begin{align*}
  \nabla D(t,x,P_{\Sd}(t,x);\Sd)
  &= \nabla \bigl[ H^b(P_{\Sd}(t,x)) g^b(t,x) \bigr] \\
  &= \DD H^b(P_{\Sd}(t,x)) \nabla P_{\Sd}(t,x) g^b(t,x) + H^b(P_{\Sd}(t,x)) \nabla g^b(t,x),
\end{align*}
where $\nabla$ denotes the (weak) $x$-gradient. Then, since time derivative and weak gradient commute, we get that $\nabla P_{\Sd}$ satisfies the ODE
\[
  \left\{ \begin{aligned}
    \frac{\di}{\di t} \nabla P_{\Sd}(t,x) &= \frac12 \sum_{b \in \Bcal} \DD H^b(P_{\Sd}(t,x)) \nabla P_{\Sd}(t,x) g^b(t,x) + H^b(P_{\Sd}(t,x)) \nabla g^b(t,x), \\
    \nabla P_{\Sd}(0,x) &= \nabla P(x).
  \end{aligned} \right.
\]
We compute
\begin{align*}
  \frac{\di}{\di t} \bigl[\nabla P_{\Sd}(t,x) - \nabla P(x) \bigr] &= \frac12 \sum_{b \in \Bcal} \biggl( \DD H^b(P_{\Sd}(t,x)) \bigl[\nabla P_{\Sd}(t,x) - \nabla P(x) \bigr] g^b(t,x) \\
  &\qquad + H^b(P_{\Sd}(t,x)) \nabla g^b(t,x) + \DD H^b(P_{\Sd}(t,x)) \nabla P(x) g^b(t,x) \biggr).
\end{align*}
Integrating in time from $0$ to $t$, taking the $\Lrm^q$-norm in $x$, and applying~\eqref{eq:H_est}, we get
\begin{align*}
  \norm{\nabla P_{\Sd}(t,\frarg) - \nabla P}_{\Lrm^q}
  &\leq C \int_0^t \norm{\nabla P_{\Sd}(\tau,\frarg) - \nabla P}_{\Lrm^q} \cdot \frac12 \sum_{b \in \Bcal} \norm{g^b(\tau,\frarg)}_\infty \\
  &\qquad\quad + \frac12 \sum_{b \in \Bcal} \norm{\nabla g^b(\tau,\frarg)}_\infty + \norm{\nabla P}_{\Lrm^q} \cdot \frac12 \sum_{b \in \Bcal} \norm{g^b(\tau,\frarg)}_\infty \dd \tau.
\end{align*}
The integral form of Gronwall's lemma now yields
\begin{align*}
  \norm{\nabla P_{\Sd}(t,\frarg) - \nabla P}_{\Lrm^q} &\leq C( 1 + \norm{\nabla P}_{\Lrm^q}) \int_0^t \frac12 \sum_{b \in \Bcal} \norm{g^b(\tau,\frarg)}_{\Crm^1} \dd \tau \\
  &\qquad \cdot \exp \biggl( C \int_0^t \frac12 \sum_{b \in \Bcal} \norm{g^b(\tau,\frarg)}_\infty \dd \tau \biggr).
\end{align*}
Combining this with~\eqref{eq:Gb_est} in Lemma~\ref{lem:Gb},
\[
  \norm{\nabla P_{\Sd}(t,\frarg) - \nabla P}_{\Lrm^q}
  \leq C(1 + \norm{\nabla P}_{\Lrm^q}) \Var(\Sd;[0,t]) \cdot \ee^{C \cdot \Var(\Sd;[0,t])}
  \leq C \cdot \Var(\Sd;[0,t]),
\]
where we have absorbed some terms into the constant $C > 0$. Together with~\eqref{eq:PSigma_Ls} this yields~\eqref{eq:PSigma_W1q}.

Further, varying the starting point and employing the Lipschitz continuity of $t \mapsto \Var(\Sd;[0,t])$ in conjunction with the additivity of the variation gives for all $s < t$ that
\[
  \norm{P_{\Sd}(s)-P_{\Sd}(t)}_{\Wrm^{1,q}} \leq C(1 + \norm{P(s)}_{\Wrm^{1,q}}) \Var(\Sd;[s,t]) \cdot \ee^{C \cdot \Var(\Sd)}
  \leq L\abs{s-t},
\]
where $L$ depends on $\Var(\Sd;[0,T])$ and $\norm{P}_{\Wrm^{1,q}}$ (which bounds $\norm{P(s)}_{\Wrm^{1,q}}$ by~\eqref{eq:PSigma_W1q}). This gives the Lipschitz continuity of $t \mapsto P_{\Sd}(t)$ with values in $\Wrm^{1,q}(\Omega;\R^{3 \times 3})$.
\end{proof}

Next, we show that we may dispense with the pointwise definition of solutions to~\eqref{eq:PSigma_ODE}.

\begin{lemma} \label{lem:PSigma_ODE_BV}
Assume that additionally $P \in \Wrm^{1,q}(\Omega;\R^{3 \times 3})$ for a $q \in (3,\infty]$. Then, the ODE~\eqref{eq:PSigma_ODE} also holds in the $\Wrm^{1,q}(\Omega;\R^{3 \times 3})$-sense, that is,
\[
  \frac{\di}{\di t} P_{\Sd}(t)  \quad \text{exists as a $\Wrm^{1,q}(\Omega;\R^{3 \times 3})$-valued map for a.e.\ $t \in [0,T]$,}
\]
and for such $t$ it holds that
\begin{equation} \label{eq:PSigma_ODE_W1q}
  \frac{\di}{\di t} P_{\Sd}(t) = x \mapsto D(t,x,P_{\Sd}(t,x);\Sd) \qquad\text{in $\Wrm^{1,q}(\Omega;\R^{3 \times 3})$.}
\end{equation}
\end{lemma}

\begin{proof}
We have seen above that $P_{\Sd} \in \Lip([0,T];\Wrm^{1,q}(\Omega;\R^{3 \times 3}))$. Then, a version of the classical Lebesgue differentiation theorem for maps with values in Banach spaces (see~\cite[Theorem~IV.3.2, p.107]{DiestelUhl77book}) in conjunction with the fact that $\Wrm^{1,q}(\Omega;\R^{3 \times 3})$ has the Radon--Nikodym property as a reflexive space (see~\cite[Corollary~III.2.13, p.76]{DiestelUhl77book}), yields the existence of $\frac{\di}{\di t} P_{\Sd}(t)$ for almost every $t \in [0,T]$. For such $t$ it holds that
\begin{align*}
  \frac{\di}{\di t} P_{\Sd}(t)
  = \lim_{\delta \to 0} \frac{P_{\Sd}(t+\delta)-P_{\Sd}(t)}{\delta}
  = \biggl( x \mapsto \frac{\di}{\di t} P_{\Sd}(t,x) \biggr),
\end{align*}
where the limit is in $\Wrm^{1,q}$ and the last equality follows via the Lipschitz continuity of $P_{\Sd}$ in time with respect to values in $\Wrm^{1,q}$ and the fact that this implies $x$-uniform pointwise Lipschitz continuity by the embedding $\Wrm^{1,q}(\Omega;\R^{3 \times 3}) \cembed \Crm(\Omega;\R^{3 \times 3})$. Thus,~\eqref{eq:PSigma_ODE_W1q} has been established.
\end{proof}

\subsection{Operations on slip systems}

We now introduce useful operations on slip trajectories, namely rescalings and concatenations, and we also define the so-called \enquote{neutral} slip trajectory.

\begin{lemma} \label{lem:Sigma_rescale}
Let $\Td \in \Disl(\cl{\Omega})$, $P \in \Lrm^s(\Omega;\R^{3 \times 3})$ for an $s \in [1,\infty]$ with $\det P = 1$ a.e.\ in $\Omega$, and $\Sd = (S^b)_b \in \Slip(\Td;[0,T])$. Let $a \colon [0,T] \to [0,T']$ be an invertible $\Crm^1$-map with $a(0) = 0$, $a(T) = T'$. Define (using the notation of Lemma~\ref{lem:rescale})
\[
  a_* \Sd := (a_* S^b)_b \in \Slip(\Td;[0,T']).
\]
Then, for the solution $P_{a_* \Sd}$ of~\eqref{eq:PSigma_ODE_W1q} the rate-independence property
\begin{equation} \label{eq:PSigma_RI}
  P_{a_* \Sd}(t') = P_{\Sd}(a^{-1}(t')),  \qquad t' \in [0,T'],
\end{equation}
holds. In particular, if $\Sd \in \Slip(\Td)$ (i.e., $[0,T] = [0,1]$) and $T' = 1$, then
\[
  (a_* \Sd)_\ff \Td = \Sd_\ff \Td,  \qquad
  (a_* \Sd)_\ff P = \Sd_\ff P.
\]
\end{lemma}

\begin{proof}
The fact that $a_* \Sd = (a_* S^b)_b \in \Slip(\Td;[0,T'])$ follows from Lemma~\ref{lem:rescale}. Turning to~\eqref{eq:PSigma_RI}, we denote by $a_* \gamma^b$ the geometric slip rate defined in Lemma~\ref{lem:Gb} with respect to $a_* S^b$. Note that for $\omega \in \Dcal^2(\R^3)$ and all $0 \leq s < t \leq T$ we obtain in the same way as in the proof of Lemma~\ref{lem:rescale} (which can be found in Lemma~3.4 of~\cite{Rindler23}) using the area formula that
\begin{align*}
  \int \int_s^t \dprb{ \gamma^b(\tau,\frarg), \omega } \dd \tau \dd x
  &= \int \int_{a(s)}^{a(t)} \dprb{ (a_* \gamma^b)(\tau',x), \omega(x) } \dd \tau' \dd x \\
  &= \int \int_s^t \dprb{ a_* \gamma^b(a(\tau),x), \omega(x) } \, \dot{a}(\tau) \dd \tau \dd x,
\end{align*}
where we changed variables in the last line. Thus,
\[
  a_* \gamma^b(t',x) = \frac{\gamma^b(a^{-1}(t'),x)}{\dot{a}(a^{-1}(t'))}
\]
{and hence
\[
  D(t',x,R; a_*\Sd) = \frac{D(a^{-1}(t'),x,R;\Sd)}{\dot{a}(a^{-1}(t'))}.
\]
For $P' := P_{\Sd} \circ a^{-1}$ we compute
\begin{align*}
  \frac{\di}{\di t'} P'(t')
  &= \biggl(\frac{\di}{\di t} P_{\Sd} \biggr)(a^{-1}(t')) \cdot \frac{1}{\dot{a}(a^{-1}(t'))} \\
  &= \frac{D(a^{-1}(t'),\frarg,P_{\Sd}(a^{-1}(t'));\Sd)}{\dot{a}(a^{-1}(t'))} \\
  &= D(t',\frarg,P'(t'); a_*\Sd).
\end{align*}}
By the uniqueness of the solution to~\eqref{eq:PSigma_ODE_W1q} we thus obtain $P' = P_{a_* \Sd}$, which implies~\eqref{eq:PSigma_RI}. The additional statements are then clear (using Lemma~\ref{lem:rescale}).
\end{proof}

\begin{lemma} \label{lem:Sigma_concat}
Let $\Td \in \Disl(\cl{\Omega})$, $\Sd^1 \in \Slip(\Td)$, $\Sd^2 \in \Slip(\Sd^1_\ff \Td)$, and $P \in \Lrm^s(\Omega;\R^{3 \times 3})$ for an $s \in [1,\infty]$ with $\det P = 1$ a.e.\ in $\Omega$. Then, there is $\Sd^2 \circ \Sd^1 \in \Slip(\Td)$, called the \term{concatenation} of $\Sd^1$ and $\Sd^2$, with
\begin{equation} \label{eq:cat_ff}
  (\Sd^2 \circ \Sd^1)_\ff \Td = \Sd^2_\ff (\Sd^1_\ff \Td),  \qquad
  (\Sd^2 \circ \Sd^1)_\ff P = \Sd^2_\ff (\Sd^1_\ff P),
\end{equation}
and
\begin{align}
  \norm{\Sd^2 \circ \Sd^1}_{\Lrm^\infty} &= \max \bigl\{ \norm{\Sd^1}_{\Lrm^\infty}, \norm{\Sd^2}_{\Lrm^\infty} \bigr\}, \label{eq:cat_Linfty} \\
  \Var(\Sd^2 \circ \Sd^1) &= \Var(\Sd^1) + \Var(\Sd^2). \label{eq:cat_Var}
\end{align}
\end{lemma}

\begin{proof}
Let $\Td = (T^b)_b$, $\Sd^1 = (S_1^b)_b$, $\Sd^2 = (S_2^b)_b$. Define $\Sd^2 \circ \Sd^1 := (V^b)_b$ with
\[
  V^b := r^{1/2}_* S_1^b + t^{1/2}_* r^{1/2}_* S_2^b,
\]
where the rescaling $r^\alpha$ and the translation $t^\tau$ ($\alpha \neq 0$, $\tau \in \R$) are given by
\[
  r^\alpha(t,x) := (\alpha t,x),  \qquad
  t^\tau(t,x) := (t+\tau,x).
\]
From Lemma~\ref{lem:rescale} we see that $\Sd^2 \circ \Sd^1 \in \Slip(\Td)$ and that~\eqref{eq:cat_Linfty},~\eqref{eq:cat_Var} hold. 

The validity of the first statement in~\eqref{eq:cat_ff} follows in a straightforward manner since, if $\Sd^1_\ff \Td = (T^b_1)_b$ and $\Sd^2_\ff(\Sd^1_\ff \Td) = (T^b_2)_b$, we have
\begin{align*}
  \partial V^b &= r^{1/2}_* \partial S_1^b + t^{1/2}_* r^{1/2}_* \partial S_2^b \\
  &= \delta_{1/2} \times T^b_1 - \delta_0 \times T^b + \delta_1 \times T^b_2 - \delta_{1/2} \times T^b_1 \\
  &= \delta_1 \times T^b_2 - \delta_0 \times T^b.
\end{align*}
The second statement in~\eqref{eq:cat_ff} is a direct consequence of~\eqref{eq:PSigma_RI} in Lemma~\ref{lem:Sigma_rescale}.
\end{proof}

\begin{lemma} \label{lem:neutral}
Let $\Td \in \Disl(\cl{\Omega})$ and $P \in \Lrm^s(\Omega;\R^{3 \times 3})$ for an $s \in [1,\infty]$ with $\det P = 1$ a.e.\ in $\Omega$. There exists a slip trajectory $\Id^{\Td} \in \Slip(\Td)$, called the \term{neutral slip trajectory}, such that
\[
  \Id^{\Td}_\ff \Td = \Td,  \qquad
  \Id^{\Td}_\ff P = P,
\]
and
\[
  \norm{\Id^{\Td}}_{\Lrm^\infty} = \Mbf(\Td), \qquad
  \Var(\Id^{\Td}) = 0.
\]
\begin{proof}
If $\Td = (T^b)_b$, set $\Id^{\Td} := (S^b)_b$ with $S^b := \dbr{(0,1)} \times T^b$.
\end{proof}
\end{lemma}

\subsection{Weak* convergence of slip trajectories}

Assume for a sequence $(\Sd_j) \subset \Lip([0,T];\Disl(\cl{\Omega}))$ with $\Sd_j = (S^b_j)_b$ and $\Sd = (S^b)_b$ that $S^b_j \toweakstar S^b$ in BV (in the sense of~\eqref{eq:BVcurr_w*}) for all $b \in \Bcal$. Then, we say that $\Sd_j$ converges \term{weakly*} to $\Sd$, in symbols \enquote{$\Sd_j \toweakstar \Sd$}. As the main compactness result we have the following:

\begin{proposition} \label{prop:LipDS_compact}
Assume that the sequence $(\Sd_j) \subset \Lip([0,T];\Disl(\cl{\Omega}))$, $\Sd_j = (S^b_j)_b$, satisfies
\[
  \supmod_j \, \bigl( \norm{\Sd_j}_{\Lrm^\infty([0,T];\Disl(\cl{\Omega}))} + \Var(\Sd_j;[0,T]) + L_j \bigr) < \infty
\]
with $L_j$ the maximum (in $b$) of the Lipschitz constants of the functions $t \mapsto \Var(S^b_j;[0,t])$. Then, there exists $\Sd \in \Lip([0,T];\Disl(\cl{\Omega}))$ and a (not relabelled) subsequence such that
\[
  \Sd_j \toweakstar \Sd.
\]
Moreover,
\begin{align*}
  \norm{\Sd}_{\Lrm^\infty([0,T];\Disl(\cl{\Omega}))} &\leq \liminf_{j \to \infty} \, \norm{\Sd_j}_{\Lrm^\infty([0,T];\Disl(\cl{\Omega}))}, \\
  \Var(\Sd;[0,T]) &\leq \liminf_{j \to \infty} \, \Var(\Sd_j;[0,T]).
\end{align*}
\end{proposition}

\begin{proof}
Let $\Sd_j = (S^b_j)_b$. We have
\[
  \Var(\partial S^b_j) \leq \Mbf(S^b_j(0+)) + \Mbf(S^b_j(T-))
  \leq 2\norm{\Sd_j}_{\Lrm^\infty([0,T];\Disl(\cl{\Omega}))}
\]
since the weak* limits $S^b_j(0+) := \wslim_{t \downarrow 0} S^b(t)$ and $S^b_j(T-) := \wslim_{t \uparrow T} S^b(t)$ exist and the mass is weakly* lower semicontinuous, see Section~\ref{sc:BVcurr}. Hence, we get that the quantities $\Var(\partial S^b_j;[0,T])$ are uniformly in $j$ bounded. The claims then follow directly from Proposition~\ref{prop:current_Helly} in conjunction with the fact that the conditions in the definition of $\Lip([0,T];\Disl(\cl{\Omega}))$ all pass to the limit. Indeed, the requirements $S_j^{-b} = -S_j^b$ and $\partial S_j^b \restrict ((0,T) \times \R^d) = 0$ for all $b \in \Bcal$ are obviously BV-weakly* continuous and $\tv{S^b}(\{0,T\} \times \R^s) = 0$ follows since the measures $\tv{S^b_j}$ are uniformly absolutely continuous by the uniform Lipschitz continuity of the $S^b_j$. For the lower semicontinuity of the variation we argue as follows: By Proposition~\ref{prop:current_Helly},
\[
  \Var(S^b;[0,T]) \leq \liminf_{n \to \infty} \, \Var(S^b_j;[0,T]).
\]
Then, Fatou's lemma implies
\begin{align*}
  \Var(\Sd;[0,T]) &= \frac12 \sum_{b \in \Bcal} \Var(S^b;[0,T]) \\
  &\leq \frac12 \sum_{b \in \Bcal} \liminf_{j \to \infty} \, \, \Var(S^b_j;[0,T]) \\
  &\leq \liminf_{j \to \infty} \frac12 \sum_{b \in \Bcal} \Var(S^b_j;[0,T]) \\
  &= \liminf_{j \to \infty} \, \Var(\Sd_j;[0,T]).
\end{align*}
The lower semicontinuity of the $\Lrm^\infty$-norm follows directly from the corresponding statement in Proposition~\ref{prop:current_Helly}.
\end{proof}

For later use we also state the compactness for elementary slips {(see Section~\ref{sc:slips})} explicitly:

\begin{proposition} \label{prop:Sl_compact}
Let $\Td \in \Disl(\cl{\Omega})$ and assume that the sequence $(\Sd_j) \subset \Slip(\Td)$, $\Sd_j = (S^b_j)_b$, satisfies
\[
  \supmod_j \, \bigl( \norm{\Sd_j}_{\Lrm^\infty} + \Var(\Sd_j) + L_j \bigr) < \infty,
\]
with $L_j$ the maximum (in $b$) of the Lipschitz constants of the functions $t \mapsto \Var(S^b_j;[0,t])$. Then, there exists $\Sd \in \Slip(\Td)$ and a (not relabelled) subsequence such that
\[
  \Sd_j \toweakstar \Sd.
\]
Moreover,
\begin{align*}
  \norm{\Sd}_{\Lrm^\infty} &\leq \liminf_{j \to \infty} \, \norm{\Sd_j}_{\Lrm^\infty}, \\
  \Var(\Sd) &\leq \liminf_{j \to \infty} \, \Var(\Sd_j).
\end{align*}
\end{proposition}

\begin{proof}
By Proposition~\ref{prop:LipDS_compact} we obtain the convergence in $\Lip([0,1];\Disl(\cl{\Omega}))$ and the lower semicontinuity assertions. From Proposition~\ref{prop:current_Helly} we further obtain that also the condition $\partial S^b_j \restrict (\{0\} \times \R^d) = - \delta_0 \times T^b$ for all $b \in \Bcal$, where we have written $\Td = (T^b)_b$, passes to the limit.
\end{proof}

Finally, we have the following continuity properties.

\begin{lemma} \label{lem:Phiff_cont}
Let $\Td \in \Disl(\cl{\Omega})$ and $\Sd_j \toweakstar \Sd$ in $\Slip(\Td)$. Then,
\[
  (\Sd_j)_\ff \Td \toweakstar \Sd_\ff \Td  \quad\text{in $\Disl(\cl{\Omega})$.}
\]
\end{lemma}

\begin{proof}
If $\Td = (T^b)_b$, $\Sd_j = (S^b_j)_b$, and $\Sd = (S^b)_b$, then $(\Sd_j)_\ff \Td = (V^b_j)_b$ with
\[
  V^b_j = \pbf_* \bigl[ \partial S^b_j + \delta_0 \times T^b \bigr]
  \toweakstar \pbf_* \bigl[ \partial S^b + \delta_0 \times T^b \bigr]
  \qquad\text{in $\Irm_1(\cl{\Omega})$}
\]
since $S^b_j \toweakstar S^b$ in $\Irm_2([0,T] \times \cl{\Omega})$. This directly implies the assertion.
\end{proof}

\begin{lemma} \label{lem:PSigma_cont}
Let $\Td \in \Disl(\cl{\Omega})$, $\Sd_j \toweakstar \Sd$ in $\Slip(\Td;[0,T])$, and $P_j \toweak P$ in $\Wrm^{1,q}(\Omega;\R^{3 \times 3})$ for a $q \in (3,\infty]$ with $\det P_j = 1$ a.e.\ in $\Omega$. Then,
\[
  (P_j)_{\Sd_j} \to P_{\Sd}  \quad\text{uniformly in $[0,T] \times \Omega$ and strongly in $\Lrm^1([0,T];\Crm(\Omega;\R^{3 \times 3}))$.}
\]
In particular, if $\Sd_j \in \Slip(\Td)$ (i.e., $[0,T] = [0,1]$), then $(\Sd_j)_\ff P_j \to \Sd_\ff P$ uniformly in $\Omega$.
\end{lemma}

\begin{proof}
It suffices to show that $(P_j)_{\Sd_j} \to P_{\Sd}$ uniformly in $[0,T] \times \Omega$. Then also the claimed convergence in $\Lrm^1([0,T];\Crm(\Omega))$ follows. If $\Sd_j \in \Slip(\Td)$, then this furthermore implies $(P_j)_{\Sd_j}(1) \to P_{\Sd}(1)$ uniformly, which is the same as $(\Sd_j)_\ff P_j \to \Sd_\ff P$ uniformly.

To see the uniform convergence in $[0,T] \times \Omega$, observe first that from Lemma~\ref{lem:PSigma_W1q} we know that the $(P_j)_{\Sd_j}$ are uniformly Lipschitz continuous in time when considered with values in $\Wrm^{1,q}$ (note that the norms $\norm{P_j}_{\Wrm^{1,q}}$ and the variations $\Var(\Sd_j;[0,T])$ are uniformly bounded by the Uniform Boundedness Principle). Hence, by the (generalized) Arzel\`{a}--Ascoli theorem we may select a subsequence of $j$'s (not specifically labeled) such that for some $P_* \in \Crm([0,T] \times \Omega;\R^{3 \times 3}) \cap \BV([0,T];\Wrm^{1,q}(\Omega;\R^{3 \times 3}))$ we have
\[
  (P_j)_{\Sd_j} \to P_*  \quad\text{uniformly in $[0,T] \times \Omega$ and weakly* in $\BV([0,T];\Wrm^{1,q}(\Omega;\R^{3 \times 3}))$.}
\]
Here we also used the compact embedding $\Wrm^{1,q}(\Omega;\R^{3 \times 3}) \cembed \Crm(\Omega;\R^{3 \times 3})$.

On the other hand, let $\gamma^b_j, \gamma^b$ and $g^b_j, g^b$ be defined as in Section~\ref{sc:slips} for the slip trajectories $\Sd_j$ and $\Sd$, respectively. Since $\Sd_j \toweakstar \Sd$ it follows that $\gamma^b_j \toweakstar \gamma^b$ in $\Lrm^\infty((0,T) \times \Omega;\Wedge_2 \R^3)$, whereby also
\[
  g^b_j \toweakstar g^b  \quad\text{in $\Lrm^\infty((0,T) \times \Omega;\R^3)$.}
\]
Rewriting the ODE~\eqref{eq:PSigma_ODE} as an integral equation and multiplying by a test function $\phi \in \Crm^\infty_c(\Omega)$, we see that $(P_j)_{\Sd_j}$ solves~\eqref{eq:PSigma_ODE} if and only if
\[
  \int_\Omega (P_j)_{\Sd_j}(t) \, \phi \dd x
  = \int_\Omega P_j \, \phi \dd x + \frac12 \sum_{b \in \Bcal} \int_\Omega \int_0^t H^b((P_j)_{\Sd_j}(\tau)) \, g^b_j(\tau) \, \phi \dd \tau \dd x,
\]
where $H^b$ is as in the proof of Lemma~\ref{lem:PSigma_W1q}. As $j \to \infty$, the above convergences in conjunction with the Lipschitz continuity of $H^b$ and the (strong $\times$ weak*)-continuity of the integral, give
\[
  \int_\Omega P_*(t) \, \phi \dd x
  = \int_\Omega P \, \phi \dd x + \frac12 \sum_{b \in \Bcal} \int_\Omega \int_0^t H^b(P_*(\tau)) \, g^b(\tau) \, \phi \dd \tau \dd x.
\]
Hence, $P_*$ solves~\eqref{eq:PSigma_ODE}. By Lemma~\ref{lem:PSigma_def}, the solution of~\eqref{eq:PSigma_ODE} for $\Sd$ is unique, whereby $P_* = P_{\Sd}$.
\end{proof}

\section{Energetic evolutions} \label{sc:evolution}

In this section we list our precise assumptions, translate the model from~\cite{HudsonRindler22} into the energetic formulation, and then state our main result, Theorem~\ref{thm:main}, which establishes the existence of an energetic solution.

\subsection{Assumptions and setup} \label{sc:assumpt}

We posit the following henceforth:

\begin{enumerate}[({A}1)] \setlength\itemsep{8pt}
  \item \label{as:general}\label{as:first} \term{Basic assumptions:}
\begin{enumerate}[(i)]
  \item $\Omega \subset \R^3$ is a bounded, connected Lipschitz domain;
  \item $\Bcal = \{ \pm b_1, \ldots, \pm b_m \} \subset \R^3 \setminus \{0\}$ is the system of Burgers vectors;
  \item $p,q \in (3,\infty)$ are the integrability exponents for the total deformation and the plastic distortion, respectively;
  \item $g \in \Wrm^{1-1/p,p}(\partial\Omega;\R^3)$ are the imposed boundary values for the total deformation;
  \item $\eta \in \Crm^\infty_c(\R^3;[0,\infty))$ is the dislocation line profile;
  \item $\zeta > 0$ is the core energy strength.
\end{enumerate}
  \item \label{as:We} \term{Elastic energy density:} $W_e \colon \R^{3 \times 3} \to [0,\infty]$ is continuous, polyconvex, and satisfies the following coercivity and control estimates for an exponent $r > p$ and all $E,F \in \R^{3 \times 3}$:
  \begin{align}
    C^{-1} \abs{E}^r - C &\leq W_e(E) \quad\text{and}\quad \text{$W_e(E) = +\infty$ if $\det E \leq 0$},  \label{eq:We_coerc}\\
    W_e(E) &\leq C_M(1 + W_e(F)) \quad\text{if $F^{-1}E \in X_M$ ($M \geq 1$),} \label{eq:We_control}
  \end{align}
  where $X_M := \setn{ A \in \R^{3 \times 3} }{ \abs{A} \leq M, \; \det A = 1 }$ for $M \geq 1$, and the constant $C_M > 0$ may depend on $M$ (but $C > 0$ in~\eqref{eq:We_coerc} is independent of $M$).
  \item \label{as:R} \term{Dissipation potential:} For $b \in \Bcal$,
\[\qquad
  R^b \colon D \to [0,\infty),  \qquad\text{where}\qquad
  D := \setb{ (P,\xi) \in \R^{3 \times 3} \times \Wedge_2\R^{1+3} }{ \det P = 1 }, 
\]
satisfies:
\begin{enumerate}[(i)]
  \item $R^b(P,\frarg)$ is convex and positively $1$-homogeneous for any $P \in \R^{3 \times 3}$ with $\det P = 1$;
  \item $R^b$ is locally Lipschitz continuous in $D$, that is, for every compact set $K \subset D$ there is $L = L(K) > 0$ such that $\abs{R^b(P_1,\xi_1) - R^b(P_2,\xi_2)} \leq L (\abs{P_1-P_2} + \abs{\xi_1-\xi_2})$ for all $(P_1,\xi_1), (P_2,\xi_2) \in K$;
  \item $R^b(P,\xi) \leq C_K \abs{\pbf(\xi)}$ for $(P,\xi) \in K$ in any compact set $K \subset D$ ($C_K > 0$ may depend on $K$);
  \item $R^b(P,\xi) \geq C^{-1} \abs{\pbf(\xi)}$ for all $(P,\xi) \in D$ ($C > 0$ independent of $P,\xi$).
\end{enumerate}
  \item \label{as:F}\label{as:last_list} \term{External loading:} $f \in \Crm^1([0,T];\Wrm^{1,p}(\Omega;\R^3)^*)$.
\end{enumerate}
\medskip

For $y \in \Wrm^{1,p}_g(\Omega;\R^3)$ and $P \in \Wrm^{1,q}(\Omega;\R^{3 \times 3})$ with $\det P = 1$ a.e.\ in $\Omega$, we define the \term{elastic energy}
\[
  \Wcal_e(y,P) := \int_\Omega W_e(\nabla y(x) P(x)^{-1}) \dd x.
\]
In this context, let us briefly comment on Assumption~\ref{as:We}. First recall that the energy density $W_e \colon \R^{3 \times 3} \to [0,\infty]$ is called \term{polyconvex} if it can be written in the form
\[
  W_e(E) = \bar{W}_{\!\!e}(E,\cof E,\det E),  \qquad E \in \R^{3 \times 3},
\]
with $\bar{W}_{\!\!e} \colon \R^{3 \times 3} \times \R^{3 \times 3} \times \R \to [0,\infty]$ continuous and convex (as a function on $\R^{3 \times 3} \times \R^{3 \times 3} \times \R \cong \R^{19}$).

\begin{example} \label{ex:We}
Consider the elastic energy density
\[
  W_e(E) := \tilde{W}(E) + \Gamma(\det E),  \qquad E \in \R^{3 \times 3},
\]
where
\begin{enumerate}[(i)]
  \item $\tilde{W} \colon \R^{3 \times 3} \to [0,\infty)$ is continuous, convex {or polyconvex}, has $r$-growth, and is $r$-coercive for some exponent $r > p$, i.e., $C^{-1} \abs{E}^r - C \leq \tilde{W}(E) \leq C(1 + \abs{E}^r)$ for a constant $C > 0$ and all $E \in \R^{3 \times 3}$ with $\det E > 0$;
  \item $\Gamma \colon \R \to [0,+\infty]$ is continuous, convex, and $\Gamma(s) = +\infty$ if and only if $s \leq 0$. 
\end{enumerate}
Then, the continuity, polyconvexity, and coercivity~\eqref{eq:We_coerc} are immediate. To see~\eqref{eq:We_control}, assume $E,F \in \R^{3 \times 3}$ with $F^{-1}E \in X_M$ ($M \geq 1$). Then,
\begin{align*}
  W_e(E) &= \tilde{W}(E) + \Gamma(\det E) \\
  &\leq C(1 + \abs{E}^r) + \Gamma(\det E) \\
  &\leq C(1 + \abs{F}^r \cdot \abs{F^{-1}E}^r) + \Gamma(\det F \cdot \det(F^{-1}E)) \\
  &\leq C(1 + \abs{F}^r \cdot M^r) + \Gamma(\det F) \\
  &\leq C M^r (1 + \tilde{W}(F) + \Gamma(\det F)) \\
  &=C_M(1 + W_e(F)).
\end{align*}
One concrete example fitting into these assumptions is $\tilde{W}(E) := \abs{E}^r$ for $r > p$ and $\Gamma(s) := 1/s$ for $s > 0$, $\Gamma(s) := +\infty$ for $s \leq 0$.
\end{example}

The \term{core energy} of the dislocation system $\Td = (T^b)_b \in \Disl(\cl{\Omega})$ {(see Section~\ref{sc:DS} for the definition of this set)} is defined as
\[
  \Wcal_c(\Td) := \frac{\zeta}{2} \sum_{b \in \Bcal} \Mbf(T^b),
\]
where $\zeta > 0$ was specified in Assumption~\ref{as:general}. More complicated expressions (e.g., with anisotropy or dependence on the type of dislocation) are possible, but we will only use the above to keep the exposition as simple as possible.

We can then define for $y \in \Wrm^{1,p}_g(\Omega;\R^3)$, $P \in \Wrm^{1,q}(\Omega;\R^{3 \times 3})$ with $\det P = 1$ a.e.\ in $\Omega$, and $\Td =(T^b)_b \in \Disl(\cl{\Omega})$ the \term{total energy}
\begin{equation} \label{eq:E}
  \Ecal(t,y,P,\Td) := \Wcal_e(y,P) - \dprb{f(t),y} + \Wcal_c(\Td),
\end{equation}
where $f$ is the external loading specified in Assumption~\ref{as:F} and $\dpr{\frarg,\frarg}$ is the duality product between $\Wrm^{1,p}(\Omega;\R^3)^*$ and $\Wrm^{1,p}(\Omega;\R^3)$. 

We next turn to the dissipation. For this, we first introduce a convenient notation for a path in the full internal variable space induced by a slip trajectory. Let $z = (P,\Td) \in \Wrm^{1,q}(\Omega;\R^{3 \times 3}) \times \Disl(\cl{\Omega})$ with $\det P = 1$ a.e.\ in $\Omega$. For $\Sd \in \Slip(\Td;[0,T])$ {(see Section~\ref{sc:slips})} we also write
\[
  \Sd \in \Slip(z;[0,T])
\]
in order to emphasize the starting point for the plastic distortion $P$. We abbreviate this to $\Sd \in \Slip(z)$ if $[0,T] = [0,1]$. The \term{(joint) path induced by $\Sd$ starting from $z$} is
\[
  (P_{\Sd}, \Sd) \in \Lip([0,T];\Wrm^{1,q}(\Omega;\R^{3 \times 3})) \times \Lip([0,T];\Disl(\cl{\Omega})),
\]
where $P_{\Sd}$ is as in~\eqref{eq:PSigma_def}. Finally, if $\Sd \in \Slip(z)$ (i.e., $[0,T] = [0,1]$), we define the \term{joint forward operator} via
\[
  \Sd_\ff z := (\Sd_\ff P, \; \Sd_\ff \Td);
\]
{see Sections~\ref{sc:slips},~\ref{sc:plast_evol} for the definitions of the constituent forward operators.}

The \term{dissipation} of $\Sd \in \Slip(z;[0,T])$, with $z = (P,\Td)$ as above, in the interval $I \subset [0,T]$ is
\[
  \Diss(\Sd;I) := \frac12 \sum_{b \in \Bcal} \int_{I \times \R^3} R^b \bigl( P_{\Sd}(t,x) , \vec{S}^b(t,x) \bigr) \dd \tv{S^b}(t,x).
\]
Here, $P_{\Sd}$ is understood as a continuous map from $[0,T] \times \cl{\Omega}$ to $\R^{3 \times 3}$. If $\Sd \in \Slip(z)$, i.e., $[0,T] = [0,1]$, then we also just write
\[
  \Diss(\Sd) := \Diss(\Sd;[0,1]).
\]

Note that $\Diss(\Sd;I)$ depends on $P$ (from $z = (P,\Td)$) through $P_{\Sd}$. However, we think of $\Sd$ as \enquote{attached} at the starting point $z$ and from the context it will always be clear where it is attached, usually through the notation \enquote{$\Sd \in \Slip(z;[0,T])$}. While this constitutes a slightly imprecise use of notation, it improves readability and hence we will adopt it in the following.

The next example presents a concrete dissipational cost similar to the one in~\cite{HudsonRindler22}.

\begin{example} \label{ex:diss_hardening}
Let $\tilde{R}^b \colon \Wedge_2\R^3 \to [0,\infty)$ be convex, positively $1$-homogeneous, Lipschitz, and satisfy the bounds
\[
  C^{-1} \abs{\xi} \leq \tilde{R}^b(\xi) \leq C \abs{\xi}
\]
for all $\xi \in \Wedge_2\R^3$ and a $b$-uniform constant $C > 0$. We remark that the (global) Lipschitz continuity is in fact automatic in this situation, see, e.g.,~\cite[Lemma~5.6]{Rindler18book}. Assume furthermore that for all $b \in \Bcal$ we are given a \enquote{hardening factor} $h^b \colon [1,\infty) \to (0,\infty)$ that is locally Lipschitz continuous, increasing, and satisfies
\begin{equation} \label{eq:tau4}
  C^{-1} \tau^4 - C \leq h^b(\tau),  \qquad \tau \geq 1,
\end{equation}
for a constant $C > 0$ (which is chosen uniform in $b$). Then set for $(P,\xi) \in D$ (as in Assumption~\ref{as:R})
\[
  R^b(P,\xi) := h^b(\abs{P}) \cdot \tilde{R}^b(P \pbf(\xi)),
\]
where $\pbf(\xi)$ denotes the pushforward of the $2$-vector $\xi$ under the spatial projection $\pbf(t,x) := x$, which is then further pushed forward under $P$. Note that $\abs{P} \geq 1$ since $\det P = 1$, so the above expression is well-defined. The first three points in Assumption~\ref{as:R} are easily verified. For the fourth point (coercivity), we observe that $P^{-1} = (\cof P)^T$ since $\det P = 1$, and so, by Hadamard's inequality,
\[
  \abs{P^{-1}}^2 \leq C \abs{P}^4 \leq C h^b(\abs{P}) 
\]
for some $C > 0$, where for the second inequality we have also used that $\tau^4 \leq (C+C^2/\min h^b) h^b(\tau)$ for all $\tau \geq 1$, which is an elementary consequence of~\eqref{eq:tau4}. Then,
\[
  \abs{\pbf(\xi)}
  \leq \abs{P^{-1}}^2 \cdot \abs{P \pbf(\xi)}
  \leq C \abs{P^{-1}}^2 \cdot \tilde{R}^b(P \pbf(\xi))
  \leq C R^b(P,\xi),
\]
which is the claim.
\end{example}

In the previous example, the hardening factor $h^b(P)$ can be interpreted as making it more energetically expensive for dislocations to glide if $\abs{P}$ becomes large. This is physically reasonable since after a large amount of plastic distortion has taken place, the crystal will have many point defects and so dislocation glide is impeded~\cite{HullBacon11book,AndersonHirthLothe17book}. It is also necessary for our mathematical framework: Without a hardening factor the dissipation no longer controls the variation and no solution may exist for positive times (see the proof of Proposition~\ref{prop:IP_solution} and also of Lemma~\ref{lem:Diss} below). This corresponds to instantaneous ripping of the specimen. For instance, even if $\det P = 1$, a principal minor of $P$ may blow up, e.g., for $P_\eps := \diag(\eps,\eps,\eps^{-2})$ with $\eps \todown 0$.

\begin{remark}
More generally, in Assumption~\ref{as:R} one could require $R^b$ to be only semielliptic instead of convex in the second argument $\xi$, see, e.g.,~\cite[Section~8.3]{KrantzParks08book} for a definition of this generalized convexity notion. This allows for more general dissipation potentials, but semiellipticity is hard to verify in general.
\end{remark}

\begin{remark} \label{rem:additive_hardening}
The present theory extends to $\Ecal$ incorporating an additional (additive) hardening or softening energy of the form
\[
  \Wcal_h(P,\Td)
\] 
for $P \in \Wrm^{1,q}(\Omega;\R^{3 \times 3})$ with $\det P = 1$ a.e.\ in $\Omega$, and $\Td \in \Disl(\cl{\Omega})$. In order for this to be compatible, the modified $\Ecal$ still needs to satisfy the conclusions of Lemma~\ref{lem:conv} below.
\end{remark}

\begin{remark} \label{rem:Wc_restrict}
The definition of the core energy above also counts the length of the \enquote{virtual} lines on the surface $\partial \Omega$, which is perhaps undesirable on physical grounds. One can also treat the more realistic core energy
\[
  \tilde{\Wcal}_c(\Td) := \frac{\zeta}{2} \sum_{b \in \Bcal} \Mbf(T^b \restrict \Omega),  \qquad \Td = (T^b)_b \in \Disl(\cl{\Omega}),
\]
which only counts the dislocation length inside $\Omega$, at the expense of further technical complications. An outline of the required modifications is as follows: Instead of $T^b$ {with $\partial T^b = 0$ (globally)} we now need to consider equivalence classes of currents {$T \in \Irm_1(\cl{\Omega})$ with $(\partial T) \restrict \Omega = 0$ (which is a weakly*-closed subspace of $\Irm_1(\cl{\Omega})$)} as follows: Set
\[
  {[T] := \setb{ \hat{T} \in \Irm_1(\cl{\Omega})}{ \hat{T} \restrict \Omega = T \restrict \Omega, \; \text{$\partial \hat{T} = 0$ globally} }}
\]
and adjust the definitions of $\Disl(\cl{\Omega})$ and $\Slip(\frarg)$ accordingly. The key point is the observation that there is always a $\tilde{T}^b \in [T^b]$ with
\[
  \Mbf(\tilde{T}^b) \leq C \cdot \Mbf(T^b \restrict \Omega)
\]
for some (domain-dependent) constant $C > 1$. Indeed, any {piecewise Lipschitz curve $\gamma \colon [0,1] \to \cl{\Omega}$ with $(\partial \dbr{\gamma}) \restrict \Omega = 0$ (where $\dbr{\gamma} \in \Irm_1(\cl{\Omega})$ denotes the integral $1$-current associated with $\gamma$) that is not globally closed (i.e., $\partial \dbr{\gamma} \neq 0$) can be closed to obtain a piecewise Lipschitz curve $\bar{\gamma} \colon [0,1] \to \cl{\Omega}$ with $\dbr{\bar{\gamma}} \restrict \Omega = \dbr{\gamma}$ and $\partial \dbr{\bar{\gamma}} = 0$ in such a way that}
\begin{equation} \label{eq:curve_ext}
  \Hcal^1(\im \bar{\gamma}) \leq C \cdot \Hcal^1(\im \gamma).
\end{equation}
For this one needs to use the property of a Lipschitz domain that for any $x, y \in \partial \Omega$ with $x \neq y$ and lying in the same connected component of $\partial \Omega$, there is an injective Lipschitz curve $\kappa \colon [0,1] \to \partial \cl{\Omega}$ from $x$ to $y$ such that
\[
  \Hcal^1(\im \kappa) \leq C \abs{x-y},
\]
which follows from a contradiction argument using the compactness and Lipschitz regularity of $\partial \Omega$. This path can be used to close $\gamma$ to $\bar{\gamma}$, yielding~\eqref{eq:curve_ext}. The extension to general integral $1$-currents then follows from standard methods.

Now, with $\tilde{\Wcal}_c$ in place of $\Wcal_c$ in a modified total energy $\tilde{\Ecal}$, the coercivity of the energy (see Lemma~\ref{lem:E_coerc} below) is weaker and we only control
\[
  \Mbf(\Td \restrict \Omega)
  = \frac12 \sum_{b \in \Bcal} \Mbf(T^b \restrict \Omega)
\]
instead of the full $\Mbf(\Td)$. However, with the argument outlined above, one can always pick the good representative $\tilde{T}^b$ in the equivalence class $[T^b]$ and run the arguments with $\tilde{T}^b$ in place of $T^b$. In this way one obtains the same existence result as in Theorem~\ref{thm:main} below with the more realistic total energy $\tilde{\Ecal}$. However, the full proof of this fact involves a few further technicalities, which are essentially straightforward, but cumbersome, since we are dealing with equivalence classes of currents everywhere (e.g., for the recovery construction, we first need to pick the good representative and then construct a recovery sequence for it). Thus, to keep the presentation as clear as possible, our main result and proof are stated without this further complication.
\end{remark}

\begin{remark} \label{rem:loadings}
Assumption~\ref{as:general}~(iv) on the imposed boundary values for the total deformation can be weakened. It is only used to obtain full coercivity in $\Wrm^{1,p}(\Omega;\R^3)$ from an $\Lrm^p$-bound on the gradient. Thus, some extensions to mixed Dirichlet/Neumann boundary conditions, which could even be time-dependent, are possible; cf.~\cite{MainikMielke09} for some techniques in this direction.
\end{remark}

\subsection{Energetic formulation} \label{sc:energetic}

In general, jumps in time cannot be excluded for rate-independent systems~\cite{MielkeRoubicek15book}. Thus, we will work with a rescaled time $s$ in which the process does not have jumps (or, more precisely, the jumps are resolved). By the rate-independence, this rescaling does not change the dynamics besides a reparameterization of the external loading. In the existence theorem to follow, we will construct a Lipschitz rescaling function $\psi \colon [0,\infty) \to [0,T]$, which is increasing and satisfies $\psi(0) = 0$, $\psi(\infty) = T_* \in (0,T]$. The original time $t$ is then related to $s$ via $t = \psi(s)$. Here, $T_*$ is the final time, i.e.\ the (original) time at which our solution blows up or we have reached the maximum time $T$. Our proof will show in particular that $T_* > 0$. 

Given such a Lipschitz continuous and increasing rescaling function $\psi \colon [0,\infty) \to [0,T]$ satisfying $\psi(0) = 0$, $\psi(\infty) = T_* \in (0,T]$, set
\[
  f_\psi := f \circ \psi
\]
and, for $s \in [0,\infty)$ and $y,P,\Td$ as in the original definition of $\Ecal$,
\begin{equation} \label{eq:Epsi}
  \Ecal_\psi(s,y,P,\Td) := \Wcal_e(y,P) - \dprb{f_\psi(s),y} + \Wcal_c(\Td).
\end{equation}
Our notion of solution is the following:

\begin{definition} \label{def:sol}
The pair $(y,z)$ with
\begin{align*}
  y &\in \Lrm^\infty([0,\infty);\Wrm^{1,p}_g(\Omega;\R^3)), \\
  z &= (P,\Sd) = \bigl( P, (S^b)_b \bigr)\in \Lip([0,\infty);\Wrm^{1,q}(\Omega;\R^{3 \times 3})) \times \Lip([0,\infty);\Disl(\cl{\Omega}))
\end{align*}
is called an \term{energetic solution to the system of dislocation-driven elasto-plasticity} with rescaling function $\psi \colon [0,\infty) \to [0,T]$, which is Lipschitz continuous, increasing, and satisfies $\psi(0) = 0$, $\psi(\infty) = T_* \in (0,T]$, if for all $s \in [0,\infty)$ the following conditions hold:
\[
\left\{
\begin{aligned}
  &\text{\term{(S) Stability:} If $\dot{\psi}(s) > 0$, then} \\
  &\quad\qquad  \Ecal_\psi(s,y(s),z(s)) \leq \Ecal_\psi(s,\hat{y},\hat{\Sd}_\ff z(s)) + \Diss(\hat{\Sd}) \\
  &\text{\phantom{(S)} for all $\hat{y} \in \Wrm^{1,p}_g(\Omega;\R^3)$, $\hat{\Sd} \in \Slip(\Sd(s))$.} \\[8pt]
  &\text{\term{(E) Energy balance:}} \\
  &\quad\qquad  \Ecal_\psi(s,y(s),z(s)) = \Ecal_\psi(0,y_0,z_0) - \Diss(\Sd;[0,s]) - \int_0^s \dprb{\dot{f}_\psi(\sigma),y(\sigma)} \dd \sigma. \\[8pt]
  &\text{\term{(P) Plastic flow:}} \\
  &\quad\qquad  \frac{\di}{\di s} P(s,x) = D(s,x,P(s,x);\Sd)  \qquad\text{and}\qquad \text{$\det P(s) = 1$ a.e.\ in $\Omega$} \\[4pt]
  &\text{\phantom{(P)} with} \\
  &\quad\qquad  D(s,x,R;\Sd) := \frac12 \sum_{b \in \Bcal} b \otimes \proj_{\langle R^{-1}b \rangle^\perp}[g^b(s,x)], \\
  &\text{\phantom{(P)}where $g^b$ is the density of $\hodge \pbf(S^b_\eta) := \hodge \pbf(\vec{S}^b_\eta) \, \tv{S^b_\eta}$ with $S^b_\eta := \eta \conv S^b$.}
\end{aligned}
\right.
\]
\end{definition}

Here and in the following, we use the notation $\Lrm^\infty(I;X)$ for the set of (Bochner-)measurable and uniformly norm-bounded functions defined on the interval $I \subset \R$ with values in the Banach space $X$, but we do not identify maps that are equal almost everywhere in $I$. In a similar vein, we use the good representative for $s \mapsto S^b(s)$, so that $z(s) = (P(s), \Sd(s)) = (P(s), (S^b(s))_b)$ is well-defined for every $s \in [0,\infty)$.

Moreover, $\Diss(\hat{\Sd})$ in~(E) is to be interpreted relative to $z_0$ (recall from Section~\ref{sc:assumpt} that the starting point is omitted in our notation). In~(S), the condition $\dot{\psi}(s) > 0$ includes the existence of $\dot{\psi}(s)$, which is the case for $\Lcal^1$-almost every $s \in [0,\infty)$ by Rademacher's theorem. The differential equation in~(P) is to be understood in $\Wrm^{1,q}(\Omega;\R^{3 \times 3})$ (as in Lemma~\ref{lem:PSigma_ODE_BV}).

Let us now motivate how the above formulation~(S),~(E),~(P) corresponds to the model developed in~\cite{HudsonRindler22}, as outlined in the Introduction. First, we observe that in general we do not have enough regularity to consider derivatives of the processes or functionals. Instead, we reformulate the model as follows: The condition~(P) corresponds directly to~\eqref{eq:plastflow}. The stability~(S) and energy balance~(E) come about as follows: The Free Energy Balance (a consequence of the Second Law of Theormodynamics) in the whole domain $\Omega$ reads as (see Section~4 in~\cite{HudsonRindler22})
\begin{equation} \label{eq:second_law}
  \frac{\dd}{\dd t}\bigl[ \Wcal_e(y(t),z(t)) + \Wcal_c(z(t)) \bigr] - \Pcal(t,y(t)) = -\Delta(t).
\end{equation}
Here, the external power is given as
\[
  \Pcal(t,y(t)) = \dprb{f(t), \dot{y}(t)},
\]
where $\dpr{\frarg,\frarg}$ is the duality product between $\Wrm^{1,p}(\Omega;\R^3)^*$ and $\Wrm^{1,p}(\Omega;\R^3)$, and we neglect the inertial term for the rate-independent formulation (cf.~Section~6.1 in~\cite{HudsonRindler22}). If we integrate~\eqref{eq:second_law} in time over an interval $[0,t] \subset [0,T]$ and use an integration by parts to observe
\[
  \int_0^t \Pcal(\tau,y(\tau)) \dd \tau = - \int_0^t \dprb{\dot{f}(\tau), y(\tau)} \dd \tau + \dprb{f(t),y(t)} - \dprb{f(0),y(0)},
\] 
we arrive at
\[
  \Ecal(t,y(t),z(t)) - \Ecal(0,y(0),z(0)) = - \Diss(\Sd;[0,t]) - \int_0^t \dprb{\dot{f}(\tau),y(\tau)} \dd \tau.
\]
This yields~(E) after the rescaling described at the beginning of this section.

The stability~(S) is a stronger version of the local stability relation
\[
  P^{-T} X^b \in \partial R^b(0),
\]
which follows from the flow rule~\eqref{eq:flowrule} or, more fundamentally, the Principle of Virtual Power (see Section~4 in~\cite{HudsonRindler22}). We refer to~\cite{MielkeRoubicek15book} for more on the equivalence or non-equivalence of~(S) \&~(E) with \enquote{differential} models of rate-independent processes. 

\begin{remark}
The pieces where $\psi$ is flat correspond to the jump transients, which are therefore explicitly resolved here. Note that there could be several $\Diss$-minimal slip trajectories connecting the end points of a jump, which lead to different evolutions for the plastic distortion. Thus, we cannot dispense with an explicit jump resolution. Moreover, the stability may not hold along such a jump transient and hence we need to require $\dot{\psi}(s) > 0$ in~(S). We refer to~\cite{OrtizRepetto99,Mielke02,Mielke03a,MielkeRossiSavare09,DalMasoDeSimoneSolombrino10,DalMasoDeSimoneSolombrino11,MielkeRossiSavare12,RindlerSchwarzacherVelazquez21} for   more on this.
\end{remark}

\begin{remark}
The stability~(S) in particular entails the elastic minimization
\[
  y(s) \in \Argmin \setb{ \Ecal_\psi(s,\hat{y},z(s)) }{ \hat{y} \in \Wrm^{1,p}_g(\Omega;\R^3) }
\]
as well as the orientation-preserving assertion
\[
  \det \nabla y(s) > 0  \quad\text{a.e.\ in $\Omega$}
\]
for all $s \in [0,\infty)$. This follows by testing with $\hat{\Sd} := \Id^{\Sd(s)} \in \Slip(\Sd(s))$ from Lemma~\ref{lem:neutral}, and also using the properties of $W_e$ in Assumption~\ref{as:We}. In this sense, we are in an elastically optimal state. This corresponds to the supposition that elastic movements are much faster than plastic movements, which is true in many materials~\cite{DeHossonyRoosMetselaar02,ArmstrongArnoldZerilli09,BenDavidEtAl14}.
\end{remark}

\begin{remark}
It can be seen without too much effort that the above formulation is indeed rate-independent: Let $a \colon [0,S] \to [0,S']$ be an invertible $\Crm^1$-map with $a(0) = 0$, $a(S) = S'$. Then, for $s' \in [0,S']$, we set
\[
  y'(s') := y(a^{-1}(s')),  \qquad 
  P'(s') := P(a^{-1}(s')),  \qquad
  \Sd' := a_* \Sd.
\]
The rescaling-invariance is obvious for~(S), where of course we now have to switch to the external force
\[
  f'_\psi(s') := f_\psi(a^{-1}(s')) = f_{\psi \circ a^{-1}}(s').
\]
For~(E), the rescaling invariance is a consequence of a change of variables together with Lemma~\ref{lem:Diss_rescale} in the following section: For $s' \in [0,S']$, we compute
\begin{align*}
  &\Ecal_{\psi \circ a^{-1}}(s',y'(s'),P'(s'),\Sd'(s')) \\
  &\qquad = \Ecal_\psi(a^{-1}(s'),y(a^{-1}(s')),P(a^{-1}(s')),\Sd(a^{-1}(s'))) \\
  &\qquad = \Ecal_\psi(0,y_0,z_0) - \int_0^{a^{-1}(s')} \dprb{\dot{f}_\psi(\sigma),y(\sigma)} \dd \sigma - \Diss(\Sd;[0,a^{-1}(s')]) \\
  &\qquad = \Ecal_{\psi \circ a^{-1}}(0,y_0,z_0) - \int_0^{s'} \dprb{\dot{f}_{\psi \circ a^{-1}}(\sigma'),y(\sigma')} \dd \sigma' - \Diss(\Sd';[0,s']).
\end{align*}
For~(P) the rate-independence has already been shown in Lemma~\ref{lem:Sigma_rescale}.
\end{remark}

\subsection{Existence of solutions}

The main result of this work is the following existence theorem:

\begin{theorem} \label{thm:main}
Assume~\ref{as:first}--\ref{as:last_list} and
\begin{enumerate}[({A}5)] \setlength\itemsep{8pt}
  \item \label{as:initial}\label{as:last} \term{Initial data:} $(y_0,z_0) = (y_0,P_0,\Td_0) \in \Wrm^{1,p}_g(\Omega;\R^3) \times \Wrm^{1,q}(\Omega;\R^{3 \times 3}) \times \Disl(\cl{\Omega})$ with $\det P_0 = 1$ a.e.\ in $\Omega$ is such that the initial stability relation
  \[
    \Ecal(0,y_0,z_0) \leq \Ecal(0,\hat{y},\hat{\Sd}_\ff z_0) + \Diss(\hat{\Sd})
  \]
  holds for all $\hat{y} \in \Wrm^{1,p}_g(\Omega;\R^3)$, $\hat{\Sd} \in \Slip(\Td_0)$.
\end{enumerate}
Then, there exists an energetic solution to the system of dislocation-driven elasto-plasticity in the sense of Definition~\ref{def:sol} satisfying the initial conditions
\[
  y(0) = y_0, \qquad
  P(0) = P_0, \qquad
  \partial S^b \restrict (\{0\} \times \R^3)  = - \delta_0 \times T^b_0 \quad\text{for all $b \in \Bcal$,}
\]
where $\Td_0 = (T^b_0)_b$. Moreover,
\[
  \Var_{\Wrm^{1,q}}(P;[0,s]) + \Var(\Sd;[0,s]) \leq C \cdot \Diss(\Sd;[0,s])
\]
for a constant $C > 0$ that depends only on the data in Assumptions~\ref{as:first}--\ref{as:last}.
\end{theorem}

\begin{remark} \label{rem:Wkp}
If we additionally assume that $P_0$ is of class $\Wrm^{k,p}$ for some $k \in \N$ with $k \geq 2$, then also $P(s)$ is of class $\Wrm^{k,p}$ for all $s \in [0,\infty)$ and $\Var_{\Wrm^{k,q}}(P;[0,s]) \leq C \cdot \Diss(\Sd;[0,s])$ for a ($k$-dependent constant) $C > 0$. If $P_0$ is smooth, then so it $P(s)$ for all $s \in [0,\infty)$. The proof of these claims follows from a straightforward generalization of Lemma~\ref{lem:PSigma_W1q}.
\end{remark}

\subsection{Properties of the energy and dissipation}

In preparation for the proof of Theorem~\ref{thm:main} in the next sections, we collect several properties of the energy and dissipation functionals. We start with the question of coercivity.

\begin{lemma} \label{lem:E_coerc}
For every $t \in [0,T]$, $y \in \Wrm^{1,p}_g(\Omega;\R^3)$ with $\det \nabla y > 0$ a.e.\ in $\Omega$, $P \in \Lrm^s(\Omega;\R^{3 \times 3})$ with $\det P = 1$ a.e.\ in $\Omega$, and $\frac1s = \frac1p - \frac1r$, it holds that
\[
  \Ecal(t,y,P,\Td) \geq C^{-1} \bigl( \norm{y}_{\Wrm^{1,p}}^p + \Mbf(\Td) \bigr) - C \bigl( \norm{P}_{\Lrm^s}^s + \norm{f(t)}_{[\Wrm^{1,p}]^*}^{p/(p-1)} + 1 \bigr)
\]
for a constant $C > 0$.
\end{lemma}

\begin{proof}
For $a,b > 0$ and all $\rho > 1$ we have the elementary inequality
\[
  \frac{a}{b} \geq \rho a^{1/\rho} - (\rho-1) b^{1/(\rho-1)},
\]
which follows from Young's inequality for $a^{1/\rho}, b$ with exponents $\rho, \rho/(\rho-1)$. Hence, for $F,P \in \R^{3 \times 3}$ with $\det P \neq 0$ we get with $\rho := r/p$, whereby $\rho-1 = r/s$, that
\[
  \abs{FP^{-1}} \geq \frac{\abs{F}}{\abs{P}}
  \geq \frac{r}{p} \abs{F}^{p/r} - \frac{r}{s} \abs{P}^{s/r}.
\]
Raising this inequality to the $r$'th power and using the coercivity in Assumption~\ref{as:We}, we get (combining constants as we go),
\begin{align*}
  \Wcal_e(y,P) &\geq \int_\Omega C^{-1} \abs{\nabla y P^{-1}}^r - C \dd x \\
  &\geq \int_\Omega C^{-1} \abs{\nabla y}^p - C \abs{P}^s - C \dd x \\
  &= C^{-1} \norm{\nabla y}_{\Lrm^p}^p - C(\norm{P}_{\Lrm^s}^s + 1) \\
  &\geq C^{-1} \norm{y}_{\Wrm^{1,p}}^p - C(\norm{P}_{\Lrm^s}^s + 1),
\end{align*}
where in the last line we further employed the Poincar\'{e}--Friedrichs inequality (the boundary values of $y$ are fixed). Moreover,
\[
  \Wcal_c(\Td) \geq C^{-1} \Mbf(\Td).
\]

On the other hand, we have for any $\eps > 0$, by Young's inequality again,
\begin{equation} \label{eq:power_est}
  -\dprb{f(t),y}
  \geq - \norm{f(t)}_{[\Wrm^{1,p}]^*} \cdot \norm{y}_{\Wrm^{1,p}}
  \geq - C_\eps \norm{f(t)}_{[\Wrm^{1,p}]^*}^{p/(p-1)} - \eps \norm{y}_{\Wrm^{1,p}}^p.
\end{equation}
Combining the above estimates, and choosing $\eps > 0$ sufficiently small to absorb the last term in~\eqref{eq:power_est} into the corresponding term originating from $\Wcal_e$, the claim of the lemma follows.
\end{proof}

The next lemma extends the classical results on the weak continuity of minors~\cite{Reshetnyak67,Ball76} and in a similar form seems to have been proved first in~\cite[Proposition~5.1]{MainikMielke09} (or see~\cite[Lemma~4.1.3]{MielkeRoubicek15book}):

\begin{lemma} \label{lem:minors}
Let $p > 3$ and $s > \frac{2p}{p-1}$. Assume that
\begin{align*}
  y_j &\toweak y  \quad\text{in $\Wrm^{1,p}$,}\\
  P_j &\to P  \quad\text{in $\Lrm^s$,}\\
  \det P_j &= 1  \quad\text{a.e.\ in $\Omega$ and for all $j \in \N$.}
\end{align*}
Then,
\begin{align*}
  \nabla y_j P_j^{-1} &\toweak \nabla y P^{-1}, \\
  \cof(\nabla y_j P_j^{-1}) &\toweak \cof(\nabla y P^{-1}), \\
  \det(\nabla y_j P_j^{-1}) &\toweak \det(\nabla y P^{-1})
\end{align*}
in $\Lrm^\sigma$ for some $\sigma > 1$.
\end{lemma}

\begin{proof}
We have, by Cramer's rule, $\nabla y_j P_j^{-1} = \nabla y_j \cdot (\cof P_j)^T$, and using, for instance, Pratt's convergence theorem,
\[
  (\cof P_j)^T \to (\cof P)^T  \quad\text{in $\Lrm^{s/2}$.}
\]
Then,
\[
  \nabla y_j P_j^{-1} \toweak \nabla y P^{-1}  \quad\text{in $\Lrm^{\sigma'}(D;\R)$}
\]
if $\frac{1}{\sigma'} := \frac1p + \frac2s < 1$, which is equivalent to our assumption $s > \frac{2p}{p-1}$.

Next, we recall that
\[
  \cof(\nabla y_j P_j^{-1})
  = \cof(\nabla y_j) \cdot \cof(P_j^{-1})
  = \cof(\nabla y_j) \cdot P_j^T.
\]
By the weak continuity of minors (see, e.g.,~\cite[Lemma~5.10]{Rindler18book}) we know that $\cof(\nabla y_j) \toweak \cof(\nabla y)$ in $\Lrm^{p/2}$. Thus,
\[
  \cof(\nabla y_j P_j^{-1}) \toweak \cof(\nabla y P^{-1})  \quad\text{in $\Lrm^{\sigma''}(D;\R)$}
\]
if $\frac{1}{\sigma''} := \frac2p + \frac1s < 1$, which is equivalent to $s > \frac{p}{p-2}$. Since our assumptions imply $s > \frac{2p}{p-1} > \frac{p}{p-2}$, we also obtain convergence in this case.

Finally,
\[
  \det(\nabla y_j P_j^{-1})
  = \det(\nabla y_j).
\]
By the weak continuity of minors again,
\[
  \det(\nabla y_j P_j^{-1}) \toweak \det(\nabla y P^{-1})  \quad\text{in $\Lrm^{\sigma'''}(D;\R)$}
\]
for $\sigma''' \in (1,p/3]$. Then take $\sigma := \min\{\sigma',\sigma'',\sigma'''\}$.
\end{proof}

We can then state a result on the lower semicontinuity of the elastic energy:

\begin{proposition} \label{prop:We_lsc}
The functional $(\hat{y},\hat{P}) \mapsto \Wcal_e(\hat{y},\hat{P})$ is weakly (sequentially) lower semicontinuous with respect to sequences $(y_j) \subset \Wrm^{1,p}_g(\Omega;\R^3)$ satisfying $\det \nabla y_j > 0$ a.e.\ in $\Omega$, and $(P_j) \subset \Wrm^{1,q}(\Omega;\R^{3 \times 3})$ with $\det P_j = 1$ a.e.\ in $\Omega$.
\end{proposition}

\begin{proof}
Let $(y_j,P_j)$ be as in the statement of the proposition with $y_j \toweak y$ in $\Wrm^{1,p}$ and $P_j \toweak P$ in $\Wrm^{1,q}$. Let $s > \frac{2p}{p-1}$. By the Rellich--Kondrachov theorem, $\Wrm^{1,q}(\Omega;\R^{3 \times 3}) \cembed \Lrm^s(\Omega;\R^{3 \times 3})$ (since $q > 3$ this holds for all $s \in [1,\infty]$) and hence $P_j \to P$ strongly in $\Lrm^s$. Then, by Lemma~\ref{lem:minors} all minors of the compound sequence $\nabla y_j P_j^{-1}$ converge weakly in $\Lrm^\sigma$ for some $\sigma > 1$. Thus, the lower semicontinuity follows in the same way as for the polyconvex integrand $W_e$ (via strong lower semicontinuity and Mazur's lemma), see, e.g.,~\cite[Theorem~6.5]{Rindler18book} for this classical argument.
\end{proof}

Next, we establish some basic properties of the dissipation.

\begin{lemma} \label{lem:Diss}
Let $z = (P,\Td) \in \Wrm^{1,q}(\Omega;\R^{3 \times 3}) \times \Disl(\cl{\Omega})$ with $\det P = 1$ a.e.\ in $\Omega$. For $\Sd \in \Slip(z)$ it holds that
\begin{equation} \label{eq:Diss_est}
  C^{-1} \Var(\Sd) \leq \Diss(\Sd)
\end{equation}
with a constant $C > 0$. Moreover, for $\Sd^1 \in \Slip(z)$, and $\Sd^2 \in \Slip(\Sd^1_\ff z)$, it holds that
\begin{equation} \label{eq:Diss_additive}
  \Diss(\Sd^2 \circ \Sd^1) = \Diss(\Sd^1) + \Diss(\Sd^2).
\end{equation}
\end{lemma}

\begin{proof}
The first claim follows directly from the properties assumed on $R^b$ in Assumption~\ref{as:R}. The second claim~\eqref{eq:Diss_additive} follows in the same way as~\eqref{eq:cat_Var} in Lemma~\ref{lem:Sigma_concat} (also using~\eqref{eq:cat_ff}).
\end{proof}

\begin{lemma} \label{lem:Diss_rescale}
Let, $\Sd \in \Slip(z;[0,T])$, where $z = (P,\Td) \in \Wrm^{1,q}(\Omega;\R^{3 \times 3}) \times \Disl(\cl{\Omega})$ with $\det P = 1$ a.e.\ in $\Omega$, and let $a \colon [0,T] \to [0,T']$ be an injective Lipschitz map with $a(0) = 0$, $a(T) = T'$. Then, for
\[
  a_* \Sd := (a_* S^b)_b \in \Slip(\Td;[0,T'])
\]
it holds that
\[
  \Diss(a_* \Sd;[0,T']) = \Diss(\Sd;[0,T]).
\]
\end{lemma}

\begin{proof}
This follows in exactly the same way as in the proof of Lemma~\ref{lem:rescale} (see Lemma~3.4 in~\cite{Rindler23}) using also Lemma~\ref{lem:Sigma_rescale}.
\end{proof}

\begin{lemma} \label{lem:Diss_lsc}
The mapping $\Sd \mapsto \Diss(\Sd;[0,T])$ is lower semicontinuous for sequences $\Sd_j \in \Slip(z;[0,T])$ with $\Sd_j \toweakstar \Sd$, where $z = (P,\Td) \in \Wrm^{1,q}(\Omega;\R^{3 \times 3}) \times \Disl(\cl{\Omega})$ with $\det P = 1$ a.e.\ in $\Omega$.
\end{lemma}

\begin{proof}
Write $\Sd_j = (S^b_j)_b$ and $\Sd = (S^b)_b$. By Reshetnyak's lower semicontinuity theorem (see, for instance,~\cite[Theorem~2.38]{AmbrosioFuscoPallara00book})
\[
  \int_{[0,T] \times \Omega} R^b(P_{\Sd},\vec{S}^b) \dd \tv{S} 
  \leq \liminf_{j \to \infty} \int_{[0,T] \times \Omega} R^b(P_{\Sd},\vec{S}^b_j) \dd \tv{S_j}.
\]
Furthermore, by Lemma~\ref{lem:PSigma_cont}, $P_{\Sd_j} \to P_{\Sd}$ uniformly in $[0,1] \times \Omega$. Thus, also using the local Lipschitz continuity of $R^b$ (see Assumption~\ref{as:R}), the fact that both $\abs{P_{\Sd_j}}$ and $\abs{\vec{S}^b_j}$ are uniformly bounded, and Fatou's lemma, we obtain
\begin{align*}
  \Diss(\Sd;[0,T]) &\leq \frac12 \sum_{b \in \Bcal} \liminf_{j \to \infty} \int_{[0,T] \times \Omega} R^b(P_{\Sd},\vec{S}^b_j) \dd \tv{S_j} \\
  &= \frac12 \sum_{b \in \Bcal} \liminf_{j \to \infty} \int_{[0,T] \times \Omega} R^b(P_{\Sd_j},\vec{S}^b_j) \dd \tv{S_j} \\
  &\leq \liminf_{j \to \infty} \, \Diss(\Sd_j;[0,T]).
\end{align*}  
This is the assertion.
\end{proof}

For convenient later use, in the following lemma we collect several convergence assertions.

\begin{lemma} \label{lem:conv}
The following hold:
\begin{enumerate}[(i)]
  \item $(t,y,P,\Td) \mapsto \Ecal(t,y,P,\Td)$ is lower semicontinuous for sequences $t_j \to t$ in $[0,T]$, $y_j \toweak y$ in $\Wrm^{1,p}_g(\Omega;\R^3)$, $P_j \to P$ in $\Lrm^s(\Omega;\R^{3 \times 3})$ with $\det P_j = 1$ a.e.\ in $\Omega$, and $\Td_j \toweakstar \Td$ in $\Disl(\cl{\Omega})$.
  \item $(t,P) \mapsto \Ecal(t,y,P,\Td)$ is continuous for sequences $t_j \to t$, $P_j \toweak P$ in $\Wrm^{1,q}(\Omega;\R^{3 \times 3})$ with $\det P_j = 1$ a.e.\ in $\Omega$, at fixed $y \in \Wrm^{1,p}_g(\Omega;\R^3)$ such that $\Wcal_e(y,P) < \infty$ and $\Td \in \Disl(\cl{\Omega})$.
  \item $(t,y) \mapsto \dpr{\dot{f}(t),y}$ is continuous for sequences $t_j \to t$ in $[0,T]$ and $y_j \toweak y$ in $\Wrm^{1,p}_g(\Omega;\R^3)$.
  \item $\Sd \mapsto \Diss(\Sd;[0,T])$ is lower semicontinuous for sequences $\Sd_j \in \Slip(z;[0,T])$ with $\Sd_j \toweakstar \Sd$, where $z = (P,\Td) \in \Wrm^{1,q}(\Omega;\R^{3 \times 3}) \times \Disl(\cl{\Omega})$ with $\det P = 1$ a.e.\ in $\Omega$.
\end{enumerate}
\end{lemma}

\begin{proof}
\proofstep{Ad~(i).}
The first term $\Wcal_e(y,P)$ in the definition of $\Ecal$, see~\eqref{eq:E}, is lower semicontinuous by Proposition~\ref{prop:We_lsc}; the second term $-\dpr{f(t),y}$ is in fact continuous since $f(t)$ is continuous in $t$ with values in the dual space to $\Wrm^{1,p}(\Omega;\R^3)$ by~\ref{as:F}; the third term $\Wcal_c(\Td)$ is weakly* lower semicontinuous by the weak* lower semicontinuity of the mass and Fatou's lemma (as in Lemma~\ref{lem:Diss_lsc}).

\proofstep{Ad~(ii).}
We first prove the continuity property for $\Wcal_e$. The compact embedding of $\Wrm^{1,q}(\Omega;\R^{3 \times 3})$ into $\Crm(\Omega;\R^{3 \times 3})$ (since $q > 3$) entails that the $P_j$ are uniformly bounded and converge uniformly to $P$. We further observe via~\eqref{eq:We_control} in~\ref{as:We} (clearly, $PP_j^{-1} \in X_M$ for some $M \geq 1$) that
\[
  W_e(\nabla yP_j^{-1}) \leq C_M (1 + W_e(\nabla yP^{-1}))  \qquad\text{a.e.\ in $\Omega$}
\]
Since taking inverses is a continuous operation on matrices from $X_M$, we get $P_j^{-1} \to P^{-1}$ a.e.\ in $\Omega$. Then,
\[
  W_e(\nabla yP_j^{-1}) \to W_e(\nabla yP^{-1})  \qquad\text{a.e.\ in $\Omega$}
\]
by the continuity of $W_e$ (see~\ref{as:We}). Thus, as $C_M (1 + W_e(\nabla yP^{-1}))$ is integrable by assumption, it follows from the dominated convergence theorem that
\[
  \Wcal_e(y,P_j) \to \Wcal_e(y,P).
\]
For the power term we argue as in~(i).

\proofstep{Ad~(iii).}
This follows again from the properties of the external force, see~\ref{as:F}.

\proofstep{Ad~(iv).}
This was proved in Lemma~\ref{lem:Diss_lsc}.
\end{proof}

We also record the following fact, which occupies a pivotal position in this work. It allows us to translate the weak* convergence of dislocation systems into a slip trajectory (of vanishing dissipation) connecting these dislocation systems to their limit. This will be crucially employed later to show stability of the limit process (see Proposition~\ref{prop:stability}).

\begin{proposition} \label{prop:D_weak*}
Assume that $\Td_j = (T^b_j)_b, \Td = (T^b)_b \in \Disl(\cl{\Omega})$ ($j \in \N$) are such that 
\[
  \supmod_j \Mbf(\Td_j) < \infty.
\]
Then,
\[
  \text{$\dist_{\Lip,\cl{\Omega}}(T^b_j,T^b) \to 0$ for all $b \in \Bcal$}
  \qquad\text{if and only if}\qquad
  \Td_j \toweakstar \Td \quad\text{in $\Disl(\cl{\Omega})$.}
\]
In this case there are $\Sd_j \in \Slip(\Td_j)$ with $(\Sd_j)_\ff \Td_j = \Td$ and
\begin{align*}
  \limsup_{j\to\infty} \norm{\Sd_j}_{\Lrm^\infty} &\leq C \cdot \limsup_{\ell \to \infty} \Mbf(\Td_\ell), \\
  \Diss(\Sd_j) &\to 0,
\end{align*}
where the constant $C > 0$ only depends on the dimensions and on $\Omega$, and $\Diss(\Sd_j)$ is understood relative to any starting point $P \in \Wrm^{1,q}(\Omega;\R^{3 \times 3})$ for a $q \in (3,\infty]$ with $\det P = 1$ a.e.\ in $\Omega$.
\end{proposition}

\begin{proof}
%Since we assume that $\kappa$ is purely atomic, and also
Using the growth properties of $R^b$ in Assumption~\ref{as:R}, the first claim follows immediately from Proposition~\ref{prop:equiv}. For the existence of the $\Sd_j$ as claimed we further obtain $S^b_j \in \Irm^\Lip_{1+1}([0,1] \times \cl{\Omega})$ with
\[
  \partial S^b_j = \delta_1 \times T^b - \delta_0 \times T^b_j, \qquad
  \Var(S^b_j) \to 0,
\]
and
\[
  \limsup_{j\to\infty} \, \; \esssup_{t \in [0,1]} \, \Mbf(S^b_j(t)) \leq C \cdot \limsup_{\ell \to \infty} \, \Mbf(T^b_\ell)
\]
from this result. Then, for $\Sd_j := (S^b_j)_b$ it holds that $(\Sd_j)_\ff \Td_j = \Td$ and
\[
  \Diss(\Sd_j) = \frac12 \sum_{b \in \Bcal} \int_{[0,1] \times \R^3} R^b \bigl( P_{\Sd_j}, \vec{S}^b_j \bigr) \dd \tv{S^b_j}
  \leq C \sum_{b \in \Bcal} \Var(S^b_j) \to 0
\]
since $P_{\Sd_j}$ remains uniformly bounded (in $j$) by Lemma~\ref{lem:PSigma_W1q} (and the continuous embedding $\Wrm^{1,q}(\Omega;\R^{3 \times 3}) \embed \Crm(\Omega;\R^{3 \times 3})$), whereby Assumption~\ref{as:R}~(iii) becomes applicable.
\end{proof}

\begin{remark}
Note that we do not claim that \emph{any} two dislocation systems $\Td_1, \Td_2 \in \Disl(\cl{\Omega})$ can be connected by a slip trajectory. Indeed, if $\Omega$ is not simply connected and has a hole (with respect to countably $1$-rectifiable loops), then there are dislocation systems that cannot be deformed into each other.
\end{remark}

\section{Time-incremental approximation scheme}   \label{sc:approx}
 
We start our construction of the energetic solution with a time-discretized problem and corresponding discrete solution. For brevity of notation it will be convenient to define the \term{deformation space}
\[
  \Ycal := \setb{ \Wrm^{1,p}_g(\Omega;\R^3) }{ \text{$\det \nabla y > 0$ a.e.\ in $\Omega$} }
\]
and the \term{internal variable space}
\[
  \Zcal := \setb{ (P,\Td) \in \Wrm^{1,q}(\Omega;\R^{3 \times 3}) \times \Disl(\cl{\Omega}) }{ \text{$\det P = 1$ a.e.\ in $\Omega$} }.
\]

\subsection{Time-incremental minimization}

Consider for $N \in \N$ the partition of the time interval $[0,T]$ consisting of the $(N+1)$ points
\[
  t^N_k := k \cdot \Delta T^N,  \qquad k=0,1,\ldots,N,  \qquad\text{where}\qquad \Delta T^N := \frac{T}{N}.
\]
Set
\[
  y^N_0 := y_0, \qquad
  z^N_0 = (P^N_0,\Td^N_0) := (P_0,\Td_0) = z_0
\]
with $y_0,z_0$ from Assumption~\ref{as:initial}. For $k = 1, \ldots, N$, we will in the following construct
\[
  (y^N_k,z^N_k,\Sd^N_k) = (y^N_k,P^N_k,\Td^N_k,\Sd^N_k) \in \Ycal \times \Zcal \times \Slip(z^N_{k-1})
\]
according to the \term{time-incremental minimization problem}
\begin{equation} \label{eq:IP} \tag{IP}
\left\{ \begin{aligned}
  &\text{$(y^N_k,\Sd^N_k)$ }\;\text{minimizes $(\hat{y},\hat{\Sd}) \mapsto \Ecal \bigl( t^N_k, \hat{y}, \hat{\Sd}_\ff z^N_{k-1} \bigr) + \Diss(\hat{\Sd})$}\\
  &\phantom{\text{$(y^N_k,\Sd^N_k)$ }}\;\text{over all $\hat{y} \in \Ycal$, $\hat{\Sd} \in \Slip(z^N_{k-1})$ with $\norm{\hat{\Sd}}_{\Lrm^\infty} \leq \gamma$} \; ;\\
  &z^N_k := (\Sd^N_k)_\ff z^N_{k-1}.
\end{aligned} \right.
\end{equation}
Here,
\begin{equation} \label{eq:beta}
  \gamma \geq \Mbf(\Td_0)
\end{equation}
is a parameter.

\begin{remark}
The assumption $\norm{\hat{\Sd}}_{\Lrm^\infty} \leq \gamma$ in the minimization is necessary because we cannot control $\norm{\hat{\Sd}}_{\Lrm^\infty}$ by the variation of $\hat{\Sd}$ alone, see Example~3.6 in~\cite{Rindler23}. The assumption~\eqref{eq:beta} is required for the well-posedness of the time-incremental problem since it makes the neutral slip trajectory admissible (see Lemma~\ref{lem:neutral}) and hence the candidate set for the minimization in~\eqref{eq:IP} is not empty. Later, when we have a time-continuous process, we can infer a uniform mass bound from the energy balance~(E) and the coercivity of $\Ecal$ (Lemma~\ref{lem:E_coerc}) and then let $\gamma \to \infty$.
\end{remark}

The existence of discrete solutions is established in the following result. Here and in the following, all constants implicitly depend on the data in Assumptions~\ref{as:first}--\ref{as:last}.

\begin{proposition} \label{prop:IP_solution}
For $N$ large enough there exists a solution $(y^N_k,z^N_k,\Sd^N_k)$ to the time-incremental minimization problem~\eqref{eq:IP} for all $k = 0,\ldots,N$. Moreover, defining
\[
  e^N_k := \Ecal(t^N_k,y^N_k,z^N_k),  \qquad d^N_k := \Diss(\Sd^N_k),
\]
and
\[
  \alpha^N_k := 1 + e^N_k + \sum_{j=1}^k d^N_j,
\]
the difference inequality
\begin{equation} \label{eq:G_diff_ineq}
  \frac{\alpha^N_k-\alpha^N_{k-1}}{\Delta T^N} \leq C \ee^{\alpha^N_{k-1}}  \qquad\text{for $k = 1,\ldots,N$}
\end{equation}
holds, where $C > 0$ is a constant that depends only on the data in the assumptions.
\end{proposition}

\begin{proof}
Assume that for $k \in \{1,\ldots,N\}$ a solution $(y^N_j,z^N_j,\Sd^N_j)_{j = 1, \ldots, k-1}$ to the time-incremental minimization problem~\eqref{eq:IP} has been constructed up to step $k-1$. This is trivially true for $k = 1$. In the following, we will show that then also a solution $(y^N_k,z^N_k,\Sd^N_k)$ to~\eqref{eq:IP} at time step $k$ exists and~\eqref{eq:G_diff_ineq} holds.

\proofstep{Step~1: Any solution $(y^N_k,z^N_k,\Sd^N_k)$ to~\eqref{eq:IP} at time step $k$, if it exists, satisfies~\eqref{eq:G_diff_ineq}.}

\medskip

To show the claim we assume that $(y^N_k,z^N_k,\Sd^N_k)$ is a solution to~\eqref{eq:IP} at time step $k$. Testing with $\hat{y} := y^N_{k-1}$ and the neutral slip trajectory $\hat{\Sd} := \Id^{\Td^N_{k-1}} \in \Slip(z^N_{k-1})$ (see Lemma~\ref{lem:neutral}), we get 
\begin{equation}  \label{eq:ENk_est}
  e^N_k + d^N_k
  \leq \Ecal(t^N_k,y^N_{k-1},z^N_{k-1}) 
  = e^N_{k-1} - \int_{t^N_{k-1}}^{t^N_k} \dprb{\dot{f}(\tau),y^N_{k-1}} \dd \tau.
\end{equation}
To bound the integral, we first estimate for any $(t,y,P,\Td) \in (0,T) \times \Ycal \times \Zcal$ by virtue of Lemma~\ref{lem:E_coerc} (with the constant $C$ potentially changing from line to line)
\begin{align*}
  &\frac{\di}{\di t} \bigl( \Ecal(t,y,P,\Td) + \norm{P}_{\Lrm^s}^s + 1 \bigr) \\
  &\qquad = - \dprb{\dot{f}(t),y} \\
  &\qquad \leq \norm{\dot{f}(t)}_{[\Wrm^{1,p}]^*} \cdot \norm{y}_{\Wrm^{1,p}} \\
  &\qquad \leq C \norm{\dot{f}(t)}_{[\Wrm^{1,p}]^*} \bigl( \Ecal(t,y,P,\Td) +\norm{P}_{\Lrm^s}^s + \norm{f(t)}_{[\Wrm^{1,p}]^*}^{p/(p-1)} + 1 \bigr)^{1/p} \\
  &\qquad \leq C(\Ecal(t,y,P,\Td) + \norm{P}_{\Lrm^s}^s + 1),
\end{align*}
where in the last line we used $a^{1/p} \leq a$ for $a \geq 1$ and $C$ also absorbs the expressions depending on $f$. Gronwall's lemma then gives that for all $\tau \geq t$ it holds that
\[
  \Ecal(\tau,y,P,\Td) + \norm{P}_{\Lrm^s}^s + 1
  \leq (\Ecal(t,y,P,\Td) + \norm{P}_{\Lrm^s}^s + 1) \ee^{C (\tau-t)}.
\]
We may also estimate, using the same arguments as above,
\begin{align*}
  \absb{\dprb{\dot{f}(\tau),y^N_{k-1}}}
  &\leq C(\Ecal(\tau,y^N_{k-1},z^N_{k-1}) + \norm{P^N_{k-1}}_{\Lrm^s}^s + 1) \\
  &\leq C(\Ecal(t^N_{k-1},y^N_{k-1},z^N_{k-1}) + \norm{P^N_{k-1}}_{\Lrm^s}^s + 1) \ee^{C (\tau-t^N_{k-1})}.
\end{align*}
Plugging this into~\eqref{eq:ENk_est},
\begin{align*}
  e^N_k + d^N_k
  &\leq e^N_{k-1} + \int_{t^N_{k-1}}^{t^N_k} C(e^N_{k-1} + 1) \ee^{C (\tau-t^N_{k-1})} + C \norm{P^N_{k-1}}_{\Lrm^s}^s \ee^{C (\tau-t^N_{k-1})} \dd \tau \\
  &= e^N_{k-1} + (e^N_{k-1} + 1)(\ee^{C \Delta T^N} - 1) + \norm{P^N_{k-1}}_{\Lrm^s}^s (\ee^{C \Delta T^N} - 1).
\end{align*}
Next, observe via an iterated application of Lemma~\ref{lem:PSigma_def} and~\eqref{eq:Diss_est} in Lemma~\ref{lem:Diss} that
\begin{equation} \label{eq:PNk-1_est}
  \norm{P^N_{k-1}}_{\Lrm^s}
  \leq \norm{P_0}_{\Lrm^s} + C \sum_{j=1}^{k-1} \Var(\Sd_j)
  \leq \norm{P_0}_{\Lrm^s} + C \sum_{j=1}^{k-1} d^N_j.
\end{equation}
Combining the above estimates,
\begin{align*}
  \alpha^N_k - \alpha^N_{k-1}
  &= e^N_k + d^N_k - e^N_{k-1} \\
  &\leq (e^N_{k-1} + 1) (\ee^{C \Delta T^N} - 1) + \biggl(\norm{P_0}_{\Lrm^s} + C \sum_{j=1}^{k-1} d^N_j \biggr)^s \cdot (\ee^{C \Delta T^N} - 1) \\
  &\leq \alpha^N_{k-1} (\ee^{C \Delta T^N} - 1) + C \ee^{\alpha^N_{k-1}}(\ee^{C \Delta T^N} - 1) \\
  &\leq C \ee^{\alpha^N_{k-1}} \Delta T^N,
\end{align*}
where we used that $a^s \leq C\ee^a$ for $a \geq 1$, and $\ee^{C\Delta T^N} - 1 \leq 2C \Delta T^N$ for $\Delta T^N$ small enough. We remark that we used the exponential function (as opposed to a polynomial expression) here mainly for reasons of convenience. We thus arrive at the claim~\eqref{eq:G_diff_ineq} at $k$.

\proofstep{Step~2: In~\eqref{eq:IP} at time step $k$, the minimization may equivalently be taken over $\hat{y} \in \Ycal$, $\hat{\Sd} \in \Slip(z^N_{k-1})$ satisfying the bounds
\begin{align} 
  \norm{\hat{y}}_{\Wrm^{1,p}} &\leq \tilde{C}(\alpha^N_{k-1}),\label{eq:y_assume} \\
  \Var(\hat{\Sd}) &\leq \tilde{C}(\alpha^N_{k-1}), \label{eq:VarSigma_assume} \\
  \norm{\hat{\Sd}}_{\Lrm^\infty} &\leq \gamma, \label{eq:SigmaLinfty_assume}
\end{align}
for a constant $\tilde{C}(\alpha^N_{k-1}) > 0$, which only depends on the data from the assumptions besides $\alpha^N_{k-1}$.}

\medskip

Recalling~\eqref{eq:IP}, we immediately have~\eqref{eq:SigmaLinfty_assume}. To see the claims~\eqref{eq:y_assume},~\eqref{eq:VarSigma_assume}, observe first that from Step~1 we may restrict the minimization in~\eqref{eq:IP} at time step $k$ to $\hat{y},\hat{\Sd}$ such that for
\[
  \hat{\alpha}^N_k(\hat{y},\hat{\Sd}) := 1 + \Ecal(t^N_k,\hat{y},\hat{\Sd}_\ff z^N_{k-1}) + \sum_{j=1}^{k-1} d^N_j + \Diss(\hat{\Sd})
\]
it holds that
\[
  \hat{\alpha}^N_k(\hat{y},\hat{\Sd}) 
  \leq \alpha^N_{k-1} + C \ee^{\alpha^N_{k-1}} T
  =: \tilde{C}(\alpha^N_{k-1})
\]
since a minimizer $(\hat{y},\hat{\Sd}) = (y^N_k,\Sd^N_k)$, if it exists, must satisfy~\eqref{eq:G_diff_ineq} and hence this bound. From~\eqref{eq:Diss_est} in Lemma~\ref{lem:Diss} we then immediately get that
\[
  \Var(\hat{\Sd}) \leq C \cdot \Diss(\hat{\Sd}) \leq C \hat{\alpha}^N_k(\hat{y},\hat{\Sd}) \leq C\tilde{C}(\alpha^N_{k-1}).
\]
Hence, the requirement~\eqref{eq:VarSigma_assume} is established after redefining $\tilde{C}(\alpha^N_{k-1})$.

Next, for all $\hat{y} \in \Ycal$, $\hat{\Sd} \in \Slip(z^N_{k-1})$ with~\eqref{eq:VarSigma_assume}, we get by virtue of Lemma~\ref{lem:E_coerc},
\begin{align*}
  &\Ecal(t^N_k,\hat{y},\hat{\Sd}_\ff z^N_{k-1}) \notag\\
  &\quad\geq C^{-1} \bigl( \norm{\hat{y}}^p_{\Wrm^{1,p}} + \Mbf(\hat{\Sd}_\ff \Td^N_{k-1}) \bigr) - C \bigl( \norm{\hat{\Sd}_\ff P^N_{k-1}}_{\Lrm^s}^s + \norm{\dot{f}(t^N_k)}_{[\Wrm^{1,p}]^*}^{p/(p-1)} + 1 \bigr)
\end{align*}
for a constant $C > 0$. We estimate similarly to~\eqref{eq:PNk-1_est},
\[
  \norm{\hat{\Sd}_\ff P^N_{k-1}}_{\Lrm^s}
  \leq \norm{P_0}_{\Lrm^s} + C \biggl( \sum_{j=1}^{k-1} \Var(\Sd^N_j) + \Var(\hat{\Sd}) \biggr) \\
  \leq C \tilde{C}(\alpha^N_{k-1}),
\]
where we also used~\eqref{eq:VarSigma_assume}. Then, using further Assumption~\ref{as:F}, we see that
\[
  C \bigl( \hat{\alpha}^N_k(\hat{y},\hat{\Sd}) + \tilde{C}(\alpha^N_{k-1})^s + 1 \bigr)
  \geq \norm{\hat{y}}^p_{\Wrm^{1,p}}.
\]
Hence, we may assume that $\hat{y}$ satisfies~\eqref{eq:y_assume} after redefining $\tilde{C}(\alpha^N_{k-1})$ once more.

\proofstep{Step~3: A solution $(y^N_k,z^N_k,\Sd^N_k)$ to~\eqref{eq:IP} at time step $k$ exists.}

\medskip

From the previous step we know that we may restrict the minimization to all $\hat{y} \in \Ycal$, $\hat{\Sd} \in \Slip(z^N_{k-1})$ satisfying the bounds~\eqref{eq:y_assume}--\eqref{eq:SigmaLinfty_assume}. Clearly, taking  $\hat{y} := y^N_{k-1}$ and $\hat{\Sd} := \Id^{\Td^N_{k-1}} \in \Slip(z^N_{k-1})$, the set of candidate minimizers is not empty (also recall~\eqref{eq:beta}). We now claim that we may then take a minimizing sequence $(\hat{y}_n,\hat{\Sd}_n) \subset \Ycal \times \Slip(z^N_{k-1})$ for~\eqref{eq:IP} such that
\begin{equation} \label{eq:incr_conv}
  \hat{y}_n \toweak y_*  \quad\text{in $\Wrm^{1,p}$}  \qquad\text{and}\qquad
  \hat{\Sd}_n \toweakstar \Sd_* \quad\text{in $\Slip(z^N_{k-1})$.}
\end{equation}
The first convergence follows by selecting a subsequence (not relabelled) using~\eqref{eq:y_assume} and the weak compactness of norm-bounded sets in $\Wrm^{1,p}_g(\Omega;\R^3)$.

For the second convergence, we observe via~\eqref{eq:VarSigma_assume},~\eqref{eq:SigmaLinfty_assume} that for $\hat{\Sd}_n$ it holds that
\[
  \supmod_n \, \Bigl( \norm{\hat{\Sd}_n}_{\Lrm^\infty} + \Var(\hat{\Sd}_n) \Bigr) < \infty.
\]
Moreover, a rescaling via Lemma~\ref{lem:rescale} shows that we may additionally assume the steadiness property
\begin{equation} \label{eq:Sigma_n_steady}
  t \mapsto t^{-1} \Var(\hat{\Sd}_n;[0,t]) \equiv L_n,  \qquad t \in (0,1],
\end{equation}
for constants $L_n \geq 0$ that are bounded by (an $n$-independent) constant $L > 0$. Crucially, this rescaling does not change the expression
\[
  \Ecal \bigl( t^N_k, \hat{y}, (\hat{\Sd}_n)_\ff z^N_{k-1} \bigr) + \Diss(\hat{\Sd}_n)
\]
by Lemmas~\ref{lem:Sigma_rescale},~\ref{lem:Diss_rescale}. Hence we may replace the original $\hat{\Sd}_n$ by its rescaled version. The steadiness property~\eqref{eq:Sigma_n_steady} now entails that all the maps $t \mapsto \Var(\hat{S}^b_n;[0,t])$, where we have written $\hat{\Sd}_n = (\hat{S}^b_n)_b$, are uniformly Lipschitz. Indeed, for $0 \leq s < t \leq 1$,
\begin{align*}
  \Var(\hat{S}^b_n;[s,t])
  &\leq 2 \Var(\hat{\Sd}_n;[s,t]) \\
  &= 2 \bigl( \Var(\hat{\Sd}_n;[0,t]) - \Var(\hat{\Sd}_n;[0,s]) \bigr) \\
  &= 2L_n (t-s) \\
  &\leq 2L (t-s).
\end{align*}
Then we get from Proposition~\ref{prop:Sl_compact} that there exists $\Sd_* \in \Slip(z^N_{k-1})$ and a subsequence (not relabelled) such that $\hat{\Sd}_n \toweakstar \Sd_*$ in $\Slip(z^N_{k-1})$.

Next, we observe that the joint functional
\begin{equation} \label{eq:joint_funct}
  (\hat{y},\hat{\Sd}) \mapsto \Ecal \bigl( t^N_k, \hat{y}, \hat{\Sd}_\ff z^N_{k-1} \bigr) + \Diss(\hat{\Sd})
\end{equation}
is lower semicontinuous with respect to the convergences in~\eqref{eq:incr_conv}. To see this, we first note that by Lemmas~\ref{lem:Phiff_cont},~\ref{lem:PSigma_cont},
\begin{align*}
  (\hat{\Sd}_n)_\ff \Td^N_{k-1} &\toweakstar (\Sd_*)_\ff \Td^N_{k-1}  \quad\text{in $\Disl(\cl{\Omega})$,} \\
  (\hat{\Sd}_n)_\ff P^N_{k-1} &\to (\Sd_*)_\ff P^N_{k-1}  \quad\text{in $\Lrm^s$.}
\end{align*}
Moreover, from Lemma~\ref{lem:PSigma_def} we get that
\[
  \det \, [(\hat{\Sd}_n)_\ff P^N_{k-1}] = 1  \quad\text{a.e.\ in $\Omega$.}
\]
The first and second term in~\eqref{eq:joint_funct} are then lower semicontinuous by 
Lemma~\ref{lem:conv}~(i) and~(iv), respectively.  We note that $y_* \in \Ycal$ since it must have finite energy by the weak lower semicontinuity of $\Ecal$, whereby also $\det \nabla y_* > 0$ a.e.\ in $\Omega$ by~\eqref{eq:We_coerc} in Assumption~\ref{as:We}. Thus, we conclude that $(y^N_k,\Sd^N_k) := (y_*,\Sd_*)$ is the minimizer of the time-incremental minimization problem~\eqref{eq:IP} at time step $k$. By Step~1, this $(y^N_k,z^N_k,\Sd^N_k)$ satisfies~\eqref{eq:G_diff_ineq}.
\end{proof}

\subsection{Discrete energy estimate and stability}

The next task is to establish that our construction indeed yields a \enquote{discrete energetic solution}.

\begin{proposition} \label{prop:incr_prop}
Let $(y^N_k,z^N_k,\Sd^N_k)_{k=0,\ldots,N}$ be a solution to the time-incremental minimization problem~\eqref{eq:IP}. Then, for all $k \in \{0,\ldots,N\}$ the following hold:
\begin{enumerate}[(i)]
  \item The discrete lower energy estimate
\begin{equation} \label{eq:incr_lower_energy}  \qquad
  \Ecal(t^N_k,y^N_k,z^N_k) \leq \Ecal(0,y_0,z_0) - \sum_{j=1}^k \Diss(\Sd^N_j) - \sum_{j=1}^k \int_{t^N_{j-1}}^{t^N_j} \dprb{\dot{f}(\tau),y^N_{j-1}} \dd \tau.
\end{equation}
\item The discrete stability
\begin{equation} \label{eq:incr_stab}  \qquad
  \Ecal(t^N_k,y^N_k,z^N_k) \leq \Ecal(t^N_k,\hat{y},\hat{\Sd}_\ff z^N_k) + \Diss(\hat{\Sd})
\end{equation}
for all $\hat{y} \in \Ycal$ and $\hat{\Sd} \in \Slip(z^N_k)$ with $\norm{\hat{\Sd}}_{\Lrm^\infty} \leq \gamma$.
\end{enumerate}
\end{proposition}

\begin{proof}
In the following we abbreviate, for $j = 0,\ldots,k$,
\[
  e^N_j := \Ecal(t^N_j,y^N_j,z^N_j), \qquad d^N_j := \Diss(\Sd^N_j).
\]
At $k=0$, the lower energy estimate~\eqref{eq:incr_lower_energy} holds trivially and the stability~\eqref{eq:incr_stab} is a part of Assumption~\ref{as:initial}. At $k = 1,2,\ldots$, testing the time-incremental minimization problem~\eqref{eq:IP} at step $j \in \{1,\ldots,k\}$ with $\hat{y} := y^N_{j-1}$ and $\hat{\Sd} := \Id^{\Td^N_{j-1}} \in \Slip(\Td^N_{j-1})$, we get like in~\eqref{eq:ENk_est} that
\[
  e^N_j + d^N_j \leq \Ecal(t^N_j,y^N_{j-1},z^N_{j-1})
  = e^N_{j-1} - \int_{t^N_{j-1}}^{t^N_j} \dprb{\dot{f}(\tau),y^N_{j-1}} \dd \tau.
\]
Iterating this estimate for $j=k,\ldots,1$ already yields~\eqref{eq:incr_lower_energy}.

Similarly, we may test~\eqref{eq:IP} at time step $k$ with $\hat{y} \in \Ycal$ and $\hat{\Sd} \circ \Sd^N_k$ for $\hat{\Sd} \in \Slip(z^N_k)$ with $\norm{\hat{\Sd}}_{\Lrm^\infty} \leq \gamma$, which satisfies $\norm{\hat{\Sd} \circ \Sd^N_k}_{\Lrm^\infty} \leq \gamma$, to obtain
\begin{align*}
  \Ecal(t^N_k,y^N_k,z^N_k) + \Diss(\Sd^N_k) &\leq \Ecal(t^N_k,\hat{y},(\hat{\Sd} \circ \Sd^N_k)_\ff z^N_{k-1}) + \Diss(\hat{\Sd} \circ \Sd^N_k) \\
  &= \Ecal(t^N_k,\hat{y},\hat{\Sd}_\ff z^N_k) + \Diss(\hat{\Sd}) + \Diss(\Sd^N_k),
\end{align*}
where we have used Lemma~\ref{lem:Sigma_concat} and $\Diss(\hat{\Sd} \circ \Sd^N_k) = \Diss(\hat{\Sd}) + \Diss(\Sd^N_k)$ by the additivity of the dissipation, see~\eqref{eq:Diss_additive} in Lemma~\ref{lem:Diss}. Canceling $\Diss(\Sd^N_k)$ on both sides, we arrive at~\eqref{eq:incr_stab}.
\end{proof}

\subsection{A-priori estimates}

In this section we establish a bound on the $\alpha^N_j$, which were defined in Proposition~\ref{prop:IP_solution}, that is uniform in $N$. This is complicated by the fact that in the coercivity estimate of $\Ecal$ at time step $j$, the term $\norm{P^N_j}_{\Lrm^s}^s$ occurs with a negative sign (see Lemma~\ref{lem:E_coerc}). The exponent $s > 1$ makes $\norm{P^N_j}_{\Lrm^s}^s$ grow superlinearly in $\sum_{\ell=1}^j \Var(\Sd^N_{\ell})$, potentially causing blow-up in finite time. In order to deal with this, we first establish a nonlinear Gronwall-type lemma:

\begin{lemma} \label{lem:Gronwall_ext}
Let $T > 0$, $N \in \N$, and let $h \colon \R \to [0,\infty)$ be a continuous and increasing function. Assume that the sequence of real numbers $a_j \in \R$, $j = 0, \ldots, N$, satisfies the difference inequality
\begin{equation} \label{eq:diff_ineq}
  \frac{a_j-a_{j-1}}{\Delta T} \leq h(a_{j-1}), \qquad
  \Delta T := \frac{T}{N}, \qquad
  j = 1, \ldots, N.
\end{equation} 
Let $A_*$ be the maximal solution, defined on a time interval $[0,T_\infty)$ (possibly $T_\infty = +\infty$), to the ODE
\begin{equation} \label{eq:Gronwall_ext_ODE}
  \left\{ \begin{aligned}
    A'(t) &= h(A(t)),  \quad t > 0, \\
    A(0) &= a_0.
  \end{aligned}\right.
\end{equation}
Then, for all $j \in \{0,\ldots,N\}$ with $j\Delta T < T_\infty$ it holds that
\[
  a_j \leq A_*(j \Delta T).
\]
\end{lemma}

We remark that the \term{maximal solution} to~\eqref{eq:Gronwall_ext_ODE} is a solution $A_* \colon [0,T_\infty) \to \R$ of~\eqref{eq:Gronwall_ext_ODE} with the property that for any other solution $A$ of~\eqref{eq:Gronwall_ext_ODE} it holds that $A \leq A_*$ on the intersection of both intervals of definition. It can be shown, see, e.g.,~\cite[Section~8.IX, p.67]{Walter70book}, that $A_*$ exists and can be maximally defined; we assume that our interval $[0,T_\infty)$ is already such a maximal domain of definition. Obviously, if a unique solution $A$ to~\eqref{eq:Gronwall_ext_ODE} exists on a maximal time interval $[0,T_\infty)$, then $A_* = A$.

\begin{proof}
First, we remark that we may assume without loss of generality that $a_{j-1} \leq a_j$ for $j = 1,\ldots,N$. Indeed, we may set
\[
  b_j := a_0 + \sum_{\ell=1}^j \max\{a_\ell - a_{\ell-1}, 0\},
\]
which is clearly increasing, satisfies $a_j \leq b_j$, and
\[
  \frac{b_j-b_{j-1}}{\Delta T}
  = \max\biggl\{ \frac{a_j-a_{j-1}}{\Delta T}, 0 \biggr\}
  \leq h(a_{j-1})
  \leq h(b_{j-1}),
\]
where we used~\eqref{eq:diff_ineq} and the monotonicity of $h$. We then use $b_j$ in place of $a_j$.

Let $a$ be the piecewise-affine interpolant of $a_j$, namely,
\[
  a(t) := a_{j-1} + \frac{a_j-a_{j-1}}{\Delta T} (t-(j-1)\Delta T) \qquad \text{if $(j-1)\Delta T \leq t \leq j \Delta T$}
\]
for $j = 1,\ldots,N$. For the left lower Dini derivative $D_- a$ of $a$ we get
\[
  D_- a(t) := \liminf_{s \toup t} \, \frac{a(t)-a(s)}{t-s}
  = \frac{a_j-a_{j-1}}{\Delta T} \qquad \text{if $(j-1)\Delta T < t \leq j \Delta T$.}
\]
Thus, by~\eqref{eq:diff_ineq} and the fact that $a(t)$ is increasing and $h$ is monotone,
\[
  D_- a(t) \leq h(a_{j-1}) \leq h(a(t)),  \qquad t \in (0,T].
\]
By a classical comparison principle for ODEs, see~\cite[Theorem~8.X and remarks, p.68]{Walter70book},
\begin{equation} \label{eq:aA*}
  a(t) \leq A_*(t),  \qquad t \in [0,T_\infty),
\end{equation}
with $A_*(t)$ given as the maximal solution to~\eqref{eq:Gronwall_ext_ODE}. This directly implies the conclusion of the lemma.
\end{proof}

For the reader's convenience we give a short direct proof of~\eqref{eq:aA*}. First, we claim:

\medskip

\noindent\textit{Let $u,v \in \Crm([0,T'])$ with the following two properties:
\begin{enumerate}[(i)]
  \item $u(0) < v(0)$, and
  \item $D_- u(t) < D_- v(t)$ if $u(t) = v(t)$ at $t \in (0,T')$.
\end{enumerate}
Then, $u(t) < v(t)$ for all $t \in [0,T']$.}

\medskip

To see this claim, let $t_0 \in [0,T']$ be the first point such that $u(t_0) = v(t_0)$. By~(i), $t_0 > 0$. For $t < t_0$ it holds that $u(t) < v(t)$ and then
\[
  \frac{u(t_0)-u(t)}{t_0-t} > \frac{v(t_0)-v(t)}{t_0-t}.
\]
Taking the lower limit as $t \toup t_0$, we obtain $D_- u(t_0) \geq D_- v(t_0)$, which contradicts~(ii). This shows the claim.

\medskip

For $0 < \eps \leq 1$ let $A_\eps$ be a maximally extended solution to 
\[
  \left\{ \begin{aligned}
    A_\eps'(t) &= h(A_\eps(t)) + \eps,  \quad t > 0, \\
    A_\eps(0) &= a_0 + \eps.
  \end{aligned}\right.
\]
We have $A_{\eps'} < A_\eps$ for all $0 < \eps' < \eps \leq 1$ by our claim. In particular, $A_\eps(t)$ is monotonically decreasing as $\eps \todown 0$ and thus $A_\eps \todown A_*$ locally uniformly (by equi-continuity) with $A_*$ the maximal solution to~\eqref{eq:Gronwall_ext_ODE}. In fact, a similar argument can be used to construct $A_*$.

Fix $T' < T_\infty$. For any $\eps > 0$ we observe $a(0) < A_\eps(0)$ and
\[
  D_- a(t) - h(a(t)) \leq 0 < \eps = A_\eps'(t) - h(A_\eps(t)),  \qquad t \in (0,T').
\]
Hence, we may apply the claim again (note that in~(ii) we only need to check $D_- a(t) < D_- A_\eps(t)$ if $a(t) = A_\eps(t)$) to obtain
\[
  a(t) < A_\eps(t),  \qquad t \in [0,T'].
\]
Since the right-hand side converges to $A_*(t)$ as $\eps \todown 0$, we obtain $a(t) \leq A_*(t)$ for $t \in [0,T']$ and then also for $t \in [0,T_\infty)$. This is~\eqref{eq:aA*}. 

We can now state a uniform energy bound up to any time before the blow-up point $T_*$ ($> 0$), or up to $T$ if there is no blowup.

\begin{lemma} \label{lem:apriori}
Let $(y^N_k,z^N_k,\Sd^N_k)_{k=0,\ldots,N}$ be a solution to~\eqref{eq:IP}. With $\alpha^N_\ell := 1 + e^N_\ell + \sum_{i=1}^\ell d^N_i$ (as in Proposition~\ref{prop:IP_solution}) and
\begin{equation} \label{eq:betaNj}
  \beta^N_j := \alpha^N_0 + \sum_{\ell=1}^j \max\{\alpha^N_\ell - \alpha^N_{\ell-1}, 0\},
\end{equation}
the a-priori estimates
\begin{equation} \label{eq:apriori}
  \norm{y^N_k}_{\Wrm^{1,p}} + \norm{P^N_k}_{\Wrm^{1,q}} + \Mbf(\Td^N_k) + \sum_{j=1}^k \Var(\Sd^N_j) \leq \tilde{C}(\beta^N_k),  \qquad
  \norm{\Sd^N_k}_{\Lrm^\infty} \leq \gamma
\end{equation}
hold. Moreover,
\begin{equation} \label{eq:Tstar}
  T_* := \sup_{\lambda > 0} \liminf_{N \to \infty} \, \sup_{j\geq 0} \, \setb{ t^N_j }{ \beta^N_j \leq \lambda } > 0.
\end{equation}
\end{lemma}

\begin{proof}
From Proposition~\ref{prop:IP_solution} we know
\[
  \frac{\alpha^N_j-\alpha^N_{j-1}}{\Delta T^N} \leq C \ee^{\alpha^N_{j-1}},  \qquad \text{$j = 1,\ldots,N$.}
\]
By the same argument as the one at the beginning of the proof of Lemma~\ref{lem:Gronwall_ext}, $\alpha^N_j \leq \beta^N_j$ and
\begin{equation} \label{eq:beta_diffquot}
  \frac{\beta^N_j-\beta^N_{j-1}}{\Delta T^N}
  \leq \max \biggl\{ \frac{\alpha^N_j-\alpha^N_{j-1}}{\Delta T^N}, 0 \biggr\}
  \leq C \ee^{\alpha^N_{j-1}}
  \leq C \ee^{\beta^N_{j-1}}.
\end{equation}

The parts of~\eqref{eq:apriori} relating to $\norm{y^N_k}_{\Wrm^{1,p}}$, $\sum_{j=1}^k \Var(\Sd^N_j)$ and $\norm{\Sd^N_k}_{\Lrm^\infty}$ follow from the coercivity of $\Ecal$ and $\Diss$ in the same way as we proved~\eqref{eq:y_assume}--\eqref{eq:SigmaLinfty_assume} (in particular, using the coercivity estimates of Lemmas~\ref{lem:E_coerc},~\ref{lem:Diss}). From Lemma~\ref{lem:PSigma_W1q} we further get
\[
  \norm{P^N_k}_{\Wrm^{1,q}}
  \leq \norm{P_0}_{\Wrm^{1,q}} + C \sum_{j=1}^k \Var(\Sd_j) \\
  \leq C\tilde{C}(\alpha^N_k)
  \leq C\tilde{C}(\beta^N_k),
\]
where the constant $C$ from~\eqref{eq:PSigma_W1q} and then also the (redefined) constant $\tilde{C}(\beta^N_k)$ depend on $\norm{P_0}_{\Wrm^{1,q}}$ and $\sum_{j=1}^k \Var(\Sd^N_j) \leq \tilde{C}(\alpha^N_k) \leq \tilde{C}(\beta^N_k)$ (also see Lemma~\ref{lem:Diss}). For the bound on $\Mbf(\Td^N_k)$, we can use again Lemma~\ref{lem:E_coerc}, but this time using the coercivity originating from the core energy.

Our next task is to show that for $T_*$ defined in~\eqref{eq:Tstar} it holds that $T_* > 0$, for which we apply the preceding Lemma~\ref{lem:Gronwall_ext} with $h(s) := C \ee^s$, which is continuous and increasing, and initial value $\beta^N_0 = \alpha^N_0$ to~\eqref{eq:beta_diffquot}. The maximal solution is easily seen to be $A_*(t) = - \log(\ee^{-\alpha^N_0}-Ct)$, which is defined on the maximal interval $[0,T_\infty)$ with $T_\infty = \ee^{-\alpha^N_0}/C > 0$. Thus, as $A_*$ is increasing, for all $0 < \tau < T_\infty$ it holds that
\[
  \beta^N_k \leq A_*(t^N_k) \leq A_*(\tau) < \infty  \qquad\text{if $t^N_k \leq \tau$}\quad\text{($k \in \{0,\ldots,N\}$; $N \in \N$).}
\]
Consequently, $T_* \geq T_\infty > 0$.
\end{proof}

\section{Proof of the existence theorem}   \label{sc:existence_proof}
 
At this stage we have an $N$-uniform bound on $\sum_{j=1}^k \Var(\Sd^N_j)$ for any $t^N_k \leq \tau < T_*$. However, when letting $N \to \infty$, this BV-type bound is too weak to prevent the formation of jumps in the dislocation trajectory. Jumps are undesirable because we need the \enquote{time index} provided by a \emph{Lipschitz} trajectory to define the path of the plastic distortion as in Section~\ref{sc:plast_evol}. Hence, we now rescale the time to make the discrete evolution uniformly Lipschitz continuous (and move the blow-up time to $+\infty$). Then we will be able to pass to the limit and complete the proof of Theorem~\ref{thm:main}.

\subsection{Rescaling of time}

Let $N \in \N$ and set, for $k = 0,\ldots,N$,
\[
  e^N_k := \Ecal(t^N_k,y^N_k,z^N_k),  \qquad d^N_k := \Diss(\Sd^N_k),
\]
as in Proposition~\ref{prop:IP_solution}. Then define the increasing sequence
\[
  s^N_k := t^N_k + \sum_{j=1}^k \Bigl( \max\{e^N_j - e^N_{j-1}, 0\} + d^N_j \Bigr).
\]
The $\{s^N_k\}_{k=0,\ldots,N}$ form a (non-equidistant) partition of the interval $[0,\sigma^N]$, where
\[
  \sigma^N := s^N_N.
\]
It can be seen from the definition of $\beta^N_k = 1 + e^N_0 + \sum_{j=1}^k \max\{e^N_j - e^N_{j-1} + d^N_j, 0\}$, see~\eqref{eq:betaNj}, that
\begin{equation} \label{eq:betaNk_sNk}
  \beta^N_k - e^N_0 - 1 \leq s^N_k \leq t^N_k + \beta^N_k + \sum_{j=1}^k d^N_j.
\end{equation}
The \term{rescaling function} $\psi^N \colon [0,\infty) \to [0,T]$ is
\[
  \psi^N(s) := \begin{cases}
    t^N_{k-1} + \dfrac{\Delta T^N}{s^N_k-s^N_{k-1}}(s-s^N_{k-1})  &\text{for $s \in [s^N_{k-1},s^N_k]$, where $k = 1,\ldots,N$,} \\
    T  &\text{for $s \geq \sigma^N$.}
  \end{cases}
\]
Clearly, $\psi^N$ is strictly increasing on $[0,\sigma^N]$ and
\[
  \psi^N(s^N_k) = t^N_k,  \qquad k = 0,\ldots,N.
\]
In the new time the time-incremental minimization problem~\eqref{eq:IP} reads as follows: For $k = 0, \ldots, N$ we have in Proposition~\ref{prop:IP_solution} constructed solutions
\[
  (y^N_k,z^N_k,\Sd^N_k) = (y^N_k,P^N_k,\Td^N_k,\Sd^N_k) \in \Ycal \times \Zcal \times \Slip(z^N_{k-1})
\]
to the following:
\begin{equation} \label{eq:IP_s}  \tag{IP$^\prime$}
\left\{ \begin{aligned}
  &\text{$(y^N_k,\Sd^N_k)$ }\;\text{minimizes $(\hat{y},\hat{\Sd}) \mapsto \Ecal_{\psi^N} \bigl( s^N_k, \hat{y}, \hat{\Sd}_\ff z^N_{k-1} \bigr) + \Diss(\hat{\Sd})$}\\
  &\phantom{\text{$(y^N_k,\Sd^N_k)$ }}\;\text{over all $\hat{y} \in \Ycal$, $\hat{\Sd} \in \Slip(z^N_{k-1})$ with $\norm{\hat{\Sd}}_{\Lrm^\infty} \leq \gamma$} \; ;\\
  &z^N_k := (\Sd^N_k)_\ff z^N_{k-1}.
\end{aligned} \right.
\end{equation}
Moreover, we may also assume that $\Sd^N_k$ is steady in the sense that
\begin{equation} \label{eq:SigmaNk_steady}
  t \mapsto t^{-1} \Var(\Sd^N_k;[0,t]) \equiv L^N_k \geq 0,  \qquad t \in (0,1].
\end{equation}
This can be achieved via Lemma~\ref{lem:rescale} (see Step~3 in the proof of Proposition~\ref{prop:IP_solution} on why this rescaling is allowed).

We will now introduce a suitable interpolant for the discrete solution. Writing $\Td^N_k = (T^{N,b}_k)_b$ and $\Sd^N_k = (S^{N,b}_k)_b \in \Slip(z^N_{k-1})$, we define the process $\bar{\Sd}^N \in \Lip([0,\infty);\Disl(\cl{\Omega}))$ (this is to be understood with the Lipschitz condition holding in every interval $[0,S]$, $S>0$, with a uniform Lipschitz constant) as $\bar{\Sd}^N := (\bar{S}^{N,b})_b$ with
\[
  \bar{S}^{N,b} := \sum_{k=1}^N (a^N_k)_* S^{N,b}_k
  + \dbr{(\sigma^N,\infty)} \times T^{N,b}_N,
\]
where $a^N_k \colon [0,1] \to [s^N_{k-1},s^N_k]$ is given as $a^N_k(t) := s^N_{k-1} + (s^N_k-s^N_{k-1})t$. This rescales every $S^{N,b}_k$ to length $s^N_k-s^N_{k-1}$ and moves the starting point to $s^N_{k-1}$. Hence, $\bar{\Sd}^N$ is the \enquote{scaled concatenation} of all the $\Sd^N_k$. One can check easily from the definition of the $s^N_k$,~\eqref{eq:SigmaNk_steady}, and Lemma~\ref{lem:Diss_rescale} that the Lipschitz constant of $s \mapsto \Diss((a^N_k)_* \Sd^N_k; [s^N_{k-1},s])$ is bounded from above by $1$. Hence, also using Lemma~\ref{lem:Diss},
\begin{equation} \label{eq:VarSigmaN}
  \Var(\bar{\Sd}^N;[s,s']) \leq C \cdot \Diss(\bar{\Sd}^N;[s,s']) \leq C \abs{s-s'},  \qquad s,s' \in [0,\infty).
\end{equation}
In particular, we have indeed
\[
  \bar{\Sd}^N \in \Lip([0,\infty);\Disl(\cl{\Omega})).
\]

Next, we define $\bar{P}^N \in \Lip([0,\infty);\Wrm^{1,q}(\Omega;\R^{3 \times 3}))$ as the solution to~\eqref{eq:PSigma_ODE} with respect to $\bar{\Sd}^N$, i.e.,
\begin{equation} \label{eq:barPN_ODE}
  \left\{ \begin{aligned}
    \frac{\di}{\di s} \bar{P}^N(s,x) &= \frac12 \sum_{b \in \Bcal} b \otimes \proj_{\langle \bar{P}^N(s,x)^{-1}b \rangle^\perp} \bigl[ \bar{g}^{N,b}(s,x) \bigr] \qquad \text{for a.e.\ $s \in (0,\infty)$,}\\
    \bar{P}^N(0,x) &= P_0(x) \qquad \text{for a.e.\ $x \in \Omega$,}
  \end{aligned} \right.
\end{equation}
where $\bar{g}^{N,b}$ is the density of the measure $\hodge \pbf(\bar{S}^{N,b}_\eta)$. This ODE is to be understood in the $\Wrm^{1,q}$-sense. The existence, uniqueness, and regularity of a solution to~\eqref{eq:barPN_ODE} follow by Lemmas~\ref{lem:PSigma_def},~\ref{lem:PSigma_W1q},~\ref{lem:PSigma_ODE_BV}. By construction of $\bar{P}^N$ and Lemma~\ref{lem:Sigma_rescale},
\[
  \bar{P}^N(s^N_k) = P^N_k,  \qquad k = 0, \ldots, N.
\]

Finally, $\bar{y}^N \in \Lrm^\infty([0,\infty);\Ycal)$ is given as
\[
  \bar{y}^N(s) := y^N_k \qquad \text{if $s \in (s^N_{k-1},s^N_k]$, where $k = 1, \ldots, N$,}
\]
and also $\bar{y}^N(0) := y_0$, $\bar{y}^N(s) := y^N_N$ for $s \geq \sigma^N$.

We can then restate the a-priori estimates of Lemma~\ref{lem:apriori} and the discrete solution properties of Proposition~\ref{prop:incr_prop} in the new time scale.

\begin{proposition} \label{prop:apriori_s}
For all $0 < S < \infty$ there is a constant $\bar{C}(S) > 0$, with the dependence on $S$ increasing and upper semicontinuous in $S$, such that the a-priori estimates
\begin{align}
  \norm{\bar{y}^N}_{\Lrm^\infty([0,S];\Wrm^{1,p})} + \norm{\bar{P}^N}_{\Lrm^\infty([0,S];\Wrm^{1,q})} + \Var(\bar{\Sd}^N;[0,S]) &\leq \bar{C}(S),  \label{eq:apriori_s_1} \\  
   \norm{\bar{\Sd}^N}_{\Lrm^\infty([0,S];\Disl(\cl{\Omega}))} &\leq \gamma   \label{eq:apriori_s_2}
\end{align}
hold. Moreover, 
\begin{equation} \label{eq:psiN_lim}
  \lim_{\sigma \to \infty} \liminf_{N \to \infty} \psi^N(\sigma)
  = T_*
\end{equation}
with $T_* > 0$ from~\eqref{eq:Tstar}.
\end{proposition}

\begin{proof}
The assertions in~\eqref{eq:apriori_s_1}~\eqref{eq:apriori_s_2} with
\[
  \bar{C}(S) := \tilde{C} \bigl( S + \Ecal(0,y_0,z_0) + 1 \bigr)
\]
follow directly from Lemma~\ref{lem:apriori}, the definitions of $\bar{y}^N, \bar{P}^N, \bar{\Sd}^N$, Lemma~\ref{lem:PSigma_W1q}, and~\eqref{eq:betaNk_sNk}.

We now show~\eqref{eq:psiN_lim}. Let $\sigma > 0$. By~\eqref{eq:betaNk_sNk}, if $s^N_k \leq \sigma$ ($k \in \{0,\ldots,N\}$, $N \in \N$), or equivalently, $t^N_k = \psi^N(s^N_k) \leq \psi^N(\sigma)$, then the quantity $\beta^N_k$ remains bounded by $\sigma + \Ecal(0,y_0,z_0) + 1$. Hence, by the definition of $T_*$ in~\eqref{eq:Tstar}, we have $T_* \geq \liminf_{N \to \infty} \psi^N(\sigma)$, and then also
\begin{equation} \label{eq:Ts_geq}
  T_* \geq \lim_{\sigma \to \infty} \liminf_{N \to \infty} \psi^N(\sigma).
\end{equation}

On the other hand, if $T' < T_*$, then there is $\lambda < \infty$ with $\beta^N_k \leq \lambda$ for all $k \in \{0,\ldots,N\}$ such that $t^N_k \leq T'$ and $N \in \N$ sufficiently large. From~\eqref{eq:apriori} and Lemma~\ref{lem:Diss} we get that
\[
  \sum_{j=1}^k d^N_j \leq C\tilde{C}(\alpha^N_k) \leq C\tilde{C}(\beta^N_k) \leq C\tilde{C}(\lambda),
\]
where we have considered the \enquote{constant} $\tilde{C}$ as an increasing function. Thus, for the times $s^N_k = [\psi^N]^{-1}(t^N_k)$ corresponding to the $t^N_k$ it holds via~\eqref{eq:betaNk_sNk} that
\[
  s^N_k \leq t^N_k + \beta^N_k + \sum_{j=1}^k d^N_j
  \leq T + \lambda + \tilde{C}(\lambda)
  =: \sigma'
\]
and we see that $s^N_k$ remains bounded by $\sigma'$ for those $k$. Then, $\psi^N(\sigma') \geq \psi^N(s^N_k) = t^N_k$ whenever $t^N_k \leq T'$ and $N$ is sufficiently large. Consequently, $\liminf_{N \to \infty} \psi^N(\sigma') \geq T'$. Letting $T' \to T_*$ we obtain
\[
  \lim_{\sigma \to \infty} \liminf_{N \to \infty} \psi^N(\sigma) \geq T_*.
\]
Together with~\eqref{eq:Ts_geq}, this completes the proof of~\eqref{eq:psiN_lim}.

We can easily make $S \mapsto \bar{C}(S)$ increasing and then replace it by its upper semicontinuous envelope.
\end{proof}

\begin{proposition} \label{prop:incr_prop_s}
For all $k \in \{0,\ldots,N\}$ the following hold:
\begin{enumerate}[(i)]
  \item The discrete lower energy estimate
\[  
  \Ecal_{\psi^N}(s^N_k,y^N_k,z^N_k) \leq \Ecal_{\psi^N}(0,y_0,z_0) - \sum_{j=1}^k \Diss(\Sd^N_j) - \sum_{j=1}^k \int_{s^N_{j-1}}^{s^N_j} \dprb{\dot{f}_{\psi^N}(\sigma),y^N_{j-1}} \dd \sigma.
\]
\item The discrete stability
\[  \qquad
  \Ecal_{\psi^N}(s^N_k,y^N_k,z^N_k) \leq \Ecal_{\psi^N}(s^N_k,\hat{y},\hat{\Sd}_\ff z^N_k) + \Diss(\hat{\Sd})
\]
for all $\hat{y} \in \Ycal$ and $\hat{\Sd} \in \Slip(z^N_k)$ with $\norm{\hat{\Sd}}_{\Lrm^\infty} \leq \gamma$.
\end{enumerate}
\end{proposition}

\begin{proof}
This is a direct translation of Proposition~\ref{prop:incr_prop}, noting that we use a change of variables for the external power integral in~(i).
\end{proof}

\subsection{Passage to the limit}

We first establish that a limit process exists as $N \to \infty$. Then we will show that this limit process has the required properties. In this context we recall that we do not identify processes that are equal almost everywhere in time.

\begin{lemma} \label{lem:limit}
There exists a subsequence of the $N$'s (not explicitly labelled) and
\begin{align*}
  y &\in \Lrm^\infty([0,\infty);\Wrm^{1,p}_g(\Omega;\R^3)) \text{ with $y(0) = y_0$}, \\
  P &\in \Lip([0,\infty);\Wrm^{1,q}(\Omega;\R^{3 \times 3})) \text{ with $\det P(s) = 1$ a.e.\ in $\Omega$ for all $s \in [0,\infty)$}, \\
  \Sd & \in \Lip([0,\infty);\Disl(\cl{\Omega})) \text{ with $\Diss(\Sd;[0,s]) \leq s$ for $s \in [0,\infty)$}, \\
  \psi &\in \Crm([0,\infty)) \text{ increasing and Lipschitz with $\psi(0) = 0$, $\psi(\infty) = T_* \in (0,T]$,}
\end{align*}
such that
\begin{align}
  \bar{P}^N &\toweakstar P  \quad\text{locally in $\BV([0,\infty);\Wrm^{1,q}(\Omega;\R^{3 \times 3}))$},  \label{eq:limit_P} \\
  \bar{\Sd}^N &\toweakstar \Sd \quad\text{locally},  \label{eq:limit_Sigma} \\
  \psi^N &\to \psi \quad\text{locally uniformly.}  \label{eq:limit_psi}
\end{align}
Moreover, for all $s \in [0,\infty)$,
\begin{align}
  y(s) &\in \Argmin \, \setb{ \Ecal(s,\hat{y},z(s)) }{ \hat{y} \in Y_s }, \label{eq:limit_y} \\
  \liminf_{N \to \infty} \int_0^s \dprb{\dot{f}_{\psi^N}(\sigma),\bar{y}^N(\sigma)} \dd \sigma &\geq \int_0^s \dprb{\dot{f}_\psi(\sigma),y(\sigma)} \dd \sigma = \int_0^s \Pi_{\mathrm{red}}(\sigma,P(\sigma)) \dd \sigma,  \label{eq:limit_power}
\end{align}
where
\begin{align*}
  Y_s &:= \setb{ \hat{y} \in \Wrm^{1,p}_g(\Omega;\R^3) }{ \norm{\hat{y}}_{\Wrm^{1,p}} \leq \bar{C}(s) }, \\
  \Pi_{\mathrm{red}}(s,P) &:= \inf \, \setb{ \dprb{\dot{f}_\psi(s),\hat{y}} }{ \hat{y} \in \Argmin \, \setb{ \Wcal_e(\hat{y},P) - \dprb{f_\psi(s),\hat{y}} }{ \hat{y} \in Y_s } },
\end{align*}
with $\bar{C}(s)$ the constant from Proposition~\ref{prop:apriori_s}.
\end{lemma}

\begin{proof}
\proofstep{Step~1.}
First, by~\eqref{eq:psiN_lim} together with a diagonal procedure we may pick a subsequence of the $N$'s (not made explicit in our notation) such that
\begin{equation} \label{eq:psiN_lim2}
  \lim_{\sigma \to \infty} \liminf_{N \to \infty} \psi^N(\sigma)
  = \lim_{\sigma \to \infty} \limsup_{N \to \infty} \psi^N(\sigma)
  = T_*.
\end{equation}
Let now $S > 0$. We know from Proposition~\ref{prop:apriori_s} that
\begin{align*}
  \norm{\bar{y}^N}_{\Lrm^\infty([0,S];\Wrm^{1,p})} + \norm{\bar{P}^N}_{\Lrm^\infty([0,S];\Wrm^{1,q})} + \Var(\bar{\Sd}^N;[0,S]) &\leq \bar{C}(S), \\
  \norm{\bar{\Sd}^N}_{\Lrm^\infty([0,S];\Disl(\cl{\Omega}))} &\leq \gamma,
\end{align*}
and also, by construction,
\[
  \det \bar{P}^N(s) = 1 \quad\text{a.e.\ in $\Omega$ and for all $s \in [0,\infty)$.} 
\]
To estimate the variation of $\bar{P}^N$, take any partition $0 = \sigma_0 < \sigma_1 < \cdots < \sigma_K = S$ of the interval $[0,S]$ and apply Lemma~\ref{lem:PSigma_W1q} to the definition~\eqref{eq:barPN_ODE} to see
\[
  \sum_{\ell=1}^K \normb{\bar{P}^N(\sigma_\ell) - \bar{P}^N(\sigma_{\ell-1})}_{\Wrm^{1,q}}
  \leq \sum_{\ell=1}^K C \cdot \Var(\bar{\Sd}^N;[\sigma_{\ell-1},\sigma_\ell])
  = C \cdot \Var(\bar{\Sd}^N;[0,S]),
\]
where the constant $C$ depends on $\norm{P_0}_{\Wrm^{1,q}}$ and $\Var(\bar{\Sd}^N;[0,S])$. A slight generalization of the above argument shows that in fact
\[
  \Var_{\Wrm^{1,q}}(\bar{P}^N;[s,t]) \leq C \cdot \Var(\bar{\Sd}^N;[s,t])
\]
for all $[s,t] \subset [0,S]$. Since also $\Var(\bar{\Sd}^N;[s,t]) \leq C \abs{s-t}$ by~\eqref{eq:VarSigmaN}, we obtain that both $\bar{P}^N$ and $\bar{\Sd}^N$ are Lipschitz continuous on $[0,S]$ with uniform (in $N$) Lipschitz constant. Hence, taking a further subsequence on the $N$'s by Proposition~\ref{prop:BVX_Helly} for $\bar{P}^N$ and Proposition~\ref{prop:LipDS_compact} for $\bar{\Sd}^N$, we obtain that there exist $P \in \Lip([0,S];\Wrm^{1,q}(\Omega;\R^{3 \times 3}))$ and $\Sd \in \Lip([0,S];\Disl(\cl{\Omega}))$ satisfying~\eqref{eq:limit_P},~\eqref{eq:limit_Sigma} in the interval $[0,S]$. Concatenating this for all intervals $[0,S]$, $S > 0$, we obtain the existence of the limit processes $P \in \Lip([0,\infty);\Wrm^{1,q}(\Omega;\R^{3 \times 3}))$ and $\Sd \in \Lip([0,\infty);\Disl(\cl{\Omega}))$.

By Proposition~\ref{prop:BVX_Helly} (use the compact embedding $\Wrm^{1,q}(\Omega;\R^{3 \times 3}) \cembed \Crm(\Omega;\R^{3 \times 3})$), we get
\[
  \det P(s) = 1  \qquad\text{a.e.\ in $\Omega$ for all $s \in [0,\infty)$.}
\]
Also, $\Diss(\Sd;[0,s]) \leq s$ follows from the construction of the rescaled time $s$; see the argument before~\eqref{eq:VarSigmaN}.

The maps $\psi^N \colon [0,\infty) \to [0,T]$ are increasing and Lipschitz continuous with Lipschitz constant bounded by $1$. Hence, taking yet another subsequence by the Arzel\`{a}--Ascoli theorem, there is $\psi \colon [0,\infty) \to [0,T]$ increasing and Lipschitz continuous with Lipschitz constant bounded by $1$, such that $\psi^N \to \psi$ locally uniformly in $[0,\infty)$, i.e.,~\eqref{eq:limit_psi} holds. As $\psi^N(0) = 0$ for all $N$, also $\psi(0) = 0$.

To show $\psi(\infty) = \lim_{s\to\infty} \psi(s) = T_*$, let $\eps > 0$. From~\eqref{eq:psiN_lim2} we may find $\sigma > 0$ such that for $s \geq \sigma$ we have
\[
  \liminf_{N \to \infty} \psi^N(s) \geq T_* - \eps,  \qquad
  \limsup_{N \to \infty} \psi^N(s) \leq T_* + \eps.
\]
Then we get $\psi(s) \geq T_* - \eps$ and $\psi(s) \leq T_* + \eps$. Letting $\eps \to 0$, we conclude that $\psi(\infty) = T_*$.

\proofstep{Step~2.}
For any $s > 0$, the weak $\Wrm^{1,p}$-topology restricted to $Y_s$ is complete, separable, and metrizable. Then, fixing $S > 0$, define
\[
  M(s,P) := \Argmin \, \setb{ \Wcal_e(\hat{y},P) - \dprb{f_\psi(s),\hat{y}} }{ \hat{y} \in Y_s } \subset \Wrm^{1,p}_g(\Omega;\R^3)
\]
for $(s,P) \in [0,S] \times \Wrm^{1,q}(\Omega;\R^{3 \times 3})$ with $\det P = 1$ a.e.\ in $\Omega$. The set $M(s,P)$ is non-empty. This follows via the Direct Method using the coercivity in Lemma~\ref{lem:E_coerc} and the lower semicontinuity in Lemma~\ref{lem:conv}~(i). In this new notation,
\[
  \Pi_{\mathrm{red}}(s,P) = \inf \, \setb{ \dprb{\dot{f}_\psi(s),\hat{y}} }{ \hat{y} \in M(s,P) }.
\]

Next, we observe that $(s,P) \mapsto M(s,P)$ is continuous in the following sense: If $s_j \to s$ in $[0,S]$ and $P_j \toweak P$ in $\Wrm^{1,q}$, then for any sequence $y_j \in M(s_j,P_j)$ with $y_j \toweak y$ in $\Wrm^{1,p}$ it holds that $y \in M(s,P)$. To see this, it suffices to combine~(i),~(ii), and~(iii) of Lemma~\ref{lem:conv}, which together imply that limits of minimizers are minimizers themselves. One can either argue directly or realize that these two statements together imply the $\Gamma$-convergence~\cite{DalMaso93book} of $\Wcal_e(\frarg,P_j) - \dpr{f_\psi(s_j),\frarg}$, from which the claimed continuity property follows. Note that here we also use the monotonicity and upper semicontinuity of the constant $\bar{C}(s) > 0$ from Proposition~\ref{prop:apriori_s} with respect to $s$. Similarly, we also obtain that $M(s,P)$ is weakly closed, hence weakly compact.

Define the set-valued map $F \colon [0,S] \toto Y_S$ via
\[
  F(s) := M(s,P(s)),  \qquad s \in [0,S].
\]
The just established continuity property for $(s,P) \mapsto M(s,P)$ together with the Lipschitz continuity of $P$ implies that $\graph(F) := \setn{ (s,y) }{ y \in F(s) }$ is closed in $[0,S] \times Y_S$.  Hence, $\graph(F)$ is a measurable set with respect to the product $\sigma$-algebra on $[0,S] \times Y_S$ (i.e., the product $\sigma$-algebra of the Lebesgue-measurable subsets of $[0,S]$ and the Borel-$\sigma$-algebra induced by the metric of weak convergence on $Y_S$). By standard results, see~\cite[Theorem~8.1.4]{AubinFrankowska90}, this then implies that the set-valued map $F$ is measurable, meaning that the preimages $F^{-1}(B) := \setb{ s \in [0,S] }{ F(s) \cap B \neq \emptyset }$ are Lebesgue-measurable for all Borel sets $B \subset Y_S$.

Set
\[
  h(s,y) := \begin{cases}
    \dprb{\dot{f}_\psi(s),y} - \Pi_{\mathrm{red}}(s,P(s))  &\text{if $y \in F(s)$,}\\
    0       &\text{otherwise}
  \end{cases}
\]
for $(s,y) \in [0,S] \times Y_S$. By Lemma~\ref{lem:conv}~(iii) in conjunction with the Lipschitz continuity of $P$, the function $h$ is measurable. Moreover, for fixed $s \in [0,S]$, the map $h(s,\frarg)$ is continuous, again by~(iii) of Lemma~\ref{lem:conv}. Finally, by similar arguments as before, we have that for every $s \in [0,S]$ there is a $y_* \in Y_s$ such that $h(s,y_*) = 0$.

We can now apply the generalized version of the Filippov measurable selection theorem that was proved in~\cite[Theorem~B.1.2]{MielkeRoubicek15book} (also see the more classical version in~\cite[Theorems~8.2.9 and~8.2.10]{AubinFrankowska90}). This theorem allows us to obtain $y \colon (0,\infty) \to \Wrm^{1,p}_g(\Omega;\R^3)$ such that
\[
  y(s) \in F(s) = \Argmin \, \setb{ \Wcal_e(\hat{y},P(s)) - \dprb{f_\psi(s),\hat{y}} }{ \hat{y} \in Y_s }
\]
and $h(s,y(s)) = 0$, i.e.,
\[
  \dprb{\dot{f}_\psi(s),y(s)} = \Pi_{\mathrm{red}}(s,P(s)) = \inf \, \setb{ \dprb{\dot{f}_\psi(s),\hat{y}} }{ \hat{y} \in F(s) }
\]
for all $s \in (0,\infty)$. We also set $y(0) := y_0$ with $y_0$ from Assumption~\ref{as:initial}. This shows~\eqref{eq:limit_y} since the set $F(s)$ is also equal to the set of minimizers of $\Ecal(s,\frarg,P(s),\Sd(s))$ (the minimization in $\hat{y}$ is independent of $\Sd(s)$). Moreover, the equality on the right-hand side of~\eqref{eq:limit_power} holds by construction.

\proofstep{Step~3.}
Next, for every $\tau \in [0,\psi(S)]$ we can find a $\tau$-dependent subsequence $N_\tau(m)$ such that, as $m \to \infty$ (note that $\psi^N$ is strictly increasing),
\[
  \bar{y}^{N_\tau(m)}((\psi^{N_\tau(m)})^{-1}(\tau)) \toweak \tilde{y}(\tau)  \quad\text{in $\Wrm^{1,p}$} 
\]
for some $\tilde{y}(\tau) \in Y_S$ and
\[
  \lim_{m\to\infty} \dprb{\dot{f}(\tau),\bar{y}^{N_\tau(m)}((\psi^{N_\tau(m)})^{-1}(\tau))} = \liminf_{N \to \infty} \, \dprb{\dot{f}(\tau),\bar{y}^N((\psi^N)^{-1}(\tau))}.
\]
Then, by Lemma~\ref{lem:conv}~(iii),
\begin{equation} \label{eq:ytilde_liminf}
  \liminf_{N \to \infty} \, \dprb{\dot{f}(\tau),\bar{y}^N((\psi^N)^{-1}(\tau))} = \dprb{\dot{f}(\tau),\tilde{y}(\tau)}.
\end{equation}
Furthermore, by construction, if $\tau \in (t^{N_\tau(m)}_{k_\tau(m)-1},t^{N_\tau(m)}_{k_\tau(m)}]$ (where $k_\tau(m) \in \{1,\ldots,N_\tau(m)\}$), then, by the rescaled time-incremental problem~\eqref{eq:IP_s},
\[
  \bar{y}^{N_\tau(m)}((\psi^{N_\tau(m)})^{-1}(\tau))
  = \bar{y}^{N_\tau(m)}(s^{N_\tau(m)}_{k_\tau(m)})
  \in M \bigl( s^{N_\tau(m)}_{k_\tau(m)},P^{N_\tau(m)}_{k_\tau(m)} \bigr).
\]
Assume that $s^{N_\tau(m)}_{k_\tau(m)} \to \sigma$, $t^{N_\tau(m)}_{k_\tau(m)} \to \tau$ with $\psi(\sigma) = \tau$. Moreover,
\[
  P^{N_\tau(m)}_{k_\tau(m)} = \bar{P}^{N_\tau(m)}(s^{N_\tau(m)}_{k_\tau(m)}) \toweak P(\sigma) \quad\text{in $\Wrm^{1,q}$,}
\]
where we have used the uniform Lipschitz continuity of $\bar{P}^N$ and the convergence $\bar{P}^N \toweakstar P$ locally in $\BV$. By the continuity property of $(s,P) \mapsto M(s,P)$ shown above, we thus have
\[
  \tilde{y}(\psi(\sigma)) \in M(\sigma,P(\sigma)) = F(\sigma)
\]
where we have also used the upper semicontinuity of $s \mapsto \bar{C}(s)$ (see Lemma~\ref{prop:apriori_s}). Furthermore,
\begin{equation} \label{eq:Pi_cont}
  \dprb{\dot{f}_\psi(\sigma),\tilde{y}(\psi(\sigma))} \geq \Pi_{\mathrm{red}}(\sigma,P(\sigma)).
\end{equation}

For $s \in [0,\infty)$,
\begin{align*}
  \int_0^s \dprb{\dot{f}_{\psi^N}(\sigma),\bar{y}^N(\sigma)} \dd \sigma 
  &= \int_0^{\psi^N(s)} \dprb{\dot{f}(\tau),\bar{y}^N((\psi^N)^{-1}(\tau))} \dd \tau \\
  &= \int_0^{\psi(s)} \dprb{\dot{f}(\tau),\bar{y}^N((\psi^N)^{-1}(\tau))} \dd \tau + \SmallO(1), 
\end{align*}
where the error term $\SmallO(1)$ vanishes as $N \to \infty$ since $\psi^N(s) \to \psi(s)$ and the integrand is uniformly bounded by Assumption~\ref{as:F} and the definition of $Y_S$. We can now apply Fatou's lemma and, in turn,~\eqref{eq:ytilde_liminf},~\eqref{eq:Pi_cont},~\eqref{eq:limit_power} to estimate
\begin{align*}
  \liminf_{N \to \infty} \int_0^s \dprb{\dot{f}_{\psi^N}(\sigma),\bar{y}^N(\sigma)} \dd \sigma
  &\geq \int_0^{\psi(s)} \dprb{\dot{f}(\tau),\tilde{y}(\tau)} \dd \tau \\
  &= \int_0^s \dprb{\dot{f}_\psi(\sigma),\tilde{y}(\psi(\sigma))} \dd \sigma \\
  &\geq \int_0^s \Pi_{\mathrm{red}}(\sigma,P(\sigma)) \dd \sigma \\
  &= \int_0^s \dprb{\dot{f}(\sigma),y(\sigma)} \dd \sigma.
\end{align*}
This establishes the lower limit inequality in~\eqref{eq:limit_power}.
\end{proof}

We now prove the stability and energy balance for the limit solution.

\begin{proposition} \label{prop:stability}
For every $S \in [0,\infty)$ there exists $\gamma(S) > 0$ such that if $\gamma > \gamma(S)$ then for all $s \in [0,S]$ with $\dot{\psi}(s) > 0$ and all $\hat{y} \in \Ycal$, $\hat{\Sd} \in \Slip(\Sd(s))$ with $\norm{\hat{\Sd}}_{\Lrm^\infty} \leq \gamma$, the stability relation
\begin{equation} \label{eq:stab}
  \Ecal_\psi(s,y(s),z(s)) \leq \Ecal_\psi(s,\hat{y},\hat{\Sd}_\ff z(s)) + \Diss(\hat{\Sd})
\end{equation}
holds.
\end{proposition}

\begin{proof}
In Proposition~\ref{prop:incr_prop_s}~(ii) we established the time-incremental stability at step $k = 0,1,2,\ldots$, namely
\begin{equation} \label{eq:incr_stab_repeat2}
  \Ecal_{\psi^N}(s^N_k,y^N_k,z^N_k) \leq \Ecal_{\psi^N}(s^N_k,\hat{y},\hat{\Sd}_\ff z^N_k) + \Diss(\hat{\Sd})
\end{equation}
for all $\hat{y} \in \Ycal$ and $\hat{\Sd} \in \Slip(z^N_k)$ with $\norm{\hat{\Sd}}_{\Lrm^\infty} \leq \gamma$.

Fix a point $s \in [0,S]$ with $\dot{\psi}(s) > 0$ and define for $N \in \N$ (more precisely, for the subsequence of $N$'s constructed in Lemma~\ref{lem:limit}) the index $k(N)$ to be the largest $k \in \{0,\ldots,N\}$ such that $s^N_{k(N)} \leq s$. For the corresponding $t^N_{k(N)} := \psi^N(s^N_{k(N)})$ we have $t^N_{k(N)} \to t := \psi(s)$ as $N \to \infty$ since the $\{t^N_k\}_{N,k}$ lie dense in $[0,T]$ (this uses the uniform convergence $\psi^N \to \psi$), Moreover, as we assumed $\dot{\psi}(s) > 0$, the Taylor expansion $\psi(s^N_{k(N)}) = \psi(s) + \dot{\psi}(s)[s^N_{k(N)}-s] + \BigO(\abs{s^N_{k(N)}-s}^2)$ then yields that also $s^N_{k(N)} \to s$.

Using that $\bar{\Sd}^N(s) \toweakstar \Sd(s)$ in $\Disl(\cl{\Omega})$, $\bar{P}(s) \toweak P(s)$ in $\Wrm^{1,q}$, and also the uniform Lipschitz continuity of $\Sd$ (with respect to a metric for the weak* convergence, e.g., the flat norm) and $P$ (with respect to $\Wrm^{1,q}$), we obtain
\[
  \Td^N_{k(N)} \toweakstar \Sd(s)  \quad\text{in $\Disl(\cl{\Omega})$,}  \qquad
  P^N_{k(N)} \toweak P(s) \quad\text{in $\Wrm^{1,q}$.}
\]
By Proposition~\ref{prop:D_weak*}, there is $\tilde{\Sd}^N_s \in \Slip(z^N_{k(N)})$ with $(\tilde{\Sd}^N_s)_\ff \Td^N_{k(N)} = \Sd(s)$ and
\[
  \Diss(\tilde{\Sd}^N_s) \to 0  \qquad\text{as $N \to \infty$.}
\]
Moreover, let $\gamma > \gamma(S) := C \cdot \bar{C}(S)$ with $C > 0$ the constant from Proposition~\ref{prop:D_weak*} and $\bar{C}(S) = \tilde{C}(\lambda(S)) > 0$ as in~\eqref{eq:apriori} of Lemma~\ref{lem:apriori}, where $\lambda(S) := S + \Ecal(0,y_0,z_0) + 1$ so that $\beta^N_k \leq \lambda(S)$ by~\eqref{eq:betaNk_sNk} (cf.\ the proof of Proposition~\ref{prop:apriori_s}). Then,
\[
  \limsup_{N\to\infty} \norm{\tilde{\Sd}^N_s}_{\Lrm^\infty} \leq C \cdot \limsup_{N \to \infty} \Mbf(\Td^N_{k(N)}) \leq C \cdot \bar{C}(S) < \gamma.
\]

For $\hat{y},\hat{\Sd}$ as in the statement of the proposition we define the following \enquote{recovery sequence} for $\hat{\Sd}$:
\[
  \hat{\Sd}^N_s := \hat{\Sd} \circ \tilde{\Sd}^N_s \in \Slip(z^N_{k(N)}).
\]
We have
\[
  \norm{\hat{\Sd}^N_s}_{\Lrm^\infty} = \max \bigl\{\norm{\hat{\Sd}}_{\Lrm^\infty}, \norm{\tilde{\Sd}^N_s}_{\Lrm^\infty}  \bigr\}
  \leq \gamma
\]
for $N > N(s)$ sufficiently large (depending on $s$, but this will not matter in the following). We also observe from~\eqref{eq:Diss_additive} in Lemma~\ref{lem:Diss} that
\[
  \Diss(\hat{\Sd}^N_s) = \Diss(\tilde{\Sd}^N_s) + \Diss(\hat{\Sd})
\]
and from Lemma~\ref{lem:Sigma_concat} that
\[
  (\hat{\Sd}^N_s)_\ff \Td^N_{k(N)} = \hat{\Sd}_\ff \Sd(s).
\]
The slip trajectory $\hat{\Sd}^N_s$ is thus admissible in~\eqref{eq:incr_stab_repeat2} at $k = k(N)$ for $N$ sufficiently large, giving
\begin{align}
  \Ecal_{\psi^N}(s^N_{k(N)},y^N_{k(N)},z^N_{k(N)}) &\leq \Ecal_{\psi^N}(s^N_{k(N)},\hat{y},(\hat{\Sd}^N_s)_\ff P^N_{k(N)},\hat{\Sd}_\ff \Sd(s)) \notag \\
  &\qquad + \Diss(\tilde{\Sd}^N_s) + \Diss(\hat{\Sd}).  \label{eq:incr_stab_hatSigma}
\end{align}
Using Lemma~\ref{lem:PSigma_cont} in conjunction with $\hat{\Sd}^N_s \toweakstar \hat{\Sd}$ and $P^N_{k(N)} \toweak P(s)$ in $\Wrm^{1,q}$,
\[
  (\hat{\Sd}^N_s)_\ff P^N_{k(N)} \toweak \hat{\Sd}_\ff P(s)  \quad\text{in $\Wrm^{1,q}$.}
\]
Passing to a (further) subsequence in $N$ (for fixed $s$, not relabelled) to obtain $y^N_{k(N)} \toweak \tilde{y}$ in $\Wrm^{1,p}$, we may use the assertions~(i),~(ii) of Lemma~\ref{lem:conv} as well as the locally uniform convergence $\psi^N \to \psi$, to pass to the lower limit $N \to \infty$ in~\eqref{eq:incr_stab_hatSigma} at $k = k(N)$, obtaining
\[
  \Ecal_\psi(s,\tilde{y},z(s)) \leq \Ecal_\psi(s,\hat{y},\hat{\Sd}_\ff z(s)) + \Diss(\hat{\Sd}).
\]
Finally observing that $\Ecal_\psi(s,y(s),z(s)) \leq \Ecal_\psi(s,\tilde{y},z(s))$ by~\eqref{eq:limit_y}, the conclusion~\eqref{eq:stab} follows.
\end{proof}

\begin{remark} \label{rem:D_proj}
As remarked in the Introduction and explained further in Section~6.2 of~\cite{HudsonRindler22}, the projection in the definition of the total plastic drift in~\eqref{eq:D} has the effect of disregarding climb. The reason why we cannot simply enforce that $\hodge \gamma^b$ is orthogonal to $P^{-1} b$ for admissible slip trajectories is that this makes it impossible to deform some dislocations into each other via Proposition~\ref{prop:D_weak*}. Indeed, such a deformation may require a slip trajectory violating the orthogonality constraint, if only on a trajectory with vanishing variation. In this case the recovery construction in the preceding proposition would fail.
\end{remark}

\begin{proposition} \label{prop:energy_balance}
For every $S \in [0,\infty)$ and $\gamma > \gamma(S)$ (with $\gamma(S)$ defined in Proposition~\ref{prop:stability}), the energy balance
\begin{equation} \label{eq:energy_bal}
  \Ecal_\psi(s,y(s),z(s)) = \Ecal_\psi(0,y_0,z_0) - \Diss(\Sd;[0,s]) - \int_0^s \dprb{\dot{f}_\psi(\sigma),y(\sigma)} \dd \sigma
\end{equation}
holds for all $s \in [0,S]$.
\end{proposition}

\begin{proof}
From Proposition~\ref{prop:incr_prop_s}~(i) we have for all $k \in \{1,\ldots,N\}$ the discrete lower energy estimate
\begin{equation} \label{eq:incr_energy}
  \Ecal_{\psi^N}(s^N_k,y^N_k,z^N_k) \leq \Ecal_{\psi^N}(0,y_0,z_0) - \sum_{j=1}^k \Diss(\Sd^N_j) - \sum_{j=1}^k \int_{s^N_{j-1}}^{s^N_j} \dprb{\dot{f}_{\psi^N}(\sigma),y^N_{j-1}} \dd \sigma.
\end{equation}
Fix a point $s \in [0,\infty)$ and define for $N \in \N$ the index $k(N)$ to be the largest $k \in \{0,\ldots,N\}$ such that $s^N_k \leq s$. Then, by Lemma~\ref{lem:limit} and Lemma~\ref{lem:conv}~(i) as well as~\eqref{eq:limit_y}, we obtain (by arguments as in the preceding proof of Proposition~\ref{prop:stability})
\[
  \Ecal_\psi(s,y(s),z(s)) \leq \liminf_{N \to \infty} \Ecal_{\psi^N}(s^N_{k(N)},y^N_{k(N)},z^N_{k(N)}).
\]
Moreover, using Lemma~\ref{lem:conv}~(iv), as well as~\eqref{eq:VarSigmaN}, we have
\begin{align*}
  \Diss(\Sd;[0,s]) &\leq \liminf_{N \to \infty} \, \Diss(\bar{\Sd}^N;[0,s]) \\
  &= \liminf_{N \to \infty} \, \Diss(\bar{\Sd}^N;[0,s^N_{k(N)}]) \\
  &= \liminf_{N \to \infty} \sum_{j=1}^{k(N)} \Diss(\Sd^N_j).
\end{align*}
Combining this with~\eqref{eq:limit_power}, we may pass to the lower limit $N \to \infty$ in~\eqref{eq:incr_energy} and obtain
\begin{align}
  \Ecal_\psi(s,y(s),z(s))
  &\leq \Ecal_\psi(0,y_0,z_0) - \Diss(\Sd;[0,s]) - \int_0^s \dprb{\dot{f}_\psi(\sigma),y(\sigma)} \dd \sigma   \notag\\
  &= \Ecal_\psi(0,y_0,z_0) - \Diss(\Sd;[0,s]) - \int_0^s \Pi_{\mathrm{red}}(\sigma,P(\sigma)) \dd \sigma.  \label{eq:energy_upper}
\end{align}

On the other hand, take any partition $0 = \sigma_0 < \sigma_1 < \cdots < \sigma_K = s$ of the interval $[0,s]$ such that $\dot{\psi}(\sigma_\ell) > 0$ ($\ell = 1,\ldots K-1$). Fix $\ell \in \{0,\ldots,m-1\}$ and let $\hat{\Sd}^K_\ell \in \Slip(\Sd(\sigma_\ell))$ to be the restriction $\Sd \restrict (\sigma_\ell,\sigma_{\ell+1})$, rescaled to unit time length (via Lemma~\ref{lem:rescale}), so that
\[
  (\hat{\Sd}^K_\ell)_\ff z(\sigma_\ell) = z(\sigma_{\ell+1}).
\]
Apply the stability estimate from Proposition~\ref{prop:stability} at $s = \sigma_\ell$ with $\hat{y} := y(\sigma_{\ell+1})$ and $\hat{\Sd} := \hat{\Sd}^K_\ell$. In this way we get for $\ell = 1,\ldots K-1$ that
\begin{align*}
  \Ecal_\psi(\sigma_\ell,y(\sigma_\ell),z(\sigma_\ell))
  &\leq \Ecal_\psi(\sigma_\ell,y(\sigma_{\ell+1}),z(\sigma_{\ell+1})) + \Diss(\Sd;[\sigma_\ell,\sigma_{\ell+1}]) \\
  &= \Ecal_\psi(\sigma_{\ell+1},y(\sigma_{\ell+1}),z(\sigma_{\ell+1})) + \Diss(\Sd;[\sigma_\ell,\sigma_{\ell+1}]) \\
  &\qquad + \int_{\sigma_\ell}^{\sigma_{\ell+1}} \dprb{\dot{f}_\psi(\sigma),y(\sigma_{\ell+1})} \dd \sigma.
\end{align*}
This estimate also holds for $\ell = 0$ by the stability assumed in~\ref{as:initial}. Rearranging and summing from $\ell = 0$ to $K-1$, we obtain
\begin{equation} \label{eq:Elow_1}
  \Ecal_\psi(s,y,z(s)) + \Diss(\Sd;[0,s])
  \geq \Ecal_\psi(0,y_0,z_0) - \sum_{\ell = 0}^{K-1} \int_{\sigma_\ell}^{\sigma_{\ell+1}} \dprb{\dot{f}_\psi(\sigma),y(\sigma_{\ell+1})} \dd \sigma.  
\end{equation}
It further holds that
\[
  \biggl| \sum_{\ell = 0}^{K-1} \int_{\sigma_\ell}^{\sigma_{\ell+1}} \dprb{\dot{f}_\psi(\sigma),y(\sigma_{\ell+1})} \dd \sigma
  - \sum_{\ell = 0}^{K-1} \int_{\sigma_\ell}^{\sigma_{\ell+1}} \dprb{\dot{f}_\psi(\sigma_{\ell+1}),y(\sigma_{\ell+1})} \dd \sigma \biggr|
  \leq \eps
\]
as soon as the partition is sufficiently fine, where we use the uniform continuity of $\dot{f}_\psi$ with values in $\Wrm^{1,p}(\Omega;\R^3)^*$ from Assumption~\ref{as:F} as well as the uniform $\Wrm^{1,p}(\Omega;\R^3)$-boundedness of $y$. Here, we note that while the condition $\dot{\psi}(\sigma_\ell) > 0$ may force gaps in the partition, on these gaps the integrand vanishes and so the above statement is not affected.

By Hahn's lemma (see~\cite[Lemma~4.12]{DalMasoFrancfortToader05}) we may now choose a sequence of partitions $0 = \sigma_0 < \sigma_1 < \cdots < \sigma_K = s$ with $\dot{\psi}(\sigma_\ell) > 0$ ($\ell = 1,\ldots K-1$) such that we have the convergence of the associated upper Riemann sum, namely
\[
  \sum_{\ell = 0}^{K-1} \int_{\sigma_\ell}^{\sigma_{\ell+1}} \dprb{\dot{f}_\psi(\sigma_{\ell+1}),y(\sigma_{\ell+1})} \dd \sigma
  \to \int_0^s \dprb{\dot{f}_\psi(\sigma),y(\sigma)} \dd \sigma
\]
as $K \to \infty$. The same remark regarding the condition $\dot{\psi}(\sigma_\ell) > 0$ as before applies. Letting $K \to \infty$ in~\eqref{eq:Elow_1} (unless $\psi$ is totally flat on $[0,s]$, whereby the whole power term vanishes) and recalling the second relation in~\eqref{eq:limit_power}, we arrive at
\begin{align*}
  \Ecal_\psi(s,y,z(s)) &\geq \Ecal_\psi(0,y_0,z_0) - \Diss(\Sd;[0,s]) - \int_0^s \dprb{\dot{f}_\psi(\sigma),y(\sigma)} \dd \sigma \\
  &= \Ecal_\psi(0,y_0,z_0) - \Diss(\Sd;[0,s]) - \int_0^s \Pi_{\mathrm{red}}(\sigma,P(\sigma)) \dd \sigma.
\end{align*}
Together with~\eqref{eq:energy_upper}, we have thus established the claimed energy balance~\eqref{eq:energy_bal}.
\end{proof}

Next, we show the plastic flow equation.

\begin{proposition} \label{prop:plastflow}
For almost every $s \in [0,\infty)$, the plastic flow equation holds at $s$ in the $\Wrm^{1,q}$-sense (as in Lemma~\ref{lem:PSigma_ODE_BV}), i.e., 
\[
  \left\{ \begin{aligned}
    \frac{\di}{\di s} P(s) &= \biggl( x \mapsto \frac12 \sum_{b \in \Bcal} b \otimes \proj_{\langle P(s,x)^{-1}b \rangle^\perp} \bigl[ g^b(s,x) \bigr] \biggr) \qquad\text{for a.e.\ $s \in [0,\infty)$,}\\
    P(0) &= P_0,
  \end{aligned} \right.
\]
where $g^b$ is the density of the measure $\hodge \pbf(S^b_\eta)$ (with $\Sd = (S^b)_b$ and $\eta$ the dislocation line profile).
\end{proposition}

\begin{proof}
The ODE holds for $\bar{P}^N$, see~\eqref{eq:barPN_ODE}. Using the convergence assertions in Lemma~\ref{lem:limit}, we can then pass to the limit using (the same technique as in the proof of) Lemma~\ref{lem:PSigma_cont}.
\end{proof}

Finally, we record the following regularity estimate:

\begin{lemma} \label{lem:estimates}
For every $S \in [0,\infty)$ there is a constant $\bar{C}(S) > 0$ such that if $s \in [0,S]$, then the estimates
\[
  \norm{y}_{\Lrm^\infty([0,S];\Wrm^{1,p})} + \norm{P}_{\Lrm^\infty([0,S];\Wrm^{1,q})} + \norm{\Sd}_{\Lrm^\infty([0,S];\Disl(\cl{\Omega}))} + \Var(\Sd;[0,S]) \leq \bar{C}(S)
\]
hold.
\end{lemma}

\begin{proof}
With the notation of the proof of Proposition~\ref{prop:stability}, $\Td^N_{k(N)} \toweakstar \Sd(s)$, and hence we have that $\norm{\Sd}_{\Lrm^\infty([0,S];\Disl(\cl{\Omega}))} \leq \bar{C}(S)$ by the estimates of~\eqref{eq:apriori} in Lemma~\ref{lem:apriori} (see the analogous argument in the proof of Proposition~\ref{prop:stability}) and the lower semicontinuity of the $\Lrm^\infty$-norm (Proposition~\ref{prop:current_Helly}). The other estimates follow directly from Proposition~\ref{prop:apriori_s} in conjunction with the assertions of Lemma~\ref{lem:limit}.
\end{proof}

\subsection{Proof of Theorem~\ref{thm:main}}

Finally, we dispense with the restriction that $\norm{\hat{\Sd}}_{\Lrm^\infty} \leq \gamma$ for the test trajectory $\hat{\Sd}$ in the stability condition~(S). From now on we make the dependence on $\gamma$ explicit and write $y_\gamma, P_\gamma, \Sd_\gamma, \psi_\gamma$ for $y, P,\Sd,\psi$.

Fix $S \in [0,\infty)$. The bounds from Lemma~\ref{lem:estimates}  (note in particular the $\gamma$-independent estimate on $\norm{\Sd}_{\Lrm^\infty([0,S];\Disl(\cl{\Omega}))}$) allow us to pass to a subsequence of $\gamma$'s (not explicitly labelled) tending to $+\infty$ such that the following hold just like in Lemma~\ref{lem:limit}: There exist
\begin{align*}
  y &\in \Lrm^\infty([0,\infty);\Wrm^{1,p}_g(\Omega;\R^3)) \text{with $y(0) = y_0$}, \\
  P &\in \Lip([0,\infty);\Wrm^{1,q}(\Omega;\R^{3 \times 3})) \text{ with $\det P(s) = 1$ a.e.\ in $\Omega$ for all $s \in [0,\infty)$}, \\
  \Sd & \in \Lip([0,\infty);\Disl(\cl{\Omega})) \text{ with $\Diss(\Sd;[0,s]) \leq s$ for $s \in [0,\infty)$}, \\
  \psi &\in \Crm([0,\infty)) \text{ increasing and Lipschitz with $\psi(0) = 0$, $\psi(\infty) = T_* \in (0,T]$}
\end{align*}
with
\begin{align*}
  P_\gamma &\toweakstar P  \quad\text{locally in $\BV([0,\infty);\Wrm^{1,q}(\Omega;\R^{3 \times 3}))$},  \\
  \Sd_\gamma &\toweakstar \Sd \quad\text{locally},  \\
  \psi_\gamma &\to \psi \quad\text{locally uniformly.}  
\end{align*}
In particular, we have $T_* > 0$ since the arguments before give a $\gamma$-independent lower bound on $T_*$ (see Lemma~\ref{lem:apriori}).

The stability~(S), the energy balance~(E), and the plastic flow equation~(P) follow from the construction and Propositions~\ref{prop:stability},~\ref{prop:energy_balance},~\ref{prop:plastflow} using the same techniques as in the previous section. We omit the repetitive details. Let us however observe that \emph{every} $\hat{\Sd} \in \Slip(\Sd(s))$ (which includes the assumption $\norm{\hat{\Sd}}_{\Lrm^\infty} < \infty$) becomes admissible for $\gamma$ sufficiently large. In this way, all parts of Definition~\ref{def:sol} follow. The initial conditions are satisfied by construction. The proof of Theorem~\ref{thm:main} is thus complete. \qed

%
%\bibliography{Plast.bib}
%\bibliographystyle{amsplain}

\providecommand{\bysame}{\leavevmode\hbox to3em{\hrulefill}\thinspace}
\providecommand{\MR}{\relax\ifhmode\unskip\space\fi MR }
\providecommand{\MRhref}[2]{%
  \href{http://www.ams.org/mathscinet-getitem?mr=#1}{#2}
}
\providecommand{\href}[2]{#2}

\end{document}